\documentclass[11pt,letterpaper]{article}
\usepackage[margin=1in]{geometry}
\usepackage[T1]{fontenc}
\usepackage{lmodern,microtype}
\usepackage{amsmath,amssymb,amsthm,mathtools}
\usepackage{enumitem,xspace}
\usepackage{xcolor}
\usepackage[authoryear,round]{natbib}
\usepackage{xurl}
\usepackage[colorlinks=true,linkcolor=blue!60!black,citecolor=blue!60!black,urlcolor=blue!60!black]{hyperref}
\usepackage[capitalise,noabbrev]{cleveref}

\theoremstyle{plain}
\newtheorem{theorem}{Theorem}
\newtheorem{proposition}{Proposition}
\newtheorem{lemma}{Lemma}
\newtheorem{corollary}{Corollary}
\theoremstyle{definition}
\newtheorem{definition}{Definition}
\newtheorem{example}{Example}
\newtheorem{assumption}{Assumption}
\theoremstyle{remark}
\newtheorem{remark}{Remark}
\crefname{theorem}{Theorem}{Theorems}
\crefname{proposition}{Proposition}{Propositions}
\crefname{lemma}{Lemma}{Lemmas}
\crefname{corollary}{Corollary}{Corollaries}
\crefname{definition}{Definition}{Definitions}
\crefname{assumption}{Assumption}{Assumptions}
\crefname{example}{Example}{Examples}
\crefname{remark}{Remark}{Remarks}
\crefname{section}{Section}{Sections}
\crefname{equation}{}{}

\newcommand{\T}{\mathcal T}
\newcommand{\Hk}{\mathcal H_\kappa}
\newcommand{\Lfam}{\mathcal L}
\newcommand{\Rfam}{\mathcal R}
\newcommand{\Afam}{\mathcal A}
\newcommand{\Nodd}{N}
\newcommand{\R}{\mathbb R}
\newcommand{\C}{\mathbb C}
\newcommand{\E}{\mathbb E}
\newcommand{\one}{\mathbf{1}}
\newcommand{\ind}{\mathbf{1}}
\newcommand{\ip}[2]{\langle #1,#2\rangle}
\newcommand{\norm}[1]{\lVert #1\rVert}
\newcommand{\abs}[1]{\lvert #1\rvert}
\newcommand{\wto}{\rightsquigarrow}
\newcommand{\indep}{\perp\!\!\!\perp}
\DeclareMathOperator{\Var}{Var}
\DeclareMathOperator{\Cov}{Cov}
\DeclareMathOperator{\TV}{TV}
\DeclareMathOperator{\aff}{aff}
\DeclareMathOperator{\conv}{conv}
\DeclareMathOperator{\spn}{span}
\DeclareMathOperator{\dist}{dist}
\DeclareMathOperator{\Law}{Law}

\DeclareMathOperator{\range}{range}
\DeclareMathOperator{\codim}{codim}
\DeclareMathOperator{\Lip}{Lip}
\DeclareMathOperator{\med}{med}
\DeclareMathOperator{\rank}{rank}
\DeclareMathOperator{\wid}{wid}
\DeclareMathOperator{\Log}{Log}
\newcommand{\Po}{\mathcal P}
\newcommand{\psik}{\tilde\psi_\kappa}
\newcommand{\Vk}{V_\kappa}
\newcommand{\Vstar}{V^\star}
\newcommand{\Pn}[1]{P^{n}_{n,#1}}
\newcommand{\clspan}{\overline{\spn}}
\newcommand{\claff}{\overline{\aff}}
\newcommand{\LamP}{\Lambda_P}
\newcommand{\FP}{F_P}
\newcommand{\assT}{\Cref{ass:paths}\xspace}
\newcommand{\assS}{\Cref{ass:additivity}\xspace}

\title{On the convolution theorem for directionally pathwise differentiable functionals}
\author{Shuoxun Xu\\[4pt]
  \normalsize Division of Biostatistics\\
  \normalsize School of Public Health, University of California, Berkeley\\[2pt]
  \normalsize \texttt{shuoxunxu\_ucb@berkeley.edu}}
\date{}
\hypersetup{pdftitle={On the convolution theorem for directionally pathwise differentiable functionals},
  pdfauthor={Shuoxun Xu},pdfsubject={Semiparametric efficiency for directionally pathwise differentiable functionals}}

\begin{document}
\raggedbottom
\maketitle

\begin{abstract}
The classical semiparametric convolution theorem characterizes the limiting distributions of regular estimators of pathwise differentiable parameters. We extend this analysis to directionally pathwise differentiable functionals. Regularity along a subspace or convex cone forces the directional derivative to agree there with a bounded linear functional, and every such estimator has a Gaussian convolution factor determined by that functional. We characterize the relevant subspaces, compare their efficiency bounds, and identify the influence functions of efficient estimators. These comparisons show when a stronger regularity requirement increases the variance bound. Under an additivity condition, the Gaussian limit experiment separates into a linear estimation problem and a nonlinear one. Their minimax risks add under squared error; for more general losses, the Gaussian component enters through convolution of the loss. Applications to treatment values, instrumental-variable bounds, and calibration give explicit efficient influence functions, compare existing procedures, and establish feasible attainment in finite-stratum models.
\end{abstract}

\section{Introduction}\label{sec:intro}

Semiparametric efficiency theory asks how accurately a parameter can be estimated and which estimators attain the efficiency bound. For a pathwise differentiable parameter, the classical convolution theorem characterizes the limiting distribution of every regular estimator as the convolution of a centered Gaussian law and a residual law \citep{Hajek1970,LeCam1972,BKRW}. The Gaussian variance is the squared norm of the parameter's canonical gradient, which is the influence function of an efficient estimator. The theorem therefore gives both a distributional lower bound and a characterization of efficient estimation.

We study these same questions for directionally pathwise differentiable functionals. Many economic parameters take this form because they combine smooth quantities through optimization, ranking, or absolute values. Examples include the value of an optimal treatment rule, endpoints of identified sets, and measures of calibration or effect heterogeneity. Their directional derivatives can be nonlinear even in models satisfying local asymptotic normality. We establish a convolution theorem for this setting, determine the local directions on which regularity can hold, and derive the corresponding efficiency bounds. We also study local minimax risk when the nonlinear part of the derivative contributes to the estimation error, and derive implications for confidence intervals from the resulting local distributions.

Consider programmes 1 and 2 with local mean values $c+d$ and $c-d$, respectively, where $c,d\in\R$. The common level $c$ is their average, and $d$ is half their difference. Suppose each value is observed with an independent unit-variance Gaussian error. Let $U$ and $V$ denote the average and half-difference of the observations, and let $G$ and $H$ denote their respective estimation errors. Then
\begin{equation}\label{eq:intro-experiment}
U=c+G,\qquad V=d+H,\qquad G\indep H,\qquad G,H\sim N(0,1/2).
\end{equation}
The value of the best programme is $c+\abs d$. Raising the common level $c$ by any $a\in\R$ raises this value by $a$ without changing $d$. Consider a measurable estimation rule $T:\R^2\to\R$ that respects every such shift: $T(u+a,v)=T(u,v)+a$ for all $u,v,a\in\R$. Defining $r(v):=T(0,v)$ gives $T(U,V)=U+r(V)$, and hence
\begin{equation}\label{eq:intro-error}
T(U,V)-(c+\abs d)=G+\{r(d+H)-\abs d\}.
\end{equation}
The two terms are independent. The first is the error in estimating the common level; the second is the error in estimating $\abs d$. For squared error, their risks add, giving $1/2+\E\{r(d+H)-\abs d\}^2$ whenever this expectation is finite.

The pooled estimate takes $r=0$. At the tie $d=0$, it attains the Gaussian variance $1/2$; when the values separate, it has drift $-\abs d$. The plug-in maximum takes $r(v)=\abs v$, retaining the additional term $\abs H$ at the tie while using the difference observation away from it. Thus a residual can describe both excess error at a tie and an estimator's response to the nonlinear part of the target. Understanding its role is essential to interpreting an efficiency bound. The direct representation in \eqref{eq:intro-error} concerns common-shift-equivariant rules in this Gaussian experiment; the general convolution theorem below uses local regularity and does not presume such a representation for an estimator at every local parameter.

Our convolution theorem uses the usual definition of regularity, applied along a subspace or convex cone of scores: the limiting distribution of the estimation error, centered at the parameter under each local alternative, is the same along every corresponding path. We show that this condition forces the directional derivative to agree on that set with a linear functional that extends boundedly to its span. The limiting law of every estimator regular on the set is then the convolution of a centered Gaussian law, whose variance is the squared norm of the representer, and a residual law. No asymptotically linear representation of the estimator is assumed. When the derivative is linear and regularity is imposed on the whole tangent space, the conclusion is the usual semiparametric convolution theorem.

For a nonlinear derivative, the efficiency comparison also depends on the subspace along which regularity is required. We characterize the maximal subspaces on which the derivative is linear and identify their intersection, which we call the core. The derivative may be linear on two subspaces without being linear on their sum, so distinct maximal subspaces can impose incompatible regularity requirements. For a maximum, the active branches are those tied for the largest value at the baseline distribution. There is a unique maximal subspace: its directions move all active branches equally to first order. For the median of three tied programme values, maximal subspaces can be incomparable and have different efficiency bounds. The core provides a variance bound shared by all maximal choices, while the larger subspaces determine which additional regularity requirements can be imposed and at what variance cost.

The two-programme example also shows why the regularity requirement affects the bound. Requiring a common error law whenever programme 1 remains best means allowing $d\ge0$. The estimator $U+V$, the observation from programme 1, satisfies this requirement and has centered error $G+H$, with variance $1$. The convolution theorem shows that this variance is unavoidable under that requirement. Thus the variance bound increases from $1/2$ to $1$ when regularity is required throughout this one-sided set. We characterize when enlarging the regularity set leaves the efficiency bound unchanged and when it increases the bound, including the conditions under which a one-sided requirement entails the full variance of a single active branch.

Efficiency under a given regularity requirement does not by itself describe risk over other local alternatives. To study that risk, we return to the decomposition in \eqref{eq:intro-error}. Under an additivity condition on the core, the directional derivative is the sum of a linear functional of the core parameter and a nonlinear function of the orthogonal parameter. The corresponding Gaussian observations are independent. With the core parameter unrestricted and the orthogonal parameter in a specified set, the squared-error minimax risk equals the core variance bound plus the minimax risk for the nonlinear component. For nonnegative, lower semicontinuous losses that are symmetric and nondecreasing in the absolute error (bowl-shaped losses), the latter problem is instead evaluated under the loss convolved with the core Gaussian law. These equalities in the Gaussian limit experiment yield local asymptotic lower bounds for arbitrary estimator sequences in the original model. We also obtain an explicit small-neighborhood expansion of the additional risk.

The applications connect this structure to influence functions and feasible estimators. For the value of an optimal individualized treatment rule, conditional means of tied treatments can be pooled using weights determined by their conditional variances and treatment probabilities. The resulting influence function attains the core bound, while a one-sided requirement that a designated rule remain optimal can entail additional variance. Conditional instrumental-variable bounds illustrate why correlations between active branches enter the calculation. Calibration error gives a case where an existing procedure already attains the bound, and sequential treatment values show why only ties at histories reached by an optimal policy affect the bound. A common conditional-max calculation organizes these examples. The explicit attaining constructions use finite strata or finite histories; comparisons with general-covariate procedures retain the learning and smoothing conditions needed for their asymptotic expansions.

The analysis builds on the connection between differentiability and regular estimation studied by \citet{vdV1991,HiranoPorter2012,Pfanzagl2000}. Their results explain why a nonlinear directional derivative cannot support a common centered error law over the whole tangent space. Our necessity argument identifies the compatible subspaces and cones without imposing moments on the limiting law. The convolution step uses the established restricted-parameter mechanism of \citet{JanssenOstrovski2013}. \citet{Kaji2021} likewise develops an efficiency analysis from the structure of a local experiment, for a different form of irregularity arising under weak identification. Here the parameter's directional derivative determines the subspaces on which regularity can hold and the corresponding Gaussian convolution factors.

Local minimax estimation offers a complementary assessment of nonlinear targets. \citet{Song2014} obtains bounds for arbitrary estimator sequences under convex loss for structured finite-dimensional composite targets. \citet{Fang2016} studies optimality within specified classes of plug-in procedures, while \citet{Ponomarev2022} studies local minimax estimation of directionally differentiable functionals. Our risk results isolate the Gaussian contribution and the residual problem whose optimum depends on the loss and local parameter set. The related inference problem is studied by \citet{FangSantos2019}, whose directional delta method and alternative resampling approach allow non-Gaussian limits. Their characterization of standard-bootstrap validity by full differentiability additionally imposes Gaussian-limit and support conditions. Our confidence-interval calculations use the local error laws of efficient estimators; constructing intervals with uniform coverage and optimal length is a further problem, related to the bias-aware analysis of \citet{ArmstrongKolesar2021}.

The treatment applications connect our results with two earlier efficiency analyses. \citet{XuGuo2025} develop optimal estimation of treatment values within a robust asymptotically linear class, including heteroscedastic and multistage outcomes. \citet{XuGuoRALU} formulate a variational theory for marginal-integral functionals, minimizing variance subject to robust unbiasedness and constructing estimators that attain the bound. Their conditional projection also yields a classical convolution bound under specified perturbations that preserve ties to first order. The variational theory and the present analysis overlap: in treatment models satisfying their respective assumptions, the optimal influence functions and bounds agree. Here the restriction on estimators is local regularity, and the directional derivative determines where that restriction is possible. This leads to comparisons across subspaces and cones and to an analysis of nonlinear residual risk. \Cref{sec:prior-work} explains the common results and the distinct questions addressed by the three papers; the applications also compare the smoothed procedures of \citet{Whitehouse2025} under the conditions for their asymptotic expansions.

\Cref{sec:setting} specifies the statistical model and states the convolution theorem. \Cref{sec:linearity} characterizes the regularity subspaces and their bounds, and \cref{sec:cost} identifies efficient estimators and compares regularity requirements. \Cref{sec:structure} studies the residual decision problem and the implications for confidence intervals. \Cref{sec:examples} develops the statistical applications, and \cref{sec:companion} discusses their relation to the variational theory and further questions. Proofs and supporting calculations are collected in the appendices.

\section{Statistical model, regularity, and the convolution theorem}\label{sec:setting}

We consider estimation of a scalar parameter in an i.i.d.\ statistical model. We use the standard notions of differentiability in quadratic mean and tangent space, with an explicit path-realization assumption stated below. We allow the parameter's pathwise derivative to be nonlinear and formulate regularity along a subspace or convex cone of scores. The convolution theorem then identifies the Gaussian factor in the limiting law of every estimator satisfying that regularity condition.

\subsection{Statistical model and differentiable paths}
Let $\Po$ be a set of probability measures on a measurable space $(\mathcal X,\mathcal A)$, let $P\in\Po$ be fixed, and let $X_1,\dots,X_n$ be i.i.d.\ $P$. Write $\mathbb P_n f=n^{-1}\sum_{i=1}^n f(X_i)$ for the empirical average, applied componentwise to vector-valued $f$. Write $L_2^0(P)$ for the mean-zero subspace of $L_2(P)$ with inner product $\ip{\cdot}{\cdot}$ and norm $\norm{\cdot}$. We write $\one$ for a vector of ones of the indicated dimension and $\ind\{B\}$ for the indicator of an event $B$. A map $t\mapsto P_t\in\Po$, $t\in[0,\varepsilon)$, $P_0=P$, is a \emph{differentiable-in-quadratic-mean (DQM) path with score $h\in L_2^0(P)$} if $\int[(dP_t^{1/2}-dP^{1/2})/t-\tfrac12h\,dP^{1/2}]^2\to0$ as $t\downarrow0$. We write $P_{t,h}$ for a DQM path with score $h$, $P_{n,h}:=P_{n^{-1/2},h}$, and $\Pn{h}$ for the $n$-fold product. All definitions below quantify over every DQM path in $\Po$ with the stated score \citep[\S 25.3]{vdV1998}. The path is required to be DQM at $t=0$; no smoothness at positive $t$ is imposed. The \emph{tangent space} $\T\subseteq L_2^0(P)$ is a linear subspace of scores.

\begin{assumption}\label{ass:paths}
$\T$ is closed in $L_2(P)$ and every $h\in\T$ is the score of some DQM path in $\Po$.
\end{assumption}
\assT requires every tangent direction to be realized by a differentiable path in the statistical model. It holds in the nonparametric model with $\T=L_2^0(P)$ \citep[Example 25.16]{vdV1998}. The quantification over all DQM paths also requires the parameter's derivative to depend only on the score; the applications verify this requirement in their observed-data models.

For a DQM path with score $h$, local asymptotic normality holds \citep[Theorem 7.2 and Lemma 25.14]{vdV1998}:
\begin{equation}\label{eq:LAN}
\log\prod_{i=1}^n\frac{dP_{n,h}}{dP}(X_i)=\frac1{\sqrt n}\sum_{i=1}^nh(X_i)-\tfrac12\norm h^2+o_{P^n}(1),
\end{equation}
hence $\Pn h$ and $P^n$ are mutually contiguous. Likelihood ratios relative to $P^n$ denote the densities of the absolutely continuous parts; the omitted singular masses vanish under local alternatives by contiguity. Throughout, $\Delta$ denotes the isonormal process on $\overline\T$, i.e.\ the centred Gaussian process with $\Cov(\Delta(g),\Delta(g'))=\ip g{g'}$.

\subsection{Directional pathwise differentiability}

Pathwise differentiability expresses the first-order change of a parameter as a continuous linear functional of the score. We retain the pathwise expansion but allow its derivative to be a positively homogeneous, nonlinear function of the score. This is the extension of differentiability used throughout the paper.

\begin{definition}\label{def:dpd}
$\Psi:\Po\to\R$ is \emph{directionally pathwise differentiable at $P$ relative to $\T$} if there is a positively homogeneous map $\gamma=\Psi'_P:\T\to\R$ ($\gamma(ch)=c\gamma(h)$ for $c\ge0$) such that for every $h\in\T$ and every DQM path with score $h$,
\[
t^{-1}\{\Psi(P_{t,h})-\Psi(P)\}\to\gamma(h)\qquad(t\downarrow0).
\]
\end{definition}

Ordinary pathwise differentiability is the case of linear $\gamma$. Composite targets $\Psi=\phi\circ\theta$ with $\theta$ pathwise differentiable into a Banach space and $\phi$ Hadamard directionally differentiable \citep{Shapiro1990,FangSantos2019} satisfy \cref{def:dpd} with $\gamma=\phi'_{\theta(P)}\circ\dot\theta_P$; the definition does not require an intermediate regular parameter. Positive homogeneity is automatic when the path family is closed under $t\mapsto ct$. Under \assT, the all-path formulation also implies that $\gamma$ is continuous (\cref{lem:autocont}). Thus continuity is supplied by the definition and path realization, rather than imposed as a separate assumption.

In the two-programme example, the derivative is $\gamma(c,d)=c+\abs d$, which is nonlinear because of the absolute value. The existence of directional limits is a condition on the parameter in the specified model; regularity of an estimator alone does not supply it, as discussed in \cref{app:rotating} and \citet{Pfanzagl2000}.

\subsection{Regularity along subspaces and cones}

Regularity requires the centered estimation error to have the same limiting distribution under local alternatives. Here we impose this condition on paths whose scores belong to a specified subspace or convex cone. A subspace includes both signs of each direction; a cone can describe one-sided alternatives, such as changes that keep a designated programme optimal. As in the differentiability definition, the condition applies to every path with the stated score.

\begin{definition}\label{def:regular}
Let $C\subseteq\T$ be a convex cone (in particular, a linear subspace). An estimator sequence $T_n=T_n(X_1,\dots,X_n)$ is \emph{regular along $C$} if there is a Borel law $L$ on $\R$ such that for every $h\in C$ and every DQM path with score $h$,
\[
\sqrt n\{T_n-\Psi(P_{n,h})\}\wto L\qquad\text{under }\Pn h .
\]
``Regular'' means regular along $\T$. \end{definition}

\subsection{The convolution theorem}\label{sec:convolution}

Gaussian tilting relates an estimator's limiting distributions at different local parameters. When the centered law is unchanged, this relation forces the directional derivative to agree with a linear functional on the regularity cone. The following theorem establishes its bounded extension and identifies the associated Gaussian component. The moment-free argument uses characteristic functions, so it does not assume that the estimator's limiting law has a variance.

\begin{theorem}[convolution]\label{thm:main}
Suppose \assT holds and $\Psi$ is directionally pathwise differentiable at $P$. Let $T_n$ be an estimator sequence with $\sqrt n\{T_n-\Psi(P)\}\wto L$ under $P^n$, and suppose that $T_n$ is regular along a convex cone $C\subseteq\T$.
\begin{enumerate}[label=(\arabic*)]
\item The derivative $\gamma$ agrees on $C$ with a bounded linear functional $\ell_C$ on $\spn C$. In particular, $\gamma$ is Lipschitz on $C$.
\item Let $\tilde\psi_C\in\clspan C$ be the Riesz representer of the continuous extension of $\ell_C$ to $\clspan C$, and put $V_C=\norm{\tilde\psi_C}^2$. There is a probability law $M_C$ such that
\[
L=N(0,V_C)*M_C.
\]
If $\spn C$ is finite-dimensional, let $(S,\Delta|_{\spn C})$ be any subsequential joint limit under $P^n$ of $\sqrt n\{T_n-\Psi(P)\}$ and the empirical score process on $\spn C$. Then $M_C$ is the law of $S-\Delta(\tilde\psi_C)$, which is independent of $\Delta|_{\spn C}$.
\end{enumerate}
\end{theorem}

\Cref{thm:main} shows that the Gaussian variance is the squared norm of the representer of the derivative on $C$. Every estimator satisfying the same regularity requirement contains that component in its limiting error; procedures differ through the residual law. With $C=\T$ and a linear derivative, $\tilde\psi_C$ is the ordinary canonical gradient. At a nonlinear point, the theorem additionally identifies the compatibility condition that the derivative must satisfy on $C$.

For a subspace $V$ on which $\gamma$ is linear, write $\tilde\psi_V$ for its representer in $\overline V$ and $V_V=\norm{\tilde\psi_V}^2$. An estimator is $V$-efficient if it is regular along $V$ with limit $N(0,V_V)$. This definition singles out the centered Gaussian law: equality of variances alone would still allow a constant residual shift. In the experiment \eqref{eq:intro-experiment}, the subspace $d=0$ has bound $1/2$. Pooling attains this law, whereas the plug-in retains the independent selection term $\abs H$.

The proof first establishes the Gaussian tilt identity on finite-dimensional subcones and continues it to their spans. The classical restricted-parameter convolution mechanism then identifies an independent residual \citep{JanssenOstrovski2013}. If the linear extension were unbounded, its Gaussian variance along finite-dimensional spans would diverge, contradicting continuity of the limiting characteristic function at zero. This explains why bounded extension is necessary even though $\gamma$ itself is continuous. Complete proofs, including the passage to closures, appear in \cref{app:theory}.

\section{Regularity subspaces and efficiency bounds}\label{sec:linearity}

The convolution theorem requires the directional derivative to be linear on the subspace along which an estimator is regular. We now characterize these subspaces and compare their variance bounds. Several maximal linearity subspaces can coexist, each allowing a different regularity requirement. Their intersection, called the core, gives a bound shared by all estimators regular along a maximal subspace. We also identify the additivity condition under which the derivative separates into linear and nonlinear parts, as needed for the local risk analysis.

Throughout this section $\gamma:\T\to\R$ is continuous (\cref{lem:autocont}) and positively homogeneous, and \assT is in force where closures are taken. Statements about estimator sequences also retain the model and differentiability assumptions of \cref{sec:setting}.

\subsection{Linearity of the directional derivative}

Regularity along a linear subspace requires the parameter's first-order change to be odd in each direction and additive when directions are combined. The definitions below describe these restrictions on the directional derivative and identify the subspaces on which the convolution theorem can apply.

\begin{definition}\label{def:odd}
$\Nodd(\gamma):=\{h\in\T:\gamma(h)+\gamma(-h)=0\}$ is the \emph{odd cone}. A linear subspace $V\subseteq\T$ is a \emph{linearity subspace} if $\gamma|_V$ is linear; a convex cone $C$ is a \emph{linearity cone} if $\gamma$ is additive on $C$ (equivalently, $\gamma$ agrees on $C$ with a linear functional on $C-C$). $\Lfam(\gamma)$ is the family of linearity subspaces and $\Lfam_{\max}(\gamma)$ its set of maximal elements under inclusion.
\end{definition}

The odd cone need not be convex or a linear subspace; oddness alone does not imply additivity. The \emph{regularity family} of $T_n$ is $\Rfam(T_n):=\{V\subseteq\T\text{ linear}:T_n\text{ is regular along }V\}$ and $\Rfam^{\rm cone}(T_n)$ is the family of convex cones along which $T_n$ is regular; both are closed under passing to smaller subspaces or cones, and $\{0\}\in\Rfam(T_n)$ iff $\sqrt n\{T_n-\Psi(P)\}$ converges in law under $P^n$. \Cref{thm:main}(1) implies $\Rfam(T_n)\subseteq\Lfam(\gamma)$ and that $\gamma|_V$ is a bounded linear functional for every $V\in\Rfam(T_n)$. The linearity family is closed under subspaces and closure, and every member lies in a closed maximal member (\cref{lem:odd}). Regularity itself passes to closure by \cref{lem:closure}. Thus maximal linearity subspaces are the largest subspaces compatible with the necessary linearity condition, and several such subspaces may coexist. Attaining estimators are constructed under the conditions in Section~\ref{sec:sufficiency}.

\begin{corollary}[impossibility]\label{cor:impossible}
If $\gamma$ is not linear on $\T$, no estimator sequence is regular. More precisely, no estimator sequence is regular along any convex cone on which $\gamma$ is not additive; in particular along any subspace not contained in the odd cone $\Nodd(\gamma)$.
\end{corollary}

\Cref{cor:impossible} shows that a nonlinear directional derivative precludes regularity over the whole tangent space. The same conclusion holds on any cone where the derivative is not additive. On subspaces and cones where regularity is possible, the convolution theorem supplies the corresponding variance bounds.

\subsection{The core and the additivity space}

The core consists of the directions shared by every maximal linearity subspace. The definition below expresses it directly in terms of $\gamma$; Theorem~\ref{thm:core} proves that this expression equals the intersection. The additivity space imposes a stronger condition: a direction must change the derivative by the same amount at every local parameter, including those outside the odd cone.

\begin{definition}\label{def:core}
The \emph{core} of $\gamma$ is
\[
\Hk(P):=\{w\in\Nodd(\gamma):\ \gamma(v+cw)=\gamma(v)+c\,\gamma(w)\ \ \forall v\in\Nodd(\gamma),\ c\in\R\}.
\]
The \emph{additivity space} is $\Afam(\gamma):=\{w\in\T:\gamma(h+w)=\gamma(h)+\gamma(w)\ \forall h\in\T\}$.
\end{definition}

\begin{theorem}[core]\label{thm:core}
\begin{enumerate}[label=(\arabic*)]
\item For $w\in\Hk(P)$ and $v\in\Nodd(\gamma)$, $\gamma$ is linear on $\spn(v,w)$; in particular $\spn(v,w)\subseteq\Nodd(\gamma)$.
\item $\Hk(P)$ is a closed linear subspace and a linearity subspace.
\item For every $V\in\Lfam(\gamma)$, $V+\Hk(P)\in\Lfam(\gamma)$. Consequently $\Hk(P)=\bigcap_{V\in\Lfam_{\max}(\gamma)}V$.
\item $\Afam(\gamma)$ is a closed linearity subspace with $\Afam(\gamma)\subseteq\Hk(P)$, with equality whenever $\Nodd(\gamma)=\T$ ($\gamma$ odd).
\item $\gamma$ is linear on $\T$ iff $\Hk(P)=\T$.
\end{enumerate}
\end{theorem}

\Cref{thm:core} shows that every linearity subspace can be enlarged by the core. Consequently, any estimator regular along a maximal subspace is also regular along the core and is subject to its variance bound. The additivity space is contained in the core because its directions preserve linearity on every linearity subspace.

The two spaces can differ. For $\gamma(h)=h_1^2/\norm h$ on $\R^2$ (with $\gamma(0)=0$), the core is $\{h_1=0\}$ but the additivity space is $\{0\}$. Thus a shared regularity subspace need not give a translation symmetry of the whole target. Their equality is the condition used for the risk factorization in \cref{sec:structure}; it holds for the maxima and order statistics treated below.

\subsection{Variance bounds on regularity subspaces}

On a specified linearity subspace, the Riesz representer of the derivative is the efficient influence function, and its squared norm gives the variance bound. The next proposition relates these influence functions by orthogonal projection, allowing the bounds to be compared within the tangent space.

\begin{proposition}\label{prop:bounds}
For $V\in\Lfam(\gamma)$ let $\tilde\psi_V\in\overline V$ be the Riesz representer of $\gamma|_V$ (extended to $\overline V$) and $V_V:=\norm{\tilde\psi_V}^2=\sup_{h\in V\setminus0}\gamma(h)^2/\norm h^2$. Then:
\begin{enumerate}[label=(\arabic*)]
\item $V\subseteq V'$ in $\Lfam(\gamma)$ implies $\tilde\psi_V=\Pi_{\overline V}\tilde\psi_{V'}$ and $V_V\le V_{V'}$;
\item $\Vstar(P):=\sup_{V\in\Lfam(\gamma)}V_V=\sup_{h\in\Nodd(\gamma)\setminus0}\gamma(h)^2/\norm h^2$;
\item $\Vk(P):=V_{\Hk(P)}$ satisfies $\Vk(P)\le V_V\le\Vstar(P)$ for every $V\in\Lfam_{\max}(\gamma)$.
\end{enumerate}
For a linearity cone $C$, $\tilde\psi_C$ denotes the representer of the linear functional on $\clspan C$ that agrees with $\gamma$ on $C$ (when bounded), and $V_C:=\norm{\tilde\psi_C}^2$ depends only on $\spn C$ and that functional.
\end{proposition}

\begin{corollary}[the Gaussian factor guaranteed by a regularity family]\label{cor:supremum}
Under the model and differentiability assumptions of \cref{thm:main}, let $T_n$ satisfy $\sqrt n\{T_n-\Psi(P)\}\wto L$ under $P^n$. Then
\[
v(T_n):=\sup_{C\in\Rfam^{\rm cone}(T_n)}V_C<\infty,\qquad
L=N(0,v(T_n))*M
\]
for a probability law $M$, and $\sup_{V\in\Rfam(T_n)}V_V\le\Vstar(P)$. The factor $N(0,v(T_n))$ is guaranteed by the regularity family. The residual $M$ may itself contain Gaussian noise, so $v(T_n)$ need not be the largest Gaussian divisor of $L$.
\end{corollary}

\Cref{cor:supremum} combines the bounds implied by one estimator's regularity family into a single convolution representation. Its Gaussian factor is required by regularity; the residual may contain further Gaussian variation. The following corollary gives the resulting lower bounds for losses, including cases where the residual has no finite variance.
A loss is bowl-shaped here if it is nonnegative, symmetric, and nondecreasing in the magnitude of its argument.

\begin{corollary}[loss functions]\label{cor:loss}
Under the assumptions of \cref{cor:supremum}, for $C\in\Rfam^{\rm cone}(T_n)$ and bowl-shaped $\ell$, $\int\ell\,dL\ge\int\ell\,dN(0,V_C)$ (Anderson's lemma, \citealp[Lemma 8.5]{vdV1998}); if $L$ has finite variance, $\Var L\ge v(T_n)$, with equality iff $M$ in \cref{cor:supremum} is a point mass.
\end{corollary}

\Cref{cor:loss} gives a common lower bound for symmetric losses that increase with the magnitude of the error. The bound is a consequence of the convolution representation rather than a separate variance calculation. In particular, attaining the centered Gaussian law removes the residual for that regularity requirement.

\Cref{prop:bounds} shows that enlarging a linearity subspace can only increase its variance bound. The core gives a lower bound shared by all maximal choices, whereas $\Vstar$ records the largest bound over linearity subspaces. A cone can span more directions than any such subspace, so its bound need not be bounded by $\Vstar$.

\subsection{Maxima of smooth parameters}

For the maximum of smooth parameters, the directional derivative is linear precisely along subspaces where all active branches change equally to first order. The largest such subspace is also the derivative's additivity space. Every active gradient has the same projection onto it, giving the efficient influence function. The next theorem establishes this geometry for sublinear derivatives; changing sign gives the superlinear case.

\begin{definition}\label{def:kink}
For a sublinear derivative $\gamma$, define $\LamP:=\{\lambda\in\overline\T:\ip\lambda h\le\gamma(h)\ \forall h\}$ and $\FP:=\clspan(\LamP-\LamP)$. Thus $\FP$ is generated by differences between the supporting gradients. For a superlinear derivative, apply the statements below to $-\gamma$, with $\LamP:=-\LamP(-\gamma)$ (the superdifferential) and $\FP$ unchanged.
\end{definition}

\begin{theorem}[sublinear derivatives]\label{thm:kink}
Let $\gamma$ be sublinear and continuous.
\begin{enumerate}[label=(\arabic*)]
\item $\LamP$ is nonempty, convex, closed and bounded, and $\gamma(h)=\max_{\lambda\in\LamP}\ip\lambda h$.
\item $\Nodd(\gamma)=\Hk(P)=\Afam(\gamma)=\FP^\perp\cap\T$ (the lineality space of $\gamma$), and $\Lfam_{\max}(\gamma)=\{\Hk(P)\}$.
\item $\gamma(h)=\ip\lambda h$ for all $h\in\Hk(P)$, $\lambda\in\LamP$; $\psik:=\tilde\psi_{\Hk}=\Pi_{\Hk}\lambda$ for every $\lambda\in\LamP$; under \assT, $\T=\Hk\oplus\FP$, $\psik$ is the minimum-norm element of $\claff\LamP=\lambda+\FP$, characterized as the unique element of the affine hull orthogonal to all tie-breaking directions, and $\Vk(P)=\Vstar(P)=\dist(0,\claff\LamP)^2$.
\item $\Vk(P)\le\norm\lambda^2$ for all $\lambda\in\LamP$, with equality iff $\lambda\perp\FP$.
\item (Polyhedral closed form, singular case included.) If $\LamP=\conv\{\lambda_1,\dots,\lambda_K\}$ with $\lambda_a\in\T$ and Gram matrix $\Sigma\succeq0$, then $\Vk(P)=\min\{w^\top\Sigma w:\one^\top w=1\}$, equal to $1/(\one^\top\Sigma^+\one)$ if $\one\in\range(\Sigma)$ and to $0$ otherwise (equivalently $\Vk=0$ iff $0\in\claff\LamP$); in the first case $\psik=\sum_aw^*_a\lambda_a$ with $w^*=\Sigma^+\one/\one^\top\Sigma^+\one$; and $\codim_\T\Hk=\rank(\lambda_a-\lambda_1)_{a\ge2}\le K-1$.
\end{enumerate}
\end{theorem}

\Cref{thm:kink} identifies the efficient influence function along the core as the minimum-norm element of the closed affine hull of the active gradients. Orthogonal projection removes the gradient components that distinguish the tied branches. For finitely many branches, the calculation reduces to variance minimization over weights that sum to one; these weights need not be nonnegative.

For independent active components with variances $\sigma_a^2>0$, the bound is $(\sum_a\sigma_a^{-2})^{-1}$. It is strictly below every component variance when at least two components are tied. In \eqref{eq:intro-experiment}, this gives $1/2$ for estimating the common level. The gain reflects the independent information contributed by both observations.

\begin{corollary}[the common Gaussian bound]\label{cor:core-convolution}
Under the model and differentiability assumptions of \cref{thm:main}, let $T_n$ satisfy $\sqrt n\{T_n-\Psi(P)\}\wto L$ under $P^n$.
\begin{enumerate}[label=(\arabic*)]
\item If $T_n$ is maximally regular, meaning regular along some $V\in\Lfam_{\max}(\gamma)$, then it is regular along $\Hk(P)$ and
\[
L=N(0,\Vk(P))*M',\qquad v(T_n)\ge\Vk(P).
\]
\item If $\gamma$ is sublinear or superlinear, maximal regularity is regularity along $\Hk(P)$, and
\[
\Vk(P)=\Vstar(P)=\dist(0,\claff\LamP)^2.
\]
Every $\Hk$-regular estimator then has $L=N(0,\Vk(P))*M$.
\end{enumerate}
\end{corollary}

\Cref{cor:core-convolution} gives the convolution conclusion associated with the core. For a maximum, the affine-hull bound applies to every core-regular estimator, whether or not it is asymptotically linear. For a general derivative, the core bound applies to every estimator regular along a maximal subspace, while different maximal subspaces can have larger bounds.

\subsection{Reduction to finitely many smooth parameters}

When the derivative has several maximal linearity subspaces, their variance bounds need not equal the core bound. For parameters that are functions of finitely many smooth parameters, the following lemma reduces the comparison to a finite-dimensional projection using the covariance matrix of their gradients. In this setting $\Hk(\gamma)$ denotes the same core as $\Hk(P)$, and $\Hk(g)$ denotes the core of the reduced derivative $g$.

\begin{lemma}[reduction]\label{lem:reduction}
Let $\lambda_1,\dots,\lambda_K\in\T$ be linearly independent with Gram matrix $\Sigma$, and write
\[
\ell(h):=(\ip{\lambda_a}h)_{a\le K},\qquad
\ell^*c:=\sum_ac_a\lambda_a.
\]
Suppose $\gamma=g\circ\ell$ with $g:\R^K\to\R$ continuous and positively homogeneous. Then $\ker\ell\subseteq\Afam(\gamma)$ and
\[
\begin{split}
\Nodd(\gamma)&=\ell^{-1}(\Nodd(g)),\qquad
V\in\Lfam(\gamma)\ \Longleftrightarrow\ \ell(V)\in\Lfam(g),\\
\Lfam_{\max}(\gamma)&=\{\ell^{-1}(U):U\in\Lfam_{\max}(g)\},\\
\Hk(\gamma)&=\ell^{-1}(\Hk(g)),\qquad
\Afam(\gamma)=\ell^{-1}(\Afam(g)).
\end{split}
\]
If $g|_U=b^\top(\cdot)$ on $U\in\Lfam(g)$ and $V=\ell^{-1}(U)$, then
\[
V^\perp\cap\T=\ell^*(U^\perp),\qquad
\tilde\psi_V=\Pi_V\ell^*b,\qquad
V_V=\min_{c\in U^\perp}(b-c)^\top\Sigma(b-c).
\]
\end{lemma}

\Cref{lem:reduction} shows that directions invisible to the finite parameter vector belong to the additivity space. The remaining problem is a projection under the covariance metric $\Sigma$, not an unweighted Euclidean calculation. This gives explicit representers for each compatible subspace and makes the comparison below possible.

\begin{proposition}[$\Vstar$ for odd polyhedral derivatives]\label{prop:vstar}
In the setting of \cref{lem:reduction} with $\Sigma$ nonsingular, let $g$ be odd, continuous, positively homogeneous and piecewise linear with pieces $b_i^\top(\cdot)$ on the closed convex cones $C_i$ of a fan. Then
\[
\Vstar(P)=\max_i\norm{\Pi_{\mathcal C_i}(\Sigma^{1/2}b_i)}^2,\qquad\mathcal C_i:=\Sigma^{-1/2}C_i,
\]
with $\Pi_{\mathcal C}$ the Euclidean projection onto a closed convex cone. The maximum is attained; if $z^\star$ maximizes $g(z)^2/z^\top\Sigma^{-1}z$ and $h^\star:=\ell^*\Sigma^{-1}z^\star$, every maximal linearity subspace containing $h^\star$ has bound $\Vstar$, and a $V$-efficient estimator for such $V$ has Gaussian component $\Vstar$. Moreover $\Vstar\le\max_ib_i^\top\Sigma b_i$, with equality iff $\Sigma b_i\in C_i$ for some maximizing $i$.
\end{proposition}

\Cref{prop:vstar} computes the largest subspace bound by projecting onto the cones where an odd polyhedral derivative selects a linear branch. This formula makes the variance comparison explicit. The median provides a concrete case in which different maximal subspaces admit regular estimators but cannot be combined into a larger regularity subspace.

\subsection{The median programme at a three-way tie: competing maximal spaces}\label{ex:median}
The median of three tied programme values has the same common-level bound as the maximum, but admits additional, mutually incompatible regularity requirements. Let $\mu(P)=\theta\one$, where $\one=(1,1,1)^\top$, and let $\Psi(Q)=\med_{j\le3}\mu_j(Q)$. Suppose the programme values are pathwise differentiable with gradients $\lambda_j\in\T$. Write $\ell(h)=(\ip{\lambda_j}h)_{j=1}^3$ and $\Sigma=(\ip{\lambda_i}{\lambda_j})_{i,j=1}^3$, and assume $\Sigma$ is nonsingular. The directional derivative is then $\gamma=\med\circ\ell$.

The median is odd, so $\Nodd(\gamma)=\T$ and its core equals its additivity space. More specifically, with $F=\spn\{\lambda_i-\lambda_j:i,j\le3\}$,
\[
\Hk=\Afam(\gamma)=\ell^{-1}(\R\one)=F^\perp\cap\T,
\qquad
\Vk=\frac{1}{\one^\top\Sigma^{-1}\one}.
\]
These are the same core and bound as for the maximum: the core consists of changes that move all three programme values equally. The median can also vary linearly when one contrast changes. Its maximal linearity spaces in the reduced parameter space are all the planes
\[
U_a=\{z\in\R^3:a^\top z=0\},\qquad
0\ne a\perp\one,
\]
with proportional normals defining the same plane. Their score-space counterparts are $V_a=\ell^{-1}(U_a)$. This continuum is the three-programme case of \cref{thm:order,prop:m1}; each plane excludes one direction that changes relative programme values.

The variance comparison is particularly transparent when $\Sigma=I_3$. The core bound is $1/3$, while the bounds on the maximal score spaces $V_a$ range over $[1/3,1/2]$. For distinct $i,j,k$, the three midpoint planes $\{2z_j=z_i+z_k\}$ have bound $1/3$. The three pair-tie planes $\{z_i=z_k\}$ have bound $1/2$. Thus the same core can be enlarged in ways that require either no additional Gaussian variation or a strictly larger component.

These bounds have explicit attaining estimators. Suppose $\hat\mu$ is jointly asymptotically linear at $P$ with component influence functions $\lambda_j$. When $\Sigma=I_3$, the average $(\hat\mu_1+\hat\mu_2+\hat\mu_3)/3$ is efficient along the core and along the score-space preimage of every midpoint plane. On a midpoint plane the median is the midpoint coordinate, which also equals the three-programme average. On the preimage of $\{z_i=z_k\}$, the pair average $(\hat\mu_i+\hat\mu_k)/2$ is efficient with variance $1/2$, since the common value of the tied pair is the median. The core-efficient average cannot be regular along this latter space, as its response does not reproduce that common pair value.

For general $\Sigma$, the covariance matrix determines both the bound on each plane and the weights of its efficient estimator. \Cref{app:median-details} gives these formulas, the precise conditions under which the core-efficient pool accommodates a maximal space, and the corresponding single-programme and plug-in comparisons. The core therefore supplies a variance bound shared by all maximal regularity subspaces, while the chosen additional contrast determines whether a different influence function and a larger variance are necessary.

\section{Efficient estimation under subspace and cone regularity}\label{sec:cost}

An estimator is efficient along a specified subspace when it attains the centered Gaussian law in the convolution theorem. We characterize this attainment through its influence function and give explicit constructions. Comparing the resulting bounds then determines whether regularity along additional directions requires a larger asymptotic variance.

\subsection{Influence functions and attainable bounds}\label{sec:sufficiency}

For an asymptotically linear estimator, regularity along a subspace requires its influence function to represent the parameter's derivative there. The Riesz representer has the smallest norm among such influence functions. We first establish attainment in the Gaussian limit experiment and then give sufficient conditions for attainment by estimator sequences.

\begin{proposition}[limit experiment]\label{prop:limitexp}
For every $V\in\Lfam(\gamma)$ the statistic $T:=\Delta(\tilde\psi_V)$ in the Gaussian-shift experiment $\{\Delta+\ip h\cdot:h\in\T\}$ satisfies: the law of $T-\gamma(h)$ under $h$ equals $N(0,V_V)$ for every $h\in V$. Hence in the limit experiment the linearity subspaces are exactly the subspaces along which some statistic is equivariant in law. For a linearity cone $C$ the same holds iff the linear functional agreeing with $\gamma$ on $C$ extends boundedly to $\spn C$ (necessity: \cref{thm:main}(1) read in the limit experiment).
\end{proposition}

Proposition~\ref{prop:limitexp} shows that the bound on each linearity subspace is attained in the limit experiment. For a cone, the linear extension must also be bounded in the score norm, as \cref{ex:cone} illustrates. Attainment in the original statistical model additionally requires an estimator sequence with the corresponding asymptotic expansion.

For an estimator sequence in the original model, asymptotic linearity makes the same relation explicit. The next theorem identifies both the directions on which the estimator has a stable centered limit and its response to every other local alternative.

\begin{theorem}[asymptotically linear estimators]\label{thm:al}
Let $T_n$ be asymptotically linear at $P$ with influence function $\phi\in L_2^0(P)$:
$\sqrt n\{T_n-\Psi(P)\}=\sqrt n\,\mathbb P_n\phi+o_{P^n}(1)$.
Define the \emph{agreement set} $\mathcal E(\phi):=\{h\in\T:\gamma(h)=\ip\phi h\}$. Then
\[
\Rfam^{\rm cone}(T_n)=\{C\text{ convex cone}:C\subseteq\mathcal E(\phi)\},
\qquad
\Rfam(T_n)=\{V\text{ linear}:V\subseteq\mathcal E(\phi)\}.
\]
For $V\in\Rfam(T_n)$, the representer and convolution laws are
\[
\tilde\psi_V=\Pi_{\overline V}\phi,\qquad
L=N(0,\norm\phi^2),\qquad
M_V=N(0,\norm{\phi-\Pi_{\overline V}\phi}^2),
\]
and $v(T_n)=\sup_{C\subseteq\mathcal E(\phi)}\norm{\Pi_{\clspan C}\phi}^2$, with the supremum over convex cones.
Under every local alternative with score $h\in\T$,
\[
\sqrt n\{T_n-\Psi(P_{n,h})\}\wto
N(\ip\phi h-\gamma(h),\norm\phi^2).
\]
\end{theorem}

Theorem~\ref{thm:al} connects asymptotic variance and local bias through the influence function. Its projection onto a regularity subspace gives the efficient influence function on that subspace; the orthogonal component gives the Gaussian convolution residual. The agreement condition determines where the limit remains centered. Thus asymptotic normality at $P$ alone does not determine the estimator's regularity or its local bias.

\begin{corollary}[sublinear derivatives]\label{cor:alkink}
If $\gamma$ is sublinear and \assT holds, an asymptotically linear $T_n$ is regular along $\Hk(P)$ iff $\Pi_{\Hk}\phi=\psik$ iff $\phi\in\claff\LamP\oplus\T^\perp$; in the nonparametric model iff $\phi\in\claff\LamP$. Then $M=N(0,\norm{\phi-\psik}^2)$.
\end{corollary}

Corollary~\ref{cor:alkink} makes this characterization especially concrete for a sublinear derivative. Core regularity requires the influence function to lie in the displayed affine class, and efficiency requires its orthogonal remainder to vanish. Estimator construction can therefore use the projection defining the bound as a choice of weights.

For a piecewise-linear target, the agreement condition also has a direct interpretation. An estimator that follows one branch remains regular on directions for which that branch continues to determine the target. In the maximum example, these are directions along which the chosen programme remains best.

For a specified linearity subspace at a full tie, this prescription gives an estimator from jointly asymptotically linear branch estimates. The common-shift restriction ensures that the weights reproduce the target value, while variance minimization determines how the branch estimates should be combined.

\begin{proposition}[sequence-level sufficiency]\label{prop:suff}
In the setting of \cref{lem:reduction} with $\Sigma$ nonsingular, suppose $\Psi(Q)=g(\mu(Q))$ near $P$ with $g$ piecewise linear and $\mu(P)=\theta\one$ (a full tie), and $\hat\mu$ jointly asymptotically linear with influence functions $(\lambda_a)$. Let $U\in\Lfam(g)$ with $\one\in U$, $g=b^\top(\cdot)$ on $U$, $V=\ell^{-1}(U)$, and $c^\star:=\arg\min_{c\in U^\perp}(b-c)^\top\Sigma(b-c)$. Then $\hat T_V:=(b-c^\star)^\top\hat\mu$ is asymptotically linear with influence function $\tilde\psi_V=\ell^*(b-c^\star)$, consistent for $\Psi(P)$, and $V$-efficient; the same holds with $\Sigma$ replaced by a consistent estimate. Hence, at a full tie, every such linearity subspace admits an explicit estimator that is regular and efficient along it.
\end{proposition}

Proposition~\ref{prop:suff} shows that, under its finite-dimensional assumptions, the Gaussian bound is attained by an explicit estimator sequence. Estimating the covariance matrix preserves its first-order behavior because the weight constraints hold exactly. This construction fixes the linearity subspace and starts from given asymptotically linear branch estimates. The applications below establish feasible selection of tied branches under their model-specific conditions.

\subsection{Characterizing efficiency and comparing requirements}\label{sec:efficiency}

The constructions above attain the bound with an asymptotically linear estimator. The convolution theorem gives a converse: any estimator that is regular along the same subspace and attains its centered Gaussian bound must have that influence function. Efficiency therefore determines first-order behavior even when the estimator was constructed by a nonlinear procedure.

\begin{theorem}[characterization and uniqueness]\label{thm:efficiency}
For $V\in\Lfam(\gamma)$: $T_n$ is $V$-efficient iff $\sqrt n\{T_n-\Psi(P)\}=\sqrt n\,\mathbb P_n\tilde\psi_V+o_{P^n}(1)$. Any two $V$-efficient sequences satisfy $\sqrt n(T_n-T_n')\to0$ in $P^n$-probability.
\end{theorem}

Theorem~\ref{thm:efficiency} shows that attaining the Gaussian bound leaves no first-order freedom beyond the efficient influence function. In particular, two efficient constructions are asymptotically equivalent. The next result gives a complementary interpretation of the same bound through risk over local alternatives, without requiring the competing estimators to be regular.

\begin{theorem}[local asymptotic minimaxity along $V$]\label{thm:lam}
For $V\in\Lfam(\gamma)$, every estimator sequence and every bowl-shaped $\ell$: $\sup_I\liminf_n\sup_{h\in I}\E_{\Pn h}\ell(\sqrt n\{T_n-\Psi(P_{n,h})\})\ge\int\ell\,dN(0,V_V)$, the supremum over finite $I\subseteq V$ \citep[Theorem 25.21]{vdV1998}. In particular the core bound $\Vk(P)$ is a local minimax bound over local alternatives in the core for all estimators, regular or not.
\end{theorem}

Theorem~\ref{thm:lam} shows that the Gaussian benchmark also limits the local minimax performance of all estimators when alternatives range over the specified subspace. Taking that subspace to be the core gives a benchmark on core alternatives. To assess precision when other directions are included, we must account for the corresponding change in the estimation problem.

\begin{proposition}[efficiency--regularity trade-off]\label{prop:tradeoff}
Let $V\subsetneq V'$ be linearity subspaces with $V_V<V_{V'}$. Then no $V$-efficient estimator is regular along $V'$, and every estimator regular along $V'$ has $v(T_n)\ge V_{V'}>V_V$. If $\gamma$ is sublinear or superlinear, every linearity subspace is contained in $\Hk$, so every $\Hk$-efficient estimator is regular along all linearity subspaces.
\end{proposition}

Proposition~\ref{prop:tradeoff} shows that regularity along a larger subspace can require a larger asymptotic variance. If the bound increases, an estimator attaining the smaller bound cannot be regular along the larger subspace. The comparison thus concerns efficiency under different regularity requirements; a comparison under a common risk criterion must also consider alternatives outside the smaller subspace.

When a derivative has competing maximal linearity subspaces, efficiency can select different influence functions for different requirements. The core gives their common variance bound, but the maximal regularity subspaces need not be comparable. The median example in \cref{ex:median} illustrates both an expansion with no increase in the bound and one with a strictly positive cost.

For a maximum, variance minimization may assign signed weights to the active branches. In the nonsingular case the efficient weights form a probability vector exactly when $\Sigma^{-1}\one\ge0$ componentwise. This distinction will determine whether the core-efficient estimator's local drift can be positive; signed weights are therefore relevant to local response as well as variance.

\subsection{Regularity along one-sided cones}

Pooling estimates of tied branches can attain the core variance bound. Regularity along the cone where a designated branch remains optimal additionally restricts the estimator's response when the tied values separate. For an exposed vertex of a finite active-gradient hull, this cone spans the tangent space, and its efficient influence function is the gradient of the designated branch.

\begin{corollary}[variance bound under one-sided regularity]\label{cor:onesided}
Let $\gamma=\max_{a\in\mathcal A^*}\ell_a$ be the derivative at a tie among finitely many branches, with $\ell_a=\ip{\lambda_a}\cdot$, and let $\lambda_a$ be a vertex of $\LamP=\conv\{\lambda_b\}$. Then the strict cone $\{h:\ell_a(h)>\ell_b(h)\ \forall b\text{ with }\lambda_b\ne\lambda_a\}$ is open and nonempty, so $C_a:=\{h:\ell_a(h)\ge\ell_b(h)\ \forall b\}$ (``component $a$ stays maximal'') spans $\T$; $\gamma=\ell_a$ on $C_a$; and every estimator regular along $C_a$ has $L=N(0,\norm{\lambda_a}^2)*M$, attained by any asymptotically linear estimator with influence function $\lambda_a$. The core-efficient estimator is regular along $C_a$ iff $\psik=\lambda_a$, equivalently $\norm{\lambda_a}^2=\Vk$ (no variance reduction relative to component $a$, \cref{thm:kink}(4)); a plug-in estimator with centered local limit $\max_b\{Z_b+\ell_b(h)\}-\gamma(h)$, where $Z_b=\Delta(\lambda_b)$, is never regular along $C_a$ when $\gamma$ is nonlinear. Without the vertex hypothesis $C_a$ may reduce to $\{0\}$, and in general the bound along $C_a$ is $\norm{\Pi_{\clspan C_a}\lambda_a}^2$ (\cref{thm:main}(2)).
\end{corollary}

\Cref{cor:onesided} shows why a one-sided requirement can entail the full component variance: although the target agrees with that component only on a cone, the linear response selected by the cone extends to its entire span. For a vertex this span is $\T$. In \eqref{eq:intro-experiment}, requiring a common law on $d\ge0$ therefore gives variance $1$, attained by $U+V$, compared with the core bound $1/2$. The vertex condition is essential; a nonvertex branch can have a smaller cone span and hence a smaller bound.

\subsection{Convolution residuals of standard estimators}\label{sec:estimators}

For a translation-equivariant reduced derivative, a weighted average of the branch estimates and their contrasts have independent Gaussian limits. This decomposition gives explicit convolution residuals for plug-in estimation, fixed weighting, independent selection and smoothing. It also describes their local bias when the branch values separate. The two propositions below use this translation property directly and do not require \assS of \cref{sec:structure}.

\begin{proposition}[plug-in decomposition for translation-equivariant reduced derivatives]\label{prop:plugin}
Let $\gamma=g\circ\ell$ as in \cref{lem:reduction} with $\Sigma$ nonsingular, and suppose
\[
g(z+c\one)=g(z)+c\theta,\qquad z\in\R^K,\ c\in\R,
\]
where $\theta:=g(\one)$ ($\theta=1$ for max, min, median and all order statistics; $\theta=0$ for the range). Set
\[
w^*=\frac{\Sigma^{-1}\one}{\one^\top\Sigma^{-1}\one},
\qquad Z=(\Delta(\lambda_a))_{a\le K}\sim N(0,\Sigma),
\qquad W=(I-\one w^{*\top})Z.
\]
Then $W\indep w^{*\top}Z$, $w^{*\top}Z\sim N(0,1/(\one^\top\Sigma^{-1}\one))$, and $w^{*\top}W=0$. Moreover,
\[
F^\perp:=\ell^{-1}(\R\one)\subseteq\Afam(\gamma),
\qquad \tilde\psi_{F^\perp}=\theta\sum_aw_a^*\lambda_a,
\qquad V_{F^\perp}=\frac{\theta^2}{\one^\top\Sigma^{-1}\one}.
\]
Suppose additionally that $\Psi(Q)=q(\mu(Q))$ near $P$, where $q:\R^K\to\R$ is measurable, $\mu$ is pathwise differentiable with influence functions $(\lambda_a)$, $q$ is Hadamard directionally differentiable at $\mu(P)$ with derivative $g$, and $\hat\mu$ is jointly asymptotically linear with those influence functions. Then $T_n=q(\hat\mu)$ is regular along $F^\perp$, with
\[
L=\Law(g(Z))=N(0,V_{F^\perp})*M,
\qquad M=\Law(g(W)).
\]
Under a local alternative with score $h$, write $\ell(h)=c\one+z_\perp$, where $w^{*\top}z_\perp=0$. The centered limit is
\[
\theta w^{*\top}Z+[g(W+z_\perp)-g(z_\perp)],
\]
and the two summands are independent. The mean of $M$ is positive for $g=\max$ when $\gamma$ is nonlinear, negative for $g=\min$, and zero for odd $g$, including the median.
\end{proposition}

Proposition~\ref{prop:plugin} identifies the plug-in residual as $g(W)$. For a maximum at a tie, this term has positive mean; for an odd reduced derivative, Gaussian symmetry gives a zero mean residual at the baseline. Under other local alternatives, the shifted contrast term also estimates changes in the relative branch values. Removing the residual therefore changes both the baseline limit distribution and the estimator's local bias.

The directional delta-method condition requires
\[
\frac{q(\mu(P)+t_jz_j)-q(\mu(P))}{t_j}\longrightarrow g(z)
\quad\text{whenever }t_j\downarrow0,\ z_j\to z\in\R^K.
\]
It holds automatically when $q=g$ and $\mu(P)=a\one$ is a full tie: translation equivariance and positive homogeneity give
\[
\frac{g(a\one+tz)-g(a\one)}t=g(z),\qquad t>0,
\]
and $g$ is continuous. At a general baseline the condition concerns the level map $q$ at $\mu(P)$; for a maximum its derivative retains only the active coordinates.

\begin{proposition}[fixed weights; independent selection; softmax at a tie]\label{prop:fixed}
Retain the Gaussian decomposition and parameter assumptions of \cref{prop:plugin}, and suppose $\theta=1$:
\begin{enumerate}[label=(\alph*)]
\item an estimator $T_n$ asymptotically linear for $\Psi(P)$ with influence function $\psi_w:=\sum_aw_a\lambda_a$, $\one^\top w=1$, is regular along $F^\perp$ with $M=N(0,w^\top\Sigma w-1/\one^\top\Sigma^{-1}\one)$ and, under a local alternative $h$, drift $\ip{\psi_w}h-\gamma(h)$ (\cref{thm:al}); when $\Hk=F^\perp$ (max, min, median and order statistics at a full tie) this equals $\ip{\psi_w-\tilde\psi_{F^\perp}}h-\gamma(h_\perp)$ with $h_\perp$ the $\Hk^\perp$-component, but not in general (\cref{app:examples});
\item in the limit experiment, $T=Z_{\hat a}$ with an external random index $\hat a\indep Z$ has $M=\Law(W_{\hat a})$: mean zero, and non-degenerate iff $\sum_aP(\hat a=a)\Var(W_a)>0$;
\item at a full tie $\mu(P)=a\one$ with $q=g$, for a symmetric smoother whose first-order limit uses uniform weights (a sufficient regime is given in \cref{ex:policy}), the uniform-weight linear limit has $M=N(0,\sigma^2_{\rm unif})$ with $\sigma^2_{\rm unif}=\one^\top\Sigma\one/K^2-1/\one^\top\Sigma^{-1}\one\ge0$, with equality iff $\Sigma\one\propto\one$ (Cauchy--Schwarz: $(\one^\top\Sigma\one)(\one^\top\Sigma^{-1}\one)\ge K^2$).
\end{enumerate}
\end{proposition}

Proposition~\ref{prop:fixed} identifies the corresponding residuals for weighted estimation, external selection and a uniform-weight smoothing limit. Under its stated first-order regime, uniform weighting attains the common-level bound precisely when the covariance matrix has constant row sums. External selection removes the selection-induced mean at the tie but generally leaves residual variation. Part (b) concerns selection independent of a fixed Gaussian evaluation experiment; an actual sample split must additionally account for the information available in the evaluation sample.

These residual calculations distinguish an estimator's variance from the Gaussian variance bound. When the bound is attained and is positive, consistent variance estimation permits Wald inference along the corresponding regularity directions. The next section evaluates risk and coverage when other local changes in the parameter are included; \cref{sec:ci} gives the resulting confidence-interval statements.

\section{Local risk beyond the core variance bound}\label{sec:structure}

The core variance bound characterizes efficiency along the core. To evaluate local risk over a larger set of alternatives, we must also account for changes in the parameter along orthogonal directions. Under an additivity condition, its directional derivative separates into linear and nonlinear parts. This decomposition allows the minimax risk in the Gaussian limit experiment to be expressed through the core bound and the risk of estimating the nonlinear part.

\begin{assumption}[additivity along the core]\label{ass:additivity}
$\Hk(P)=\Afam(\gamma)$.
\end{assumption}

When $\gamma$ is Lipschitz, \assS is equivalent (\cref{thm:clarke}) to $\Hk(P)\subseteq(\partial^C\gamma(0)-\partial^C\gamma(0))^\perp$, where $\partial^C\gamma(0)$ is the Clarke generalized gradient, defined in \cref{app:geometry}. \assS holds for sublinear and superlinear $\gamma$ (\cref{thm:kink}(2)), for odd $\gamma$ (\cref{thm:core}(4)), for order statistics (\cref{thm:order}, \cref{app:order}), and for $\gamma=g\circ\ell$ whenever it holds for $g$ (\cref{lem:reduction}); it can fail for polyhedral $\gamma$ (\cref{prop:addpl}(4)). Thus the core need not be the space of additive translations for a general directional derivative. Under \Cref{ass:paths,ass:additivity} write $h=h_\kappa+h_\perp$ with $h_\kappa\in\Hk(P)$ and $h_\perp\in\Hk(P)^\perp\cap\T$, and let $\psik:=\tilde\psi_{\Hk}$, $s:=\norm\psik$, $G\sim N(0,\Vk)$.

\begin{proposition}[splitting of the limit experiment]\label{prop:split}
Under \Cref{ass:paths,ass:additivity}: $\gamma(h)=\ip\psik{h_\kappa}+\gamma(h_\perp)$, and the Gaussian-shift limit experiment indexed by $h\in\T$ is the independent product of the experiments indexed by $h_\kappa$ and by $h_\perp$. The target is a linear functional of the first coordinate (bound $\Vk$) plus $\gamma|_{\Hk^\perp}$, whose core is $\{0\}$.
\end{proposition}

Proposition~\ref{prop:split} decomposes the Gaussian limit experiment into two independent components and the directional derivative into the sum of two functionals. The first component concerns a linear functional; the second concerns the remaining nonlinear functional. In the two-programme example, these are estimation of the common level $c$ and of $\abs d$.

\begin{proposition}[local bias of core-efficient estimators]\label{prop:blind}
Under \Cref{ass:paths,ass:additivity}, let $T_n$ be $\Hk$-efficient. For every $h\in\T$, $\sqrt n\{T_n-\Psi(P_{n,h})\}\wto N(-\gamma(h_\perp),\Vk(P))$ under $\Pn h$. The drift is $\le0$ for all $h$ iff $\gamma\ge\ip\psik\cdot$ on $\T$; for sublinear $\gamma$ iff $\psik\in\LamP$; for odd nonlinear $\gamma$ the drift takes both signs.
\end{proposition}

Proposition~\ref{prop:blind} determines the local bias of a core-efficient estimator. Its influence function $\psik$ is orthogonal to the score component perpendicular to the core, leaving the term $-\gamma(h_\perp)$ in its centered limit. This conclusion follows from core regularity and attainment of the centered Gaussian bound; asymptotic normality at $P$ alone is insufficient. An estimator that also responds to the nonlinear part of the derivative can have a different local bias and a nondegenerate convolution residual.

\subsection{Optimal risk in the product experiment}

We now minimize risk over all randomized decision rules in the Gaussian limit experiment. Throughout this subsection the experiment is that of \cref{prop:split}: $Z=(Z_\kappa,Z_\perp)$, independent Gaussian shifts indexed by $h_\kappa\in\Hk$ and $h_\perp\in\Hk^\perp\cap\T$, target $\tau(h)=\ip\psik{h_\kappa}+\gamma(h_\perp)$. We allow the core parameter to range over all of $\Hk$ and restrict only the orthogonal parameter to a set $B$. This product parameter space is essential to the risk equalities below. Throughout, $B$ is nonempty and minimax means the infimum over rules of the supremum of risk over the stated parameter set. These values also equal the suprema of minimax values over finite nonempty parameter subsets, by \cref{lem:global-minimax}. The lemma applies because the Gaussian-shift laws are dominated by the law at zero; the appendix verifies this without a separability assumption.

\begin{theorem}[additivity of squared-error minimax risk]\label{thm:lam-sq}
Under \Cref{ass:paths,ass:additivity}, for nonempty $B\subseteq\Hk^\perp\cap\T$ define
\[
\begin{split}
\mathcal M(B)&:=\inf_T\sup_{h\in\Hk\times B}\E_h(T-\tau(h))^2,\\
\mathcal M_\perp(B)&:=\inf_{T_\perp}\sup_{h_\perp\in B}\E_{h_\perp}(T_\perp-\gamma(h_\perp))^2,
\end{split}
\]
where the infima are over randomized statistics of $Z$ and of $Z_\perp$, respectively. Then
\[
\mathcal M(B)=\Vk(P)+\mathcal M_\perp(B).
\]
In particular, $\mathcal M(\{0\})=\Vk(P)$.
\end{theorem}

Theorem~\ref{thm:lam-sq} shows that under squared error the two estimation tasks contribute additively to the optimal risk. The Gaussian bound is the full answer when the orthogonal parameter is fixed at zero. For a larger $B$, the additional term is the optimal risk of estimating the residual target from its own experiment. The theorem identifies that decision problem without prescribing a universally optimal residual rule.

In the two-programme experiment \eqref{eq:intro-experiment}, pooling has risk $1/2+d^2$, whereas the plug-in has risk $1/2+\E(\abs{d+H}-\abs d)^2\le1$. In fact,
\begin{equation}\label{eq:two-programme-minimax}
\inf_T\sup_{c,d\in\R}\E_{c,d}\{T(U,V)-(c+\abs d)\}^2=1,
\end{equation}
and the plug-in attains this value. Its error at the tie is the non-Gaussian variable $G+\abs H$, so a nondegenerate convolution residual is compatible with minimax optimality over a larger parameter set. Pooling instead attains the core bound at the tie. These are different precision requirements in the same experiment; neither comparison determines which confidence interval has the shortest length. A direct verification of \eqref{eq:two-programme-minimax} is in \cref{app:ci-proofs}.

\begin{theorem}[bowl-shaped losses: the regular factor enters by convolution]\label{thm:lam-conv}
Under \Cref{ass:paths,ass:additivity}, let $\ell:\R\to[0,+\infty]$ be bowl-shaped and lower semicontinuous. Put $\ell_\kappa(c):=\E\ell(c+G)$ for $G\sim N(0,\Vk)$; this loss is again bowl-shaped and lower semicontinuous, and $\ell_\kappa(c)=c^2+\Vk$ for squared error. For nonempty $B\subseteq\Hk^\perp\cap\T$, define
\[
\begin{split}
\mathcal M_\ell(B)&:=\inf_T\sup_{h\in\Hk\times B}\E_h\ell(T-\tau(h)),\\
\mathcal N_\ell(B)&:=\inf_U\sup_{h_\perp\in B}\E_{h_\perp}\ell_\kappa(U(Z_\perp)-\gamma(h_\perp)),
\end{split}
\]
where the infima are over randomized rules in the respective experiments. Then
$\mathcal M_\ell(B)=\mathcal N_\ell(B)$.
\end{theorem}

Theorem~\ref{thm:lam-conv} extends the risk decomposition to other bowl-shaped losses. Estimation of the nonlinear component is evaluated under the loss averaged over the independent Gaussian error. Squared error gives the additive formula because this averaging adds $\Vk$; for a tail loss it gives $\ell_\kappa(c)=P(\abs{c+G}>t)$. The same Gaussian variance bound therefore enters each risk calculation through convolution with the chosen loss.

For small orthogonal neighborhoods, the residual risk has an explicit leading term. The relevant quantity is the range of the residual target over the neighborhood. Gaussian experiments indexed by nearby points are difficult to distinguish, so even the best procedure cannot substantially improve on centering that range to first order in squared risk.

\begin{proposition}[residual risk on small orthogonal neighborhoods]\label{prop:hedging}
Under \Cref{ass:paths,ass:additivity}, let $B_r:=\{h_\perp\in\Hk^\perp\cap\T:\norm{h_\perp}\le r\}$, take squared error, and let $D:=\sup_{h\in B_1}\gamma(h)-\inf_{h\in B_1}\gamma(h)$ (finite: $\gamma$ is continuous at $0$ and positively homogeneous). Then
\[
\frac{D^2r^2}{4}\Big(1-\frac{2r}{\sqrt{2\pi}}\Big)_+\ \le\ \mathcal M_\perp(B_r)\ \le\ \frac{D^2r^2}{4},\qquad\text{hence}\qquad\mathcal M_\perp(B_r)=\frac{D^2}{4}r^2+o(r^2).
\]
The additional minimax risk over orthogonal alternatives of radius $r$ is of second order in $r$, with coefficient $D^2/4$.
\end{proposition}

Proposition~\ref{prop:hedging} shows that the additional squared-error risk on an orthogonal ball of radius $r$ is $D^2r^2/4+o(r^2)$. Combining it with Theorem~\ref{thm:lam-sq} quantifies the increase above $\Vk$ when nearby changes in the nonlinear target must also be estimated. The expansion concerns small radii in the Gaussian limit experiment; its role for estimator sequences is through the local risk lower bound described below.

The risk equalities in this section are statements about the Gaussian product experiment with an unrestricted core parameter. Lemma~\ref{lem:transfer} carries their lower bounds to estimator sequences in the original model through finite collections of local alternatives. This transfer supplies a lower bound for every sequence; constructing a sequence that attains the residual optimum requires additional model-specific approximation and estimation arguments. The applications distinguish these requirements from attainment of the core Gaussian bound.

\subsection{Confidence intervals and local coverage}\label{sec:ci}

The local error distribution determines how an efficiency result can be used for inference. A positive Gaussian bound gives the usual Wald interval when it is attained. The same estimator can have a shifted limit on other local paths, and that shift determines its coverage. Let $\Phi$ be the standard normal distribution function and $z_p=\Phi^{-1}(p)$.

\begin{corollary}[confidence intervals for efficient estimators]\label{cor:ci-local}
Suppose \assT holds, $\Psi$ is directionally pathwise differentiable at $P$, and $T_n$ is $V$-efficient for $V\in\Lfam(\gamma)$. Put $s^2=V_V>0$, let $\hat s_n\ge0$ satisfy $\hat s_n\to s$ in $P^n$-probability, and define
\[
\delta(h)=\ip{\tilde\psi_V}h-\gamma(h),\qquad
I_n=\left[T_n-z_{1-\alpha/2}\frac{\hat s_n}{\sqrt n},\,
T_n+z_{1-\alpha/2}\frac{\hat s_n}{\sqrt n}\right],\quad 0<\alpha<1.
\]
For every fixed $h\in\T$ and every DQM path with score $h$,
\begin{equation}\label{eq:ci-local-coverage}
\Pn h\{\Psi(P_{n,h})\in I_n\}\longrightarrow
\Phi\!\left(z_{1-\alpha/2}-\frac{\delta(h)}s\right)
-\Phi\!\left(-z_{1-\alpha/2}-\frac{\delta(h)}s\right).
\end{equation}
This limit is $1-\alpha$ if $\delta(h)=0$, in particular on $V$, and is smaller otherwise. If $\delta(h)\le0$, the lower confidence bound
\begin{equation}\label{eq:ci-lower}
\underline\Psi_n=T_n-z_{1-\alpha}\frac{\hat s_n}{\sqrt n}
\end{equation}
satisfies
$\Pn h\{\Psi(P_{n,h})\ge\underline\Psi_n\}\to
\Phi(z_{1-\alpha}-\delta(h)/s)\ge1-\alpha$.
\end{corollary}

Corollary~\ref{cor:ci-local} translates the efficient influence function and directional derivative into coverage statements. For $V=\Hk$ under \assS, \cref{prop:blind} gives $\delta(h)=-\gamma(h_\perp)$. For conditional maxima with nonnegative optimal weights, this drift is nonpositive and the lower bound remains valid along every fixed local path, including paths that break a tie. Signed optimal weights need not have this property. These are assertions for each fixed path; uniform coverage over a set of paths additionally requires uniform approximation. The condition $V_V>0$ also matters: convergence to a degenerate limit does not justify studentization when the variance bound is zero. The proof is in \cref{app:ci-proofs}.

For pooling in \eqref{eq:intro-experiment}, the error is $G-\abs d$. The Gaussian interval $[U-z_{1-\alpha/2}/\sqrt2,U+z_{1-\alpha/2}/\sqrt2]$ has coverage $\Phi(z_{1-\alpha/2}+\sqrt2\abs d)-\Phi(-z_{1-\alpha/2}+\sqrt2\abs d)$, which decreases from $1-\alpha$ at $d=0$ to zero as $\abs d\to\infty$. Its lower endpoint with $z_{1-\alpha}$ instead has coverage at least $1-\alpha$ for every $c,d$. Thus the direction of inference matters even for the estimator attaining the Gaussian bound.

Intervals for an estimator retaining a convolution residual depend on the entire error law. In the product experiment of \cref{prop:split}, consider a rule formed by the efficient linear statistic plus $r(Z_\perp)$. Write $R_h=r(Z_\perp)-\gamma(h_\perp)$, independent of $G\sim N(0,\Vk)$. For a specified set $B$ of orthogonal parameters, a common half-length can be calibrated by
\begin{equation}\label{eq:ci-general-critical}
c_\alpha(B;r)=\inf\{c\ge0:\sup_{h_\perp\in B}P_{h_\perp}(\abs{G+R_h}>c)\le\alpha\}.
\end{equation}
This expression allows non-Gaussian residuals and uses the tail loss of \cref{thm:lam-conv}. It is a calibration problem in the stated limit experiment; the convolution theorem alone neither identifies all local residual laws nor provides consistent estimators of their quantiles. For plug-in estimators, the derivative-based resampling method of \citet{FangSantos2019} supplies quantile estimates for the baseline limiting error law under its conditions. Its pointwise consistency and additional local size results must be distinguished from uniform coverage over a statistical model.

For $N(b,s^2)$ errors with $s>0$ and $\abs b\le b_0$, the corresponding common half-length is $s\,\operatorname{cv}_\alpha(b_0/s)$, where $\operatorname{cv}_\alpha(t)$ is the $1-\alpha$ quantile of $\abs{N(t,1)}$ \citep{ArmstrongKolesar2021}. Pooling with $\abs d\le\rho$ gives $s=1/\sqrt2$ and $b_0=\rho$. The neighborhood is part of the coverage requirement; its local parameter cannot generally be consistently estimated. This calculation illustrates how the derivative determines a bias allowance. Uniformly feasible calibration and optimization over interval procedures are further questions, distinct from minimizing variance or squared error.

\section{Applications}\label{sec:examples}

The examples below derive efficient influence functions and construct estimators that attain the corresponding convolution bounds. They also calculate local bias and residual risk outside the regularity subspaces. A general conditional-maximum formula covers optimal treatment values, instrumental-variable bounds and calibration error: these examples respectively illustrate unequal conditional variances, correlated branches and efficiency of an existing estimator. Sequential treatment requires comparing complete optimal policies. Marginal programme values then give explicit residual distributions for pooling, plug-in estimation, smoothing and sample splitting.

\subsection{Conditional maxima and their efficiency bounds}\label{ex:conditional-max}
A conditional maximum depends on both its branch values and the covariate distribution. Where branches tie, their gradients may have different conditional variances and may be correlated. The efficient influence function combines these gradients using their conditional covariance matrix. Many individualized decisions and partially identified parameters have the form
\[
\Psi(Q)=\E_Q q_Q(X),\qquad q_Q(x)=\max_{j\le J}b_j(Q,x),
\]
where $X$ may be continuously distributed. Write $b_j(x)=b_j(P,x)$, $q=q_P$, and $\mathcal J^*(x)=\arg\max_j b_j(x)$. The conditions below concern the observed-data model: causal or instrumental-variable assumptions are used to identify its branches.

Suppose the branches admit uniformly bounded versions over the model and that $r_j\in L_2(P)$ satisfy $\E[r_j\mid X]=0$. Assume, for every DQM path $P_t$ with score $h\in\T$,
\begin{equation}\label{eq:cm-path}
\left\|\frac{b_j(P_t,\cdot)-b_j}{t}-\dot b_{j,h}\right\|_{L_1(P_X)}\longrightarrow0,
\qquad \dot b_{j,h}(X):=\E[r_jh\mid X],\quad j\le J.
\end{equation}
This is an expansion along statistical paths, including paths that change the marginal law of $X$. A sufficient observed-data model takes each branch to be a known bounded affine combination of conditional moment ratios, with bounded observed variables and denominators bounded away from zero throughout the model; \cref{lem:cm-paths} verifies \eqref{eq:cm-path} along every DQM path. No separation between the best and second-best branches is needed for this derivative calculation.

Choose the smallest active index $j_0(x)$ and let $\lambda_0=q(X)-\Psi(P)+r_{j_0(X)}$. For a measurable weight vector $w(X)$, ``active affine weights'' means $w_j(X)=0$ off $\mathcal J^*(X)$ and $\sum_jw_j(X)=1$; the weights may be negative. Write
\[
G_{jk}(X)=\E[r_jr_k\mid X],\qquad
\phi_w=q(X)-\Psi(P)+\sum_jw_j(X)r_j.
\]
All statements about $\phi_w$ below require $\phi_w\in L_2(P)$. The proposition identifies the conditional maximum's derivative, the directions along which regularity is possible, and the resulting Gaussian bound and residual law. For the unrestricted tangent space, its conditional covariance projection agrees with \citet[Theorem~6.1]{XuGuoRALU}; the path expansion here identifies the associated subspace as the target's core.

\begin{proposition}[conditional maxima]\label{prop:cm}
Suppose \assT and \eqref{eq:cm-path} hold. Then
\begin{equation}\label{eq:cm-derivative}
\gamma(h)=\E[(q(X)-\Psi(P))h]
+\E\!\left[\max_{j\in\mathcal J^*(X)}\dot b_{j,h}(X)\right].
\end{equation}
Its unique maximal linearity subspace and efficient gradient are
\begin{equation}\label{eq:cm-core}
\begin{split}
\Hk&=\{h\in\T:\dot b_{j,h}(X)=\dot b_{k,h}(X)
\text{ for all }j,k\in\mathcal J^*(X),\ P\text{-a.s.}\},\\
\psik&=\Pi_{\Hk}\lambda_0,\qquad \Vk=\norm{\psik}^2.
\end{split}
\end{equation}
Every estimator regular along this space has limit law $N(0,\Vk)*M$.

If $\T=L_2^0(P)$, let $w^*(X)$ be a measurable solution of
\begin{equation}\label{eq:cm-qp}
v_*(X)=\min_{\substack{w_j=0\ (j\notin\mathcal J^*(X))\\ \sum_jw_j=1}}w^\top G(X)w.
\end{equation}
Such a solution exists even when $G(X)$ is singular, and its combined residual is unique. In this case
\begin{equation}\label{eq:cm-efficient}
\psik=\phi_{w^*},\qquad
\Vk=\Var(q(X))+\E[v_*(X)].
\end{equation}
For any active affine weights $w$ and an estimator asymptotically linear with influence function $\phi_w$, the estimator is core-regular, and its convolution residual is
\begin{equation}\label{eq:cm-gap}
M=N\!\left(0,\E[(w-w^*)^\top G(X)(w-w^*)]\right).
\end{equation}
It attains the core bound exactly when $\phi_w=\phi_{w^*}$ in $L_2(P)$.
\end{proposition}

\Cref{prop:cm} separates the variation in the conditional target from the noise in estimating its active branches. The latter is minimized using their full conditional covariance, and the convolution conclusion turns this calculation into a lower bound for all core-regular estimators. Consequently, a feasible estimator with a proved asymptotically linear expansion with $\phi_{w^*}$ attains that bound. The remaining question is how it behaves when a local alternative changes the relative values of the tied branches. Its local drift follows from \cref{thm:al}:
\begin{equation}\label{eq:cm-drift}
\E\!\left[\sum_jw_j^*(X)\dot b_{j,h}(X)
-\max_{j\in\mathcal J^*(X)}\dot b_{j,h}(X)\right].
\end{equation}
This drift compares the first-order change estimated by the weighted average with the change in the best branch. It need not have a fixed sign when optimal affine weights are negative. The same influence function therefore describes both attainable precision along the core and the response to alternatives outside it. To use this conclusion for a particular estimator, its learning conditions must establish the asymptotically linear expansion.

The conditional affine scores of \citet{Whitehouse2025} take $b_j(P,X)=s_j(X;g_j)$, where $g_j(W)=\E[U_j\mid W]$ and $X$ is part of $W$. With their conditional Riesz representer $\zeta_j(W)$, the corresponding residual is $r_j=\zeta_j(W)^\top\{U_j-g_j(W)\}$, which has conditional mean zero given $X$. Under the conditions of their cross-fitted Corollary~C.2 (the extension of Theorem~4.2 described in Remark~4.4), the estimator has a full-sample asymptotically linear expansion on its $n$ observations with the active uniform-weight influence function $\phi_u$, where $u_j(X)=\ind\{j\in\mathcal J^*(X)\}/|\mathcal J^*(X)|$.

In a model satisfying the path conditions above, \eqref{eq:cm-gap} gives this estimator's exact excess variance on the same sample-size scale. Equality holds precisely when the active block of $G(X)$ has equal row sums almost surely. Thus the covariance calculation determines whether an existing estimator attains the bound and quantifies its excess variance otherwise. The following examples illustrate both possibilities. Conditional minima are obtained by applying the proposition to the negatives of all branches; proofs of the general calculation are in \cref{app:cm}.

Structural restrictions on the statistical model can change its efficiency bound. The projection in \cref{prop:cm} then uses the tangent space of the restricted model. For example, suppose the treatment contrasts satisfy
\[
\mu_a(x)-\mu_{a_0}(x)=\theta(P)^\top f_a(x),\qquad f_{a_0}(x)=0,
\]
with known bounded features and a pathwise differentiable finite-dimensional parameter $\theta$. This is the linear-blip restriction used by \citet[Section 3.1]{Whitehouse2025}. If $\dot\theta_h$ is its path derivative, the core in the restricted model is
\begin{equation}\label{eq:cm-blip-core}
\Hk=\{h\in\T:(f_a(X)-f_b(X))^\top\dot\theta_h=0
\text{ for all active }a,b,\ \text{a.s.}\}.
\end{equation}
For any fixed active selector $d$, let $\lambda_d\in\T$ be the gradient of its rule value in that model. Then $\psik=\Pi_{\Hk}\lambda_d$. Indeed, differentiating the displayed contrast identity yields exactly the conditional tie restrictions, while the rule-value derivative agrees with the target derivative on them. The conditional quadratic formula for the unrestricted model need not give this projection. Likewise, an asymptotically linear structural estimator determines a target influence function whose variance can be compared with $\norm{\Pi_{\Hk}\lambda_d}^2$ once its expansion has been verified.

\subsection{The value of an optimal individualized treatment rule}\label{ex:otr}
The precision of an estimated treatment mean depends on the treatment's assignment probability and conditional outcome variance. When several treatments have the same optimal conditional mean, efficient value estimation can pool their observations. An individualized rule can assign different treatments to different patients, and its optimal value is
\[
\Psi(P)=\E\{q(X)\},\qquad q(x)=\max_{a\le K}\mu_a(x),
\]
where $\mu_a(x)=\E[Y\mid X=x,A=a]$. Under consistency, unconfoundedness, and overlap, this is the mean outcome under an optimal individualized rule. The covariate $X$ may take values in a general standard Borel space.

To verify directional pathwise differentiability along every statistical path, consider the model with a common outcome bound $\abs Y\le B_0$ and $\pi_{a,Q}(X)\ge\epsilon$ at every law $Q$, evaluated at $P$ with $\pi_a(X)\ge\epsilon+\eta$ for some $\eta>0$. The argument in \cref{app:policy-model} gives \assT with $\T=L_2^0(P)$, and \cref{lem:cm-paths} supplies the all-path conditional expansions.

Let $\mathcal A^*(x)=\arg\max_a\mu_a(x)$, and write
\[
R_a(O)=\frac{\ind\{A=a\}}{\pi_a(X)}\{Y-\mu_a(X)\},\qquad
c_a(x)=\frac{\sigma_a^2(x)}{\pi_a(x)},\qquad
\dot\mu_{a,h}(x)=\E\{R_a h\mid X=x\},
\]
where $\sigma_a^2(x)=\Var(Y\mid X=x,A=a)>0$ almost surely. The conditional residual covariance is diagonal, with entries $c_a(x)$. Hence the conditional affine minimum of \cref{prop:cm} has the weights
\begin{equation}\label{eq:otr-weights}
w_a^*(x)=\frac{c_a(x)^{-1}}{\sum_{b\in\mathcal A^*(x)}c_b(x)^{-1}}
\quad(a\in\mathcal A^*(x)),\qquad w_a^*(x)=0\quad(a\notin\mathcal A^*(x)).
\end{equation}
For general $X$, the core consists of scores for which the $\dot\mu_{a,h}(X)$ are equal across $a\in\mathcal A^*(X)$ almost surely, and the optimal influence function and bound are
\begin{equation}\label{eq:otr-general}
\begin{split}
\psik&=q(X)-\Psi(P)+\sum_a w_a^*(X)R_a,\\
\Vk&=\Var\{q(X)\}+\E\left[\left\{\sum_{a\in\mathcal A^*(X)}c_a(X)^{-1}\right\}^{-1}\right].
\end{split}
\end{equation}
The influence function uses the conditional maximum for the covariate contribution and pools outcome residuals according to their precision. Both the assignment probability and the conditional outcome variance enter that precision. No finite-support condition is needed for this calculation, and the weights are positive and sum to one even if the conditional variances approach zero.

The softmax estimator of \citet[Theorem 3.3 and Corollary C.2]{Whitehouse2025}, under its smoothing and nuisance-rate conditions, has an asymptotically linear representation with weights $1/m(X)$ on $\mathcal A^*(X)$, where $m(X)=\abs{\mathcal A^*(X)}$. Here the cross-fitted representation uses all $n$ observations. The two influence functions have the exact variance difference
\begin{equation}\label{eq:otr-uniform-gap}
V_{\rm unif}-\Vk
=\E\left[\sum_{a\in\mathcal A^*(X)}c_a(X)
\{m(X)^{-1}-w_a^*(X)\}^{2}\right].
\end{equation}
Equation~\eqref{eq:otr-uniform-gap} shows that the softmax influence function attains the core bound exactly when $c_a(X)$ is constant across the active treatments almost surely. This follows by conditional orthogonality: $c_a(X)w_a^*(X)$ is constant over the active set, so the efficient residual is orthogonal to every change of normalized weights. The comparison concerns two procedures in the same core-regular class; their behavior when treatment values separate requires the local drift calculation below.

For binary treatment, \citet[Supplement A.12]{XuGuo2025} derive the heteroscedastic optimal weights and establish optimality within their robust asymptotically linear class, with feasible attainment by adaptive smoothing. The calculation here identifies the directions along which that influence function is efficient, compares its variance with the bound when regularity is required throughout a specified treatment's optimality cone, and determines its local bias when the tied treatment values separate.

We next give a feasible construction and an explicit variance comparison under one-sided regularity for finitely many prespecified strata, $X\in\{1,\ldots,J\}$. Retain the model above and suppose $p_x:=P(X=x)>0$ for every stratum. The finite support makes all empirical conditional means root-$n$ consistent and allows consistent recovery of the base law's active sets. In this setting the directional derivative is
\begin{equation}\label{eq:otr-derivative}
\gamma(h)=\E\{(q(X)-\Psi(P))h\}
+\sum_xp_x\max_{a\in\mathcal A^*(x)}\dot\mu_{a,h}(x).
\end{equation}
Indeed, let $\mathcal D$ be the finite set of rules $d:\{1,\ldots,J\}\to\{1,\ldots,K\}$. Their values satisfy
\[
\Psi(Q)=\max_{d\in\mathcal D}\E_Q\{\mu_{d(X),Q}(X)\}.
\]
The active rules choose an element of $\mathcal A^*(x)$ in every stratum. Thus this individualized target is covered by the finite maximum calculation, although the gradients of its many active rules need not have a nonsingular Gram matrix.

Fix an optimal rule $d$, so $d(x)\in\mathcal A^*(x)$ for every $x$. The cone
\[
C_d=\{h\in\T:\dot\mu_{d(x),h}(x)\ge\dot\mu_{a,h}(x)
\text{ for every }x\text{ and }a\in\mathcal A^*(x)\}
\]
consists of directions under which that rule remains optimal to first order.

\begin{corollary}[geometry and bounds for the individualized value]\label{cor:otr-geometry}
Under the model conditions above, the unique maximal linearity subspace is
\begin{equation}\label{eq:otr-core}
\Hk=\{h\in\T:\dot\mu_{a,h}(x)=\dot\mu_{b,h}(x)
\text{ for every }x\text{ and }a,b\in\mathcal A^*(x)\}.
\end{equation}
The efficient influence function relative to $\Hk$ and its variance bound are
\begin{equation}\label{eq:otr-if-bound}
\begin{split}
\psik(O)&=q(X)-\Psi(P)+\sum_aw_a^*(X)R_a(O),\\
\Vk&=\Var\{q(X)\}
+\sum_xp_x\left\{\sum_{a\in\mathcal A^*(x)}\frac{\pi_a(x)}{\sigma_a^2(x)}\right\}^{-1}.
\end{split}
\end{equation}
Every estimator regular along $\Hk$ has a limit law $N(0,\Vk)*M$. The cone $C_d$ spans $\T$, and every estimator regular along $C_d$ has a limit law $N(0,V_d)*M_d$, where
\begin{equation}\label{eq:otr-cone-price}
\begin{split}
V_d&=\Var\{q(X)\}+\sum_xp_xc_{d(x)}(x),\\
V_d-\Vk&=\sum_xp_x\left[c_{d(x)}(x)
-\left\{\sum_{a\in\mathcal A^*(x)}c_a(x)^{-1}\right\}^{-1}\right].
\end{split}
\end{equation}
The variance increase $V_d-\Vk$ is strictly positive if at least one stratum has more than one best treatment, and is zero otherwise.
\end{corollary}

\Cref{cor:otr-geometry} compares efficiency under two regularity requirements in the same observed-data model. Along the core, all optimal treatments move together within each stratum, and their observations can be pooled using \eqref{eq:otr-weights}. Along $C_d$, the designated rule can separate from the other optimal rules, so its own gradient determines the Gaussian component. The difference in \eqref{eq:otr-cone-price} is the resulting variance cost, positive whenever a stratum contains a treatment tie. Both conclusions apply to every estimator satisfying the respective regularity requirement.

An attaining estimator first identifies the tied treatments and then pools their means using estimated precision. The tolerance must retain differences at the estimation-noise scale while excluding every fixed positive gap. Let $\hat p_x$, $\hat\pi_a(x)$, $\hat\mu_a(x)$, and $\hat\sigma_a^2(x)$ be the empirical stratum frequencies, treatment proportions, cell means, and cell variances (using the cell size as divisor). Choose a deterministic tolerance $\tau_n\downarrow0$ with $\sqrt n\tau_n\to\infty$, and set
\begin{equation}\label{eq:otr-estimator}
\begin{split}
\widehat{\mathcal A}(x)&=\{a:\max_b\hat\mu_b(x)-\hat\mu_a(x)\le\tau_n\},\\
\hat w_a(x)&=
\frac{\ind\{a\in\widehat{\mathcal A}(x)\}\hat\pi_a(x)/\hat\sigma_a^2(x)}
{\sum_{b\in\widehat{\mathcal A}(x)}\hat\pi_b(x)/\hat\sigma_b^2(x)},\\
\widehat\Psi_{\mathrm{eff}}&=\sum_x\hat p_x\sum_a\hat w_a(x)\hat\mu_a(x).
\end{split}
\end{equation}
On the event that a treatment cell is empty or an empirical cell variance is zero, define $\widehat\Psi_{\mathrm{eff}}$ to be the overall sample mean instead; this event has probability tending to zero. This convention makes the estimator defined for every sample.

\begin{corollary}[feasible attainment and local drift]\label{cor:otr-attain}
Under the same model conditions and the stated tolerance condition, the estimator in \eqref{eq:otr-estimator} satisfies
\begin{equation}\label{eq:otr-al}
\sqrt n\{\widehat\Psi_{\mathrm{eff}}-\Psi(P)\}
=n^{-1/2}\sum_{i=1}^n\psik(O_i)+o_{P^n}(1),
\end{equation}
and is regular and efficient along $\Hk$.
The fixed-rule estimate $\sum_x\hat p_x\hat\mu_{d(x)}(x)$, defined by the same fallback convention on exceptional samples, is regular and attains $V_d$ along $C_d$. For every $h\in\T$ and every DQM path with score $h$,
\begin{equation}\label{eq:otr-drift}
\begin{split}
\sqrt n\{\widehat\Psi_{\mathrm{eff}}-\Psi(P_{n,h})\}
&\wto N(\delta(h),\Vk)\quad\text{under }\Pn h,\\
\delta(h)&=\sum_xp_x\left\{\sum_{a\in\mathcal A^*(x)}w_a^*(x)\dot\mu_{a,h}(x)
-\max_{a\in\mathcal A^*(x)}\dot\mu_{a,h}(x)\right\}\le0.
\end{split}
\end{equation}
Equality $\delta(h)=0$ holds exactly when $h\in\Hk$.
\end{corollary}

\Cref{cor:otr-attain} shows that learning the weights does not add a convolution residual: the empirical procedure has the efficient influence function. The tolerance identifies the fixed law's active sets with probability tending to one, and estimating normalized weights creates no first-order error when the corresponding means are equal. For computation of standard errors, put $\hat q_x=\sum_a\hat w_a(x)\hat\mu_a(x)$. A consistent estimate of $\Vk$ is
\begin{equation}\label{eq:otr-variance}
\widehat V_\kappa
=\sum_x\hat p_x(\hat q_x-\widehat\Psi_{\mathrm{eff}})^2
+\sum_x\hat p_x\left\{\sum_{a\in\widehat{\mathcal A}(x)}
\frac{\hat\pi_a(x)}{\hat\sigma_a^2(x)}\right\}^{-1},
\end{equation}
with the overall empirical outcome variance used on the exceptional event specified above. These assertions are pointwise at each law in the stated model; under local paths, the limit is centered exactly on $\Hk$. With $\hat s_n^2=\widehat V_\kappa$, \cref{cor:ci-local} supplies Wald intervals along the core and a lower confidence bound along every fixed local path, since \eqref{eq:otr-drift} is nonpositive. The ordinary two-sided interval has the local coverage in \eqref{eq:ci-local-coverage}; learning the active treatments does not remove that drift.

The two bounds now have a direct estimation interpretation. Consider one stratum and two tied treatments with conditional variances one and assignment probabilities $0.2$ and $0.8$. The efficient weights are $(0.2,0.8)$ and $\Vk=1$. Equal weighting has variance $25/16$, whereas requiring regularity whenever treatment $2$ remains optimal has bound $V_d=5/4$, attained by its sample mean. The reduction from $25/16$ to $1$ compares two estimators regular along the same core; the increase from $1$ to $5/4$ is the cost of the stronger, one-sided requirement. Along this cone, the specified treatment determines the target even as its value separates from the other treatment, so pooling both no longer estimates a common local mean. Under a local alternative that breaks the tie, \eqref{eq:otr-drift} gives the resulting negative centered shift of efficient pooling. Proofs of both corollaries are given in \cref{app:otr}.

\subsection{Conditional instrumental-variable bounds}\label{ex:bp}
An instrument can bound a treatment effect when unmeasured confounding prevents point identification. The efficiency issue differs from treatment pooling because several expressions for a bound use the same observed cell probabilities. Their estimation errors are correlated, so efficient estimation requires their joint covariance matrix. Let $Z=(X,V,D,Y)$, where the instrument $V$, treatment $D$ and outcome $Y$ are binary, and put $e_v(x)=P(V=v\mid X=x)$ and $p_{yd\mid v}(x)=P(Y=y,D=d\mid X=x,V=v)$. For readability suppress $x$ below. The conditional Balke--Pearl lower bound is the maximum of eight affine branches \citep{Levis2025}:
\begin{equation}\label{eq:bp-branches}
\begin{aligned}
b_1&=p_{11\mid1}+p_{00\mid0}-1,&
b_2&=p_{11\mid0}+p_{00\mid1}-1,\\
b_3&=-p_{01\mid1}-p_{10\mid1},&
b_4&=-p_{01\mid0}-p_{10\mid0},\\
b_5&=p_{11\mid0}-p_{11\mid1}-p_{10\mid1}-p_{01\mid0}-p_{10\mid0},\\
b_6&=p_{11\mid1}-p_{11\mid0}-p_{10\mid0}-p_{01\mid1}-p_{10\mid1},\\
b_7&=p_{00\mid1}-p_{01\mid1}-p_{10\mid1}-p_{01\mid0}-p_{00\mid0},\\
b_8&=p_{00\mid0}-p_{01\mid0}-p_{10\mid0}-p_{01\mid1}-p_{00\mid1}.
\end{aligned}
\end{equation}
Write $q(x)=\max_jb_j(x)$, $\Psi_L=\E q(X)$, and $\mathcal J^*(x)=\arg\max_jb_j(x)$. Under consistency, exclusion and conditional independence of the instrument from the potential treatment and outcome variables, $\Psi_L$ is a lower bound on the average treatment effect. We estimate this observed-law functional in the nonparametric model with a common instrument-overlap bound $e_{v,Q}\ge\epsilon$, at a law with $e_{v,P}\ge\epsilon+\eta$ for some $\eta>0$. Its tangent space is $L_2^0(P)$. Thus the efficiency comparison uses observed-law perturbations; if the statistical model also imposes instrumental inequalities, the projection must instead use its own feasible tangent space.

Let $a_{j,ydv}$ denote the coefficient of $p_{yd\mid v}$ in \eqref{eq:bp-branches}. Define
\begin{equation}\label{eq:bp-residual}
r_j(Z)=\sum_{y,d,v}a_{j,ydv}\frac{\ind\{V=v\}}{e_v(X)}
 \{\ind\{Y=y,D=d\}-p_{yd\mid v}(X)\}.
\end{equation}
For $a_{j,v}=(a_{j,00v},a_{j,01v},a_{j,10v},a_{j,11v})^\top$ and $p_v$ in the same order, the conditional Gram matrix is
\begin{equation}\label{eq:bp-gram}
G_{jk}(x)=\sum_{v=0}^1\frac{a_{j,v}^{\top}
 \{\operatorname{diag}(p_v(x))-p_v(x)p_v(x)^\top\}a_{k,v}}{e_v(x)}.
\end{equation}
Unlike treatment-specific outcome residuals, these branch residuals share multinomial cells and are correlated. Their full covariance enters the efficiency calculation.

\begin{corollary}[efficiency for the bound]\label{cor:bp}
In the stated observed-law model, the unique maximal linearity subspace is
\[
\Hk=\{h:\E[(r_j-r_k)h\mid X]=0
 \text{ for all }j,k\in\mathcal J^*(X),\ P\text{-a.s.}\}.
\]
Let $G_*(x)$ be the active submatrix. Choose
\[
w^*(x)\in\arg\min_{\one^\top w=1}w^\top G_*(x)w
\]
and extend it by zero to inactive branches. Then
\begin{equation}\label{eq:bp-bound}
\psik=q(X)-\Psi_L+\sum_jw_j^*(X)r_j(Z),\qquad
\Vk=\Var(q(X))+\E[w^{*\top}Gw^*].
\end{equation}
Every core-regular estimator has limit $N(0,\Vk)*M$. The minimization permits signed weights and singular $G_*$; the minimizing residual is unique even when its weights are not.

If $X$ has a finite prespecified support and every $(X,V,D,Y)$ cell has positive probability, the empirical construction below is asymptotically linear with influence function \eqref{eq:bp-bound}, is core-efficient, and consistently estimates $\Vk$.
\end{corollary}

\Cref{cor:bp} shows that the variance bound for the endpoint depends on the covariance of its active branch gradients. Signed weights can exploit their covariance, and a singular Gram matrix creates ambiguity in the weights without creating ambiguity in the efficient residual. The resulting Gaussian variance is a bound for every core-regular estimator.

A feasible construction estimates the common active value after learning which branches attain it. Use empirical conditional cell probabilities and the estimated branches $\hat b_j(x)$. Select
$\widehat{\mathcal J}(x)=\{j:\max_k\hat b_k(x)-\hat b_j(x)\le\tau_n\}$, where $\tau_n\downarrow0$ and $\sqrt n\tau_n\to\infty$, and form $\hat G$ from \eqref{eq:bp-gram}. Among its affine minimizers on $\widehat{\mathcal J}(x)$, take the one with smallest Euclidean norm, $\hat w(x)$. With $\hat p_x$ the empirical stratum proportion, set
\begin{equation}\label{eq:bp-estimator}
\begin{split}
\hat q_w(x)&=\sum_j\hat w_j(x)\hat b_j(x),\qquad
\hat\Psi_L=\sum_x\hat p_x\hat q_w(x),\\
\hat V_L&=\sum_x\hat p_x\bigl[\{\hat q_w(x)-\hat\Psi_L\}^2+
 \hat w(x)^\top\hat G(x)\hat w(x)\bigr].
\end{split}
\end{equation}
Set both estimates to zero if any categorical cell is empty, an event whose probability tends to zero. Full cell support makes the null space of each active Gram matrix depend only on the known branch coefficients. Consequently the selected optimizer is continuous in the cell probabilities, including when the Gram matrix is singular; \cref{app:bp-cal} proves this and the resulting attainment.

Under their smoothing and nuisance-estimation conditions, the cross-fitted extension of \citet[Corollaries 4.5 and C.2]{Whitehouse2025} has an asymptotically linear expansion using all observations, with uniform active weights $u_j(x)=\ind\{j\in\mathcal J^*(x)\}/|\mathcal J^*(x)|$. Its excess variance over \eqref{eq:bp-bound} is exactly
\begin{equation}\label{eq:bp-uniform-gap}
V_{\mathrm{unif}}-\Vk
=\E[(u-w^*)^\top G(u-w^*)]\ge0.
\end{equation}
Equation~\eqref{eq:bp-uniform-gap} expresses the entire loss through the covariance of the active branches. Equal weighting attains the bound when its combined residual is the efficient one; otherwise the displayed difference is the variance of an independent Gaussian residual. When $\Vk>0$, the attaining estimator and $\hat V_L$ give pointwise Wald inference and the local coverage formula in \cref{cor:ci-local}. Since the optimal weights can be signed, a conservative lower bound outside the core requires checking the drift sign; it does not follow from the maximum representation alone. The same argument treats the upper endpoint by replacing $Y$ with $1-Y$ and changing the sign of the resulting lower endpoint.

\subsection{Calibration error of a fixed predictor}\label{ex:calibration}
The covariance calculation can also explain why an existing estimator is already optimal. Calibration error is such an example: at a calibrated prediction, the two active branch residuals cancel exactly under equal weighting. Let $S=\theta(O)$ be the output of a fixed predictor, let $m(s)=\E[Y\mid S=s]$, and consider
$\Psi_{\mathrm{cal}}=\E|m(S)-S|$. Work in the nonparametric model for $(O,Y)$ with common bounds on $|Y|$ and $|\theta(O)|$; the distribution of $S$ can be continuous. The two branches are $b_+(s)=m(s)-s$ and $b_-(s)=s-m(s)$, with residuals $r_\pm=\pm\{Y-m(S)\}$. At a calibrated value $m(s)=s$, the affine average of the two residuals is zero. Away from calibration there is only one active branch.

\begin{corollary}[efficient calibration error]\label{cor:calibration}
In this model, the unique maximal linearity subspace and its efficient influence function are
\begin{equation}\label{eq:cal-core-if}
\begin{split}
\Hk&=\{h:\ind\{m(S)=S\}\E[(Y-m(S))h\mid S]=0\},\\
\psik&=|m(S)-S|-\Psi_{\mathrm{cal}}
 +\operatorname{sgn}\{m(S)-S\}\{Y-m(S)\},\\
\Vk&=\Var(|m(S)-S|)
 +\E[\ind\{m(S)\ne S\}\Var(Y\mid S)],
\end{split}
\end{equation}
where $\operatorname{sgn}(0)=0$. Every core-regular estimator has limit $N(0,\Vk)*M$.
If $S$ has finite support with positive stratum probabilities, let $\hat m(s)$ and $\hat p_s$ be the empirical stratum mean and proportion, setting $\hat m(s)=s$ in any unobserved stratum. For $\tau_n\downarrow0$ and $\sqrt n\tau_n\to\infty$, the estimator
\begin{equation}\label{eq:cal-estimator}
\hat\Psi_{\mathrm{cal}}
=\sum_s\hat p_s|\hat m(s)-s|\ind\{|\hat m(s)-s|>\tau_n\}
\end{equation}
is asymptotically linear with influence function \eqref{eq:cal-core-if} and attains $\Vk$ along the core. The sample variance of
$\operatorname{sgn}\{\hat m(S_i)-S_i\}\ind\{|\hat m(S_i)-S_i|>\tau_n\}(Y_i-S_i)$ consistently estimates $\Vk$.
\end{corollary}

\Cref{cor:calibration} establishes the optimality of equal weighting at calibration. The influence function in \eqref{eq:cal-core-if} is exactly the limit influence function of \citet[Corollary 4.8]{Whitehouse2025}. Under their smoothing, gap and nuisance-estimation conditions, their cross-fitted extension in Corollary C.2 therefore furnishes a core-efficient procedure using all observations also with general covariates. The negative correlation between the two branches eliminates their conditional residual instead of producing a gain from unequal weights.

At perfect calibration, $\Vk=0$: the Gaussian component vanishes, but a perturbation with $\dot m_h(s)=\E[(Y-m(S))h\mid S=s]$ has derivative $\E|\dot m_h(S)|$, which can be positive outside the core. Thus the problem of detecting a local departure from calibration remains even when the core bound is zero. When $\Vk>0$, \cref{cor:ci-local} applies with the stated variance estimator. At $\Vk=0$, a confidence procedure must instead address the remaining nonlinear estimation problem; the zero core bound does not give a zero-length interval with uniform coverage. Proofs of both applications are given in \cref{app:bp-cal}.

\subsection{Sequential treatment and histories reached by optimal policies}\label{ex:dynamic}
Sequential treatment introduces a further issue: changing a later decision matters only on histories reached by an optimal earlier decision. Let $H_t$ be the complete history before treatment $A_t\in\{1,\ldots,K_t\}$, $t=1,\ldots,T$, and let $Y$ be the terminal outcome. Histories take values in standard Borel spaces. Under consistency and sequential unconfoundedness, the value of a deterministic policy $d=(d_1,\ldots,d_T)$ is identified by replacing each treatment assignment with $A_t=d_t(H_t)$ in the observed transition laws. Write this intervention law as $P^d$ and its value as $\Psi_d(P)=\E_{P^d}Y$. The target is $\Psi(P)=\sup_d\Psi_d(P)$.

Consider the unrestricted observed-data model with $\abs Y\le B_0$ and treatment probabilities $\pi_{t,a,Q}(H_t)\ge\epsilon$ at every law $Q$. At $P$, require the uniform slack $\pi_{t,a}(H_t)\ge\epsilon+\eta$ for some $\eta>0$. The tangent space is $L_2^0(P)$ and \assT holds. General histories are permitted; finite histories will be needed only for the feasible construction below.

Put $V_{T+1}=Y$ and define the optimal continuation values backwards:
\[
Q_t(H_t,a)=\E[V_{t+1}\mid H_t,A_t=a],\qquad
V_t(H_t)=\max_a Q_t(H_t,a),\qquad
\mathcal A_t^*(H_t)=\arg\max_a Q_t(H_t,a).
\]
Thus $\Psi(P)=\E V_1(H_1)$. The derivative also has a backward recursion. For a score $h\in L_2^0(P)$, start with $u_{T+1,h}=0$ and set
\begin{equation}\label{eq:dynamic-recursion}
\begin{split}
k_{t,h}(H_t,a)&=\E[(V_{t+1}-Q_t(H_t,a))h\mid H_t,A_t=a],\\
u_{t,h}(H_t)&=\max_{a\in\mathcal A_t^*(H_t)}
\{k_{t,h}(H_t,a)+\E[u_{t+1,h}\mid H_t,A_t=a]\},\\
\gamma(h)&=\E[(V_1-\Psi(P))h]+\E u_{1,h}.
\end{split}
\end{equation}
The first term inside the maximum differentiates the transition law; the second propagates the change in the continuation value. Keeping both terms accounts for changes in the covariate distribution between decisions.

Let $\mathcal D^*(P)=\{d:\Psi_d(P)=\Psi(P)\}$ be the set of complete optimal policies. For such a policy define
\begin{equation}\label{eq:dynamic-gradient}
W_t^d=\prod_{s=1}^t\frac{\ind\{A_s=d_s(H_s)\}}{\pi_{s,A_s}(H_s)},\qquad
\lambda_d=V_1-\Psi(P)+\sum_{t=1}^TW_t^d\{V_{t+1}-Q_t(H_t,A_t)\}.
\end{equation}
All these gradients are bounded, including when conditional outcome variances differ across histories and treatments.

\begin{corollary}[recursive derivative and dynamic efficiency]\label{cor:dynamic}
In the stated model, \eqref{eq:dynamic-recursion} is the directional derivative along every DQM path, and
\[
\gamma(h)=\max_{d\in\mathcal D^*(P)}\ip{\lambda_d}h.
\]
The unique maximal linearity subspace and its efficient influence function satisfy
\begin{equation}\label{eq:dynamic-core}
\begin{split}
\Hk&=\{h:\ip{\lambda_d-\lambda_e}h=0\text{ for all }d,e\in\mathcal D^*(P)\},\\
\psik&=\underset{\psi\in\claff\{\lambda_d:d\in\mathcal D^*(P)\}}{\arg\min}\norm\psi,
\qquad \Vk=\norm\psik^2.
\end{split}
\end{equation}
Every core-regular estimator has a limit law $N(0,\Vk)*M$. For a specified optimal policy $d$, let $C_d=\{h:\ip{\lambda_d}h=\gamma(h)\}$. The convolution bound along this cone is $\norm{\Pi_{\clspan C_d}\lambda_d}^2$; in the finite-history model it equals $\norm{\lambda_d}^2$ whenever $\lambda_d$ is a vertex of the hull of distinct active gradients.
\end{corollary}

\Cref{cor:dynamic} expresses the core through the gradients of complete optimal policy values. The equality conditions in \eqref{eq:dynamic-core} require their derivatives to agree. If a history is reached by none of them, changing a treatment only there changes none of the active gradients and imposes no core restriction. For example, with two decisions, suppose the first treatment has a unique optimum and two second-stage treatments tie only on the branch following the other first treatment. Those second-stage ties alone do not make the optimal value directionally nonlinear. Requiring equality of their local derivatives would impose unnecessary invariance restrictions.

When the history tree is finite, an attaining estimator first estimates all policy values, retains policies whose estimated values are within a tolerance of the maximum, and pools their estimates using their joint covariance. Suppose the observational history and treatment cells have positive probability. Estimate the value and influence function of every fixed policy by the empirical transition probabilities and terminal cell means. Retain policies whose estimated value is within $\tau_n$ of the largest, where $\tau_n\downarrow0$ and $\sqrt n\tau_n\to\infty$. On these policies let $\widehat\Sigma$ be the empirical Gram matrix of the estimated influence functions and use the normalized affine weights
\begin{equation}\label{eq:dynamic-pooling}
\hat w=\underset{\one^\top w=1}{\arg\min}
\{w^\top\widehat\Sigma w+\rho_n\norm w^2\},\qquad
\widehat\Psi_{\mathrm{dyn}}=\sum_d\hat w_d\widehat\Psi_d,
\quad \rho_n\downarrow0,\quad\sqrt n\rho_n\to\infty.
\end{equation}
The regularization selects stable weights even when different policies have linearly dependent gradients; it disappears from the limiting efficiency bound. Then
\[
\sqrt n\{\widehat\Psi_{\mathrm{dyn}}-\Psi(P)\}
=n^{-1/2}\sum_{i=1}^n\psik(O_i)+o_{P^n}(1),\qquad
\hat w^\top\widehat\Sigma\hat w\ \longrightarrow_P\ \Vk.
\]
The estimator is regular and efficient along $\Hk$; along any DQM local alternative its centered limit is $N(\ip\psik h-\gamma(h),\Vk)$. Thus pooling removes the Gaussian residual while retaining the explicit local drift when competing policy values separate. The construction establishes feasible attainment on a finite history tree; enumerating its policies may be expensive when the tree is large. When $\Vk>0$, the displayed variance estimate and \cref{cor:ci-local} also give Wald inference along the core. A lower confidence bound outside the core requires checking that the local drift is nonpositive; affine pooling alone does not ensure this. Its full specification and proof, together with the recursive result, are in \cref{app:dynamic}.

\citet[Supplement D]{XuGuo2025} establish multistage heteroscedastic optimality within their robust asymptotically linear class and give feasible estimators. \citet[Appendix D]{Whitehouse2025} develop softmax inference for two-stage treatment, including a model with restricted blip functions. The present calculation identifies which histories and policy contrasts determine the Gaussian component and the remaining nonlinear error. Under additional structural restrictions, the policy gradients must first be projected onto that model's tangent space before applying \eqref{eq:dynamic-core}.

\subsection{The value of the best programme under unconfoundedness}\label{ex:policy}
Selecting one programme for the whole population poses a different estimation problem from assigning treatment individually. Here the target takes the maximum after averaging over covariates. Its finite vector of programme values makes it possible to exhibit the convolution residuals of pooling, plug-in estimation, smoothing and sample splitting, and then compare their local risks. Let $O=(X,A,Y)$, where $A\in\{1,\dots,K\}$ is a programme indicator, $X$ denotes pre-programme covariates in a standard Borel space, and $Y$ is the outcome. Write $Y(a)$ for the potential outcome under programme $a$. Under consistency $Y=Y(A)$, unconfoundedness $Y(a)\indep A\mid X$, and overlap $\pi_a(X):=P(A=a\mid X)\ge\epsilon>0$, the programme value is $\mu_a(P)=\E[Y(a)]=\E[\mu_a(X)]$, where $\mu_a(X):=\E[Y\mid X,A=a]$. The target $\Psi(P)=\max_a\mu_a(P)$ is the value of the best of $K$ prespecified programmes. Assume $\E Y^2<\infty$. For this application take $\T=L_2^0(P)$ and a model satisfying \assT in which, along every DQM path with score $h$,
\begin{equation}\label{eq:policy-pd}
\begin{split}
\mu_a(P_{t,h})&=\mu_a(P)+t\ip{\lambda_a}h+o(t),\\
\lambda_a(O)&=\frac{\ind\{A=a\}}{\pi_a(X)}\{Y-\mu_a(X)\}
                 +\mu_a(X)-\mu_a(P).
\end{split}
\end{equation}
The component efficient influence function has the familiar augmented inverse-probability form; see \citet{Hahn1998} for the underlying treatment-effect efficiency calculation.

These model conditions hold, for example, in the nonparametric model with a common outcome bound $\abs Y\le B_0$ and overlap $\pi_{a,Q}(X)\ge\epsilon$ for every law $Q$, evaluated at a law $P$ satisfying $\pi_a(X)\ge\epsilon+\eta$ for some $\eta>0$. Bounded outcomes make the remainder in the component expansion quadratic in Hellinger distance, and the overlap margin permits likelihood tilts realizing every $L_2^0(P)$ score. \Cref{app:policy-model} verifies both assertions. More generally, the calculations below require \eqref{eq:policy-pd} and \assT, without a common outcome bound.

The inverse-probability-weighted outcome residuals have conditional mean zero given $X$, and their cross-products vanish for distinct programmes. Consequently the Gram matrix is
\[
\begin{split}
\Sigma_{aa}&=\E\Big[\frac{\sigma_a^2(X)}{\pi_a(X)}\Big]+\Var\big(\mu_a(X)\big),\\
\Sigma_{ab}&=\Cov\big(\mu_a(X),\mu_b(X)\big)\qquad(a\ne b),
\end{split}
\]
where $\sigma_a^2(X):=\Var(Y\mid X,A=a)$. The cross-covariance therefore comes entirely from the conditional mean functions. Let $\mathcal A^*$ be the programmes attaining the maximum at $P$, write $m=\abs{\mathcal A^*}\ge2$, and assume their gradients are not all identical. Then $\gamma(h)=\max_{a\in\mathcal A^*}\ip{\lambda_a}h$ is genuinely nonlinear and sublinear, with $\LamP=\conv\{\lambda_a:a\in\mathcal A^*\}$. We allow the active Gram matrix $\Sigma_*:=\Sigma_{\mathcal A^*}$ to be singular.

For the estimator comparisons, suppose the component estimates are jointly asymptotically linear at $P$:
\begin{equation}\label{eq:policy-al}
\sqrt n\{\hat\mu-\mu(P)\}
=n^{-1/2}\sum_{i=1}^n\lambda(O_i)+o_{P^n}(1).
\end{equation}
This is an additional condition on the component estimators, such as suitably fitted augmented inverse-probability estimates. These comparisons use the active set and pooling weights at $P$. Under $\Pn h$, contiguity and \eqref{eq:policy-pd} give the active-vector limit $Z+\ell(h)$, with $Z\sim N(0,\Sigma_*)$ and $\ell(h)=(\ip{\lambda_a}h)_{a\in\mathcal A^*}$.

To identify the Gaussian component common to these estimators, first consider local changes that move the best programme values together. Their core is
\[
\Hk=\{h\in\T:\ip{\lambda_a-\lambda_b}h=0
                 \text{ for all }a,b\in\mathcal A^*\}.
\]
Thus a local alternative belongs to the core exactly when it moves the tied programme values equally to first order. Regularity along $\Hk$ expresses invariance to perturbations preserving their relative value.

\Cref{thm:kink} now identifies how precisely this common value can be estimated: every estimator regular along $\Hk$ has limit $N(0,\Vk)*M$, where
\[
\Vk=\min_{\one^\top w=1}w^\top\Sigma_*w,
\qquad \psik=\sum_{a\in\mathcal A^*}w_a^*\lambda_a
\]
for any minimizing weights $w^*$. The gradient $\psik$ is unique even when its weight representation is not. Under \eqref{eq:policy-al}, the pool $\sum_{a\in\mathcal A^*}w_a^*\hat\mu_a$ attains this bound. If $\Sigma_*$ is nonsingular, then
\[
w^*=\frac{\Sigma_*^{-1}\one}{\one^\top\Sigma_*^{-1}\one},
\qquad \Vk=\frac1{\one^\top\Sigma_*^{-1}\one};
\]
\cref{thm:kink}(5) gives both singular cases, including $\Vk=0$ when $\one\notin\range(\Sigma_*)$. In general $\Vk\le\Sigma_{aa}$, with equality exactly when $\lambda_a=\psik$. Under nonsingularity this is equivalent to $w^*=e_a$. For independent active components with positive variances, the inequality is strict for every $a$ (\cref{rem:dividend}).

A different Gaussian bound results if the estimator must also accommodate changes under which a particular programme separates from the others and remains best. For $a\in\mathcal A^*$, define
\[
C_a=\{h\in\T:\ip{\lambda_a-\lambda_b}h\ge0
                       \text{ for all }b\in\mathcal A^*\}.
\]
These are the directions in which programme $a$ remains best to first order. The Gaussian factor for regular estimation along $C_a$ has variance
\[
V_{C_a}=\norm{\Pi_{\clspan C_a}\lambda_a}^2.
\]
If $\lambda_a$ is a vertex of $\LamP$, with duplicate gradients identified, \cref{cor:onesided} gives $\clspan C_a=\T$ and $V_{C_a}=\Sigma_{aa}$. Then $\hat\mu_a$ attains the bound, the pool is regular along $C_a$ exactly when $\psik=\lambda_a$, and the additional variance is $\Sigma_{aa}-\Vk$. Positive definiteness of $\Sigma_*$ ensures the vertex condition for every active component. Without it, the projection formula remains valid; a nonvertex cone can even coincide with the core, so the increase in its bound need not equal $\Sigma_{aa}-\Vk$.

The same decomposition identifies what pooling misses when programme values separate. Write $h=h_\kappa+h_\perp$ with $h_\kappa\in\Hk$, $h_\perp\in F_P$, and put $\ell(h)=c\one+\delta$, where $\delta=\ell(h_\perp)$ and $c=\ip{\psik}{h_\kappa}$. The pool has limit
\[
N\!\left(-\max_{a\in\mathcal A^*}\delta_a,\Vk\right).
\]
When $\Sigma_*$ is nonsingular, $\delta\in\ell(F_P)$ is equivalent to $\one^\top\Sigma_*^{-1}\delta=0$. The drift is nonpositive for every tie-breaking direction exactly when $\psik\in\LamP$; under nonsingularity this is the simplex condition $\Sigma_*^{-1}\one\ge0$ (\cref{rem:simplex}). Under that condition the pool underestimates the local best-programme value by $\max_a\delta_a$. Outside it the signed drift can be positive.

Other estimators respond to the differences between programme values and therefore retain an additional error component. To isolate it, write
$Z_a=\Delta(\lambda_a)$, $G=\Delta(\psik)$ and
$W_a=\Delta(\lambda_a-\psik)$ for $a\in\mathcal A^*$.
Then $Z=G\one+W$, $G\indep W$, and
$\Cov(W)=\Sigma_*-\Vk\one\one^\top$,
also when the active Gram matrix is singular. The plug-in
$\max_a\hat\mu_a$ is regular along $\Hk$ and has limit
$N(0,\Vk)*\Law(\max_{a\in\mathcal A^*}W_a)$, by the directional delta method for the active maximum and the displayed Gaussian decomposition (cf.\ \cref{prop:plugin}). The residual has positive mean whenever two
active gradients differ, which is precisely the condition for a
nonlinear derivative in this model.

Smoothing uses the same component estimates but averages near-tied values instead of selecting their maximum. The limiting weights depend on how rapidly the smoothing parameter increases relative to the estimation scale. For a concrete marginal-value smoother, set
\[
 S_\beta(x)=\frac{\sum_{a=1}^Kx_a e^{\beta x_a}}
                         {\sum_{a=1}^Ke^{\beta x_a}},
 \qquad T_n^{\rm sm}=S_{\beta_n}(\hat\mu).
\]
Choose $\beta_n/\log n\to\infty$ and $\beta_n/\sqrt n\to0$.
The first condition makes the contribution of programmes with a
fixed positive value gap negligible at the root-$n$ scale;
the second makes the weights on the tied programmes uniform
to first order. Indeed, their weighted value differs from their
uniform average by $O_P(\beta_n/n)$; if inactive programmes are present, with smallest gap
$g_0>0$, their contribution is $O_P(e^{-\beta_n g_0/2})$.
Both terms are $o_P(n^{-1/2})$. Thus $T_n^{\rm sm}$ is asymptotically linear at $P$
with influence function $m^{-1}\sum_{a\in\mathcal A^*}\lambda_a$.
It is regular along $\Hk$ with residual
\[
 N\!\left(0,
 \frac{\one^\top\Sigma_*\one}{m^2}-\Vk\right).
\]
This variance is nonnegative and vanishes exactly when the
average active gradient equals $\psik$. Under a local shift
$c\one+\delta$, its drift is $m^{-1}\sum_{a\in\mathcal A^*}\delta_a-\max_a\delta_a$.
The formula concerns smoothing the vector of marginal programme
values. The softmax approach of \citet{Whitehouse2025} instead
targets $\E\max_a\mu_a(X)$, the value of individualized treatment;
a tie among the marginal values need not be a conditional tie.

Sample splitting changes the asymptotic variance as well as the selection step. Evaluating a selected programme on independent data removes the selection-induced mean shift at a tie, but the evaluation sample contains fewer observations. For an actual one-way sample split, let a fraction
$\tau\in(0,1)$ of the $n$ observations select
$\hat a=\arg\max_a\hat\mu_a^{\rm tr}$ and let the remaining fraction
$q=1-\tau$ estimate its value, $T_n^{\rm sp}=\hat\mu_{\hat a}^{\rm ev}$.
Assume that the two component vectors are asymptotically linear
on their respective subsamples with the displayed influence
functions. Under local mean shifts $\ell(h)=c\one+\delta$, let
$Z^{\rm tr},Z^{\rm ev}$ be independent $N(0,\Sigma_*)$
vectors. With any fixed tie-breaking rule, the centred
root-$n$ limit is
\[
 \begin{split}
 A_\delta&=\arg\max_{a\in\mathcal A^*}
       \{\delta_a+\tau^{-1/2}Z_a^{\rm tr}\},\\
 \sqrt n\{T_n^{\rm sp}-\Psi(P_{n,h})\}
 &\wto q^{-1/2}Z^{\rm ev}_{A_\delta}
             +\delta_{A_\delta}-\max_{b\in\mathcal A^*}\delta_b.
 \end{split}
\]
Identical active gradients can be identified in this limit:
their local shifts and evaluation coordinates coincide.
Along the core all the local shifts are equal, so the limit
law is invariant. At $P$, its residual relative to the
full-sample core Gaussian is
\[
 N\!\left(0,(q^{-1}-1)\Vk\right)
 *\Law\!\left(q^{-1/2}W^{\rm ev}_{A_0}\right).
\]
The residual has mean zero and variance
$(q^{-1}-1)\Vk+q^{-1}\sum_{a\in\mathcal A^*}
P(A_0=a)(\Sigma_{aa}-\Vk)$; it is nondegenerate precisely
when this expression is positive. The first term records the
information lost by evaluating on a fraction of the sample.
For equal halves and $\Sigma_*=I_2$, the limit
variance is $2$, including residual variance $3/2$ above the
core bound $1/2$.

These decompositions give both the convolution residuals and the local biases of the four estimators. Pooling attains the core bound, while plug-in estimation, smoothing and sample splitting have the displayed residual distributions. To compare their risk when local alternatives separate the best programmes, we must also account for their estimation of the resulting nonlinear change in the parameter.

Here the full risk decomposition is available because \assS holds: $\gamma$ is sublinear. With
\[
B_r=\{h_\perp\in F_P:\norm{h_\perp}\le r\},
\]
\cref{thm:lam-sq} gives $\mathcal M(B_r)=\Vk+\mathcal M_\perp(B_r)$ over the local parameter set $\Hk\times B_r$. The leading residual cost has an explicit form.

\begin{proposition}[small-neighborhood residual risk for programme values]\label{prop:hedging-policy}
In the setting above, with $\psik=\sum_{a\in\mathcal A^*}w_a^*\lambda_a$ and $\LamP=\conv\{\lambda_a:a\in\mathcal A^*\}$, the constant $D$ of \cref{prop:hedging} is
\[
D=\max_{a\in\mathcal A^*}\sqrt{\Sigma_{aa}-\Vk}
  +\dist\big(0,\LamP-\psik\big),
\]
so that $\mathcal M_\perp(B_r)=\tfrac{D^2}{4}r^2+o(r^2)$. If $\psik\in\LamP$, then $D^2=\max_{a\in\mathcal A^*}(\Sigma_{aa}-\Vk)$: the leading residual risk is one quarter of $r^2$ times the largest variance reduction from pooling. When $\Sigma_*$ is nonsingular, the condition $\psik\in\LamP$ is equivalent to $\Sigma_*^{-1}\one\ge0$.
\end{proposition}

\Cref{prop:hedging-policy} quantifies the residual task using the same covariance matrix that determines the Gaussian bound. Under the convex-hull condition, the largest gain from pooling tied programmes also determines the leading risk from allowing their values to separate. For two tied programmes with independent unit-variance components, $\Vk=1/2$ and $D^2=1/2$, so tie-breaking of radius $r$ costs $r^2/8$ to leading order on top of $\Vk$. In the parametrization $(c+d,c-d)$ of the introduction, $\abs d\le\rho$ corresponds to $r=\rho\sqrt2$, giving cost $\rho^2/4$. The Gaussian bound and this residual cost together describe optimal precision in the local Gaussian experiment.

\subsection{Other targets}
The same calculations apply to other directionally pathwise differentiable parameters. A finite maximum or minimum of smooth moments gives the same active-gradient projection for other endpoints of identified sets, including Manski bounds, and for the maximum predictive advantage in forecast comparisons. Where a scalar support-function target is represented by a finite minimum, the superlinear version applies \citep[\S 2.2]{HiranoPorter2012}. Each case requires the path expansion and estimator conditions of its statistical model.

Absolute values and rankings lead to different geometries. The derivative of $\E\abs{\tau(X)}$ at a law with $P(\tau(X)=0)>0$ has the conditional absolute-value structure illustrated by calibration error. The $k$-th best programme value and ranked effects instead lead to the order-statistic spaces of \cref{app:order}. For the value of personalization, $\E\max_a\mu_a(X)-\max_a\E\mu_a(X)$, the difference of two maxima gives the subspace $(F_1+F_2)^\perp$ contained in the core, by \cref{prop:dc}(3), without establishing the additivity condition needed for a full risk factorization. Process-valued versions require the further arguments discussed in \cref{sec:companion}.

\section{Discussion}\label{sec:companion}

This paper develops a convolution theorem for directionally pathwise differentiable parameters. The theorem characterizes the limiting laws of estimators that are regular along a subspace or convex cone, while the geometry of the derivative identifies the relevant subspaces and their efficiency bounds. An estimator that is regular along a given subspace and attains its centered Gaussian bound is asymptotically linear with the corresponding efficient influence function. Under the additivity condition, the local Gaussian minimax problem also separates into linear and nonlinear components. Together, these results extend the usual semiparametric analysis of regularity, efficiency, and local risk to parameters with nonlinear directional derivatives.

\subsection{Relation to the earlier treatment-value and variational analyses}\label{sec:prior-work}
The three papers approach efficiency through related questions. \citet{XuGuo2025} study estimation and inference for optimal treatment values, derive a variance bound within a robust asymptotically linear class, and construct adaptive smoothing estimators that attain it. Their Supplement A.12 allows heteroscedastic outcomes, and Supplement D gives the multistage extension. For binary treatment, the optimal weights used here are shared with those results. Our treatment applications connect them to regularity along the core and examine the changes in variance and local bias under other local alternatives.

\citet{XuGuoRALU} develop the variance minimization problem for marginal-integral functionals. Their constraints require the estimating function to remain unbiased when one specified nuisance component is misspecified and the others are correct, with the target's structural component held at its true value. The choice of nuisance components and admissible distributions therefore enters the efficiency comparison. Their results characterize the estimating functions compatible with these identities, determine their minimum variance, and give conditions for feasible attainment. This formulation does not require directional pathwise differentiability.

The theories have a concrete intersection in the conditional projection. Theorem~6.1 of \citet{XuGuoRALU} projects a branch gradient onto the directions in which tied branches have the same first-order change. Under its path conditions, it identifies the canonical gradient in the specified feasible local experiment and hence the convolution bound for every estimator regular in that experiment. Corollary~6.2 links this bound to the variational optimum for affine functions of treatment-specific conditional means, when the model supports its prescribed perturbations and an admissible estimating function realizes the minimizing influence function. In the treatment models considered here, when those conditions and our path assumptions hold, this influence function and variance agree with the core projection and bound. The equality concerns the optimizer and its variance; robust unbiasedness and local regularity specify different requirements on estimators.

Our analysis starts from the directional derivative and determines the subspaces and cones on which regular estimation is possible. It compares their convolution bounds without requiring an asymptotically linear representation, characterizes efficient estimators, and, under the additivity condition, separates the Gaussian contribution from the nonlinear minimax problem. These questions apply beyond conditional maxima: the median example has several incomparable maximal regularity subspaces. The overlap in efficient influence functions thus connects the three papers, while the variational constraints, feasible constructions, and local regularity and risk questions give each analysis its own scope.

\subsection{Further estimation questions}
The distinction between Gaussian efficiency and residual risk suggests several questions. The core equals the target's additivity space for the principal examples, but these spaces can differ for a general directional derivative. In that case the common Gaussian bound and a decomposition of the full local decision problem concern different structures. A more complete account of their relationship would clarify when the common-core convolution bound also corresponds to a decomposition of the local minimax problem. For higher order statistics, the enumeration of competing maximal spaces also remains open beyond the cases treated in \cref{app:order}.

The present results concern scalar targets in models satisfying \assT. Conditional covariates and histories may be general in the derivative calculations, whereas the empirical attaining constructions here use finite strata or a finite history tree. A vector-valued target requires a common regularity set for all coordinates and a joint Gaussian convolution statement; process-valued targets and dependent sampling require further arguments. Sequential observations within independent trajectories, as in \cref{ex:dynamic}, retain i.i.d.\ sampling at the trajectory level.

The Gaussian-limit minimax equalities and the sequence-level lower bounds serve different roles. The former identify the residual decision problem under \assS; the latter follow from the finite-experiment transfer and do not by themselves construct an attaining estimator for every model. The small-radius result quantifies squared-error risk over tie-breaking alternatives, not confidence-interval length. The feasible constructions and existing procedures cited in the applications establish estimation and inference guarantees under their respective conditions. Extensions to LAMN experiments, to upper semicontinuous losses through an enlarged loss, and to targets whose first derivative is linear but whose higher-order behavior is irregular remain separate questions. The example in \cref{app:rotating}, read alongside \citet{Pfanzagl2000}, also illustrates why the existence of directional limits must be stated explicitly.

\subsection{Uniform confidence intervals}
For a directionally pathwise differentiable target, what is the shortest attainable confidence interval when coverage must hold uniformly over distributions with unknown ties and near ties, and how can that interval be constructed? Corollary~\ref{cor:ci-local} answers a narrower question by deriving coverage along each fixed local path. Uniform inference requires controlling paths and nuisance estimation together, including sequences whose active sets change with the sample size. The local error need not be Gaussian, and an estimator minimizing squared risk need not yield an interval minimizing length.

The Gaussian product experiment provides a starting point. Equation~\eqref{eq:ci-general-critical} calibrates fixed half-lengths for specified centers, while allowing both endpoints to depend on the orthogonal observation permits asymmetric and variable-length intervals. A general theory would identify a length lower bound under a stated uniform coverage requirement and construct a feasible procedure attaining it. Such a result must distinguish fixed length from worst-case expected length, and unrestricted intervals from intervals satisfying an equivariance restriction. Derivative-based resampling \citep{FangSantos2019} and bias-aware interval optimization \citep{ArmstrongKolesar2021} supply relevant constructions and comparison principles. Determining their sharp implications for nonlinear directional derivatives, estimating the needed geometry, and transferring coverage and length guarantees uniformly to the original model form a separate inference problem. The confidence-interval results here make this question precise without presuming its solution.

\appendix
\section{Proofs of the convolution and efficiency results}\label{app:theory}
This appendix contains the proofs for \cref{sec:setting,sec:linearity,sec:cost,sec:structure}. The first subsection establishes the likelihood and continuity lemmas; subsequent subsections give proofs in the order of the main results. All path and target assumptions are those stated in the main text.

\subsection{Likelihood and continuity lemmas}
\begin{lemma}[automatic continuity]\label{lem:autocont}
Under \assT, \cref{def:dpd} implies that $\gamma$ is continuous on $\T$.
\end{lemma}
\begin{proof}
Let $h_j\to h$ in $\T$, let $P^{(j)}_t$ be DQM paths with scores $h_j$ and $P_t$ a DQM path with score $h$. Choose $t_j\downarrow0$ so small that the DQM defect $\int[(dP^{(j)1/2}_{t_j}-dP^{1/2})/t_j-\tfrac12h_j\,dP^{1/2}]^2$ and the quotient error $\abs{t_j^{-1}\{\Psi(P^{(j)}_{t_j})-\Psi(P)\}-\gamma(h_j)}$ are both at most $1/j$. Define $\tilde P_t:=P_t$ for $t\notin\{t_j\}$ and $\tilde P_{t_j}:=P^{(j)}_{t_j}$. Then $\tilde P$ is a DQM path with score $h$: at $t=t_j$ its defect is at most $2/j+\tfrac12\norm{h_j-h}^2\to0$, and elsewhere it is that of $P_t$. Applying \cref{def:dpd} to $\tilde P$ gives $t_j^{-1}\{\Psi(P^{(j)}_{t_j})-\Psi(P)\}\to\gamma(h)$, hence $\gamma(h_j)\to\gamma(h)$.
\end{proof}

\begin{lemma}[tilt identity and analytic continuation]\label{lem:tilt}
Let $C\subseteq\T$ be a convex cone, $T_n$ regular along $C$ with limit law $L$, and $h_1,\dots,h_d\in C$ linearly independent with Gram matrix $G$, $W:=\spn(h_i)$, $C':=\{h_a:=\sum_ia_ih_i:a\in\R^d_{\ge0}\}\subseteq C$. Let $S_n:=\sqrt n\{T_n-\Psi(P)\}$, $\Delta_n:=\sqrt n\,\mathbb P_n(h_1,\dots,h_d)^\top$, and let $(S,\Delta)$ be a subsequential joint limit of $(S_n,\Delta_n)$ under $P^n$ (it exists by tightness and Prohorov's theorem; $S\sim L$, $\Delta\sim N(0,G)$). Then $g(a):=\gamma(h_a)$ satisfies $\abs{g(a)}\le K\abs a$ on $\R^d_{\ge0}$ for some $K<\infty$, there is a linear $\ell$ on $\R^d$ with $g=\ell$ on $\R^d_{\ge0}$, and, along the subsequence,
\begin{equation}\label{eq:tilt}
\phi_L(t)\,e^{it\ell(a)}=e^{-a^\top Ga/2}\,\E\big[e^{itS+a^\top\Delta}\big]\qquad\text{for all }a\in\R^d,\ t\in\R .
\end{equation}
Consequently $S_{n_k}-\ell(a)\wto L$ under $P^{n_k}_{n_k,a}$ for every $a\in\R^d$: the limiting statistic is equivariant in distribution along the whole span $W$ for the linear target $\ell$, whether or not $\gamma=\ell$ off the cone. Moreover $\ell$ agrees with $\gamma$ on all of $C\cap W$, not only on $C'$.
\end{lemma}
\begin{proof}
\emph{LAN.} For $a\ge0$ fix a DQM path with score $h_a$ and write $P_{n,a}$. By \eqref{eq:LAN},
\[
\Lambda_n(a):=\sum_i\log\frac{dP_{n,a}}{dP}(X_i)
=a^\top\Delta_n-\tfrac12a^\top Ga+o_{P^n}(1),
\qquad \Delta_n\wto N(0,G).
\]

\emph{Le Cam's third lemma} \citep[Theorem 6.6]{vdV1998}. $S_{n_k}\wto S^{(a)}$ under $P^{n_k}_{n_k,a}$, where $P(S^{(a)}\in B)=\E[\ind_B(S)e^{a^\top\Delta-a^\top Ga/2}]$, so $\phi_{S^{(a)}}(t)=e^{-a^\top Ga/2}\Phi_t(a)$ with $\Phi_t(a):=\E[e^{itS+a^\top\Delta}]$.
\emph{Regularity.} $c_n(a):=\sqrt n\{\Psi(P_{n,a})-\Psi(P)\}\to g(a)$ and $S_n-c_n(a)\wto L$ under $\Pn a$; hence $S^{(a)}\overset d=Y+g(a)$ with $Y\sim L$, and $\phi_L(t)e^{itg(a)}=e^{-a^\top Ga/2}\Phi_t(a)$ for $a\ge0$.
\emph{Boundedness.} For $\abs a\le1$, Cauchy--Schwarz and $\E e^{2a^\top\Delta}=e^{2a^\top Ga}$ give $P(\abs{S^{(a)}}>M)\le P(\abs S>M)^{1/2}e^{\norm G/2}\to0$ uniformly, so $\{S^{(a)}\}_{\abs a\le1,a\ge0}$ is tight and $K:=\sup\{\abs{g(a)}:\abs a\le1,a\ge0\}<\infty$ (if $g(a_k)\to\pm\infty$ then $Y+g(a_k)$ is not tight); homogeneity gives $\abs{g(a)}\le K\abs a$.
\emph{Analyticity.} $\Phi_t$ is entire on $\C^d$: expand $e^{a^\top\Delta}=\sum_m(a^\top\Delta)^m/m!$, dominate by $e^{\abs a\abs\Delta}$, and use Fubini. Fix $t_0\ne0$ with $\phi_L(t_0)\ne0$; $G_0(a):=e^{-a^\top Ga/2}\Phi_{t_0}(a)/\phi_L(t_0)$ is real-analytic on $\R^d$ with $G_0(0)=1$. On a ball $U\ni0$ with $\abs{G_0-1}<1$ and $K\abs{t_0}\abs a<\pi$, $\tilde g:=(it_0)^{-1}\Log G_0$ is analytic and equals $g$ on $U\cap\R^d_{\ge0}$ (there $G_0=e^{it_0g}$ with $\abs{t_0g}<\pi$). For $a\ge0$ and $s\downarrow0$, $g(a)=g(sa)/s=\tilde g(sa)/s\to D\tilde g(0)a=:\ell(a)$.
\emph{Continuation.} With $g=\ell$ on $\R^d_{\ge0}$, both sides of the tilt identity are real-analytic in $a$ and agree on a set with nonempty interior, hence on $\R^d$, for each $t$; and $\phi_{S^{(a)}}(t)=\phi_L(t)e^{it\ell(a)}$ for all $a$ means $S^{(a)}\overset d=Y+\ell(a)$.
\emph{Consistency on $C\cap W$.} For $h\in C\cap W$ with coordinates $a(h)$ (not necessarily $\ge0$), regularity along $C$ gives $S_{n_k}-c_{n_k}(h)\wto L$ with $c_n(h)\to\gamma(h)$, while the continued identity gives $S_{n_k}\wto L*\delta_{\ell(a(h))}$ under the same laws; a probability law has no nonzero translation period, so $\ell(a(h))=\gamma(h)$.
\end{proof}

\begin{lemma}[convolution with an explicit noise factor]\label{lem:noise}
In the setting of \cref{lem:tilt}, after orthonormalizing the basis of $W$ (so $G=I_d$ and $\ell(a)=\psi^\top a$ with $\psi\in\R^d$ the coordinates of the representer $\tilde\psi_W$ of $\ell$), the variable $R:=S-\psi^\top\Delta$ is independent of $\Delta$, and
\[
L=N(0,\abs\psi^2)*\Law(R),\qquad\phi_R(t)=\phi_L(t)e^{\abs\psi^2t^2/2}.
\]
The noise factor is the law of the residual of $S$ after removal of the efficient Gaussian component $\Delta(\tilde\psi_W)$, and it is determined by $L$ and $\abs\psi$ alone.
\end{lemma}
\begin{proof}
Identity \eqref{eq:tilt} holds for complex $a$ (both sides are entire; read $\abs a^2$ as $a^\top a$). At $a=b-it\psi$, $b\in\R^d$, the left side is $\phi_L(t)e^{it\psi^\top b+t^2\abs\psi^2}$ and the right side is $e^{-\abs b^2/2+itb^\top\psi+t^2\abs\psi^2/2}\E[e^{itR+b^\top\Delta}]$; hence $\E[e^{itR}e^{b^\top\Delta}]=\phi_L(t)e^{t^2\abs\psi^2/2}e^{\abs b^2/2}$ for all real $t,b$. At $b=0$: $\phi_R=\phi_Le^{\abs\psi^2t^2/2}$. Both sides are entire in $b$, so at $b=is$: $\E[e^{itR+is^\top\Delta}]=\phi_R(t)\phi_\Delta(s)$, i.e.\ $R\indep\Delta$.
\end{proof}

\begin{lemma}[regularity passes to closures]\label{lem:closure}
If $T_n$ is regular along a linear subspace $V$ with limit $L$, it is regular along $\overline V\cap\T$ with the same limit. Consequently $\Rfam(T_n)$ is closed under closure and maximal regularity subspaces are closed.
\end{lemma}
\begin{proof}
Let $h\in\overline V\cap\T$ and $h_j\to h$ in $V$. DQM at $P$ gives, for one observation, $\int(\sqrt{dP_{n,h}}-\sqrt{dP_{n,h_j}})^2=\norm{h-h_j}^2/(4n)+o(1/n)$, i.e.\ Hellinger affinity $\rho_{n,j}=1-\norm{h-h_j}^2/(8n)+o(1/n)$; affinities multiply over products, so $\rho_{n,j}^n\to e^{-\norm{h-h_j}^2/8}$, and the inequality $\TV\le\sqrt{1-\rho^2}$ (with $\TV(P,Q)=\sup_A\abs{P(A)-Q(A)}$) gives
\[
\limsup_n\TV(\Pn h,\Pn{h_j})\le\sqrt{1-e^{-\norm{h-h_j}^2/4}}\le\tfrac12\norm{h-h_j}.
\]
No finite-sample absolute continuity is used; DQM paths may carry asymptotically negligible singular parts. For bounded Lipschitz $f$ and $S_n=\sqrt n\{T_n-\Psi(P)\}$,
\[
\abs{\E_{n,h}f(S_n-c_n(h))-\E_{n,h_j}f(S_n-c_n(h_j))}\le2\norm f_\infty\TV(\Pn h,\Pn{h_j})+\Lip(f)\,\abs{c_n(h)-c_n(h_j)} .
\]
Regularity along $V$ gives $\E_{n,h_j}f(S_n-c_n(h_j))\to\int f\,dL$, and $c_n(h)\to\gamma(h)$, $c_n(h_j)\to\gamma(h_j)$, $\gamma(h_j)\to\gamma(h)$ (\cref{lem:autocont}). Letting $n\to\infty$ and then $j\to\infty$, $\E_{n,h}f(S_n-c_n(h))\to\int f\,dL$ for every bounded Lipschitz $f$, i.e.\ $S_n-c_n(h)\wto L$ under $\Pn h$.
\end{proof}

\subsection{Supporting geometry}
\begin{lemma}\label{lem:odd}
\begin{enumerate}[label=(\arabic*)]
\item $\Nodd(\gamma)$ is a closed cone with $\Nodd=-\Nodd$, and $\gamma(ch)=c\gamma(h)$ for all $c\in\R$, $h\in\Nodd(\gamma)$.
\item A linear subspace $V$ is a linearity subspace iff $V\subseteq\Nodd(\gamma)$ and $\gamma$ is additive on $V$. Every one-dimensional subspace of $\Nodd(\gamma)$ is a linearity subspace.
\item $\Lfam(\gamma)$ is closed under passing to subspaces; the closure in $\T$ of a linearity subspace is a linearity subspace; every linearity subspace is contained in a maximal one; maximal linearity subspaces are closed.
\item $\gamma$ is linear on $\T$ iff $\Lfam_{\max}(\gamma)=\{\T\}$.
\end{enumerate}
\end{lemma}
\begin{proof}
(1) Closedness and symmetry are clear; for $h\in\Nodd$ and $c<0$, $\gamma(ch)=\abs c\gamma(-h)=-\abs c\gamma(h)=c\gamma(h)$. (2) Linear on $V$ implies odd on $V$, hence $V\subseteq\Nodd$; conversely additivity on $V\subseteq\Nodd$ together with (1) gives homogeneity for all real scalars. A line $\R h$ with $h\in\Nodd$ is additive by (1). (3) If $V\in\Lfam$ and $h_k\to h$, $h'_k\to h'$ in $V$, continuity gives $\gamma(h+h')=\gamma(h)+\gamma(h')$ and $\gamma(-h)=-\gamma(h)$, so $\overline V\cap\T\in\Lfam$. The union of a chain in $\Lfam$ is a linear subspace on which $\gamma$ is linear, so Zorn's lemma yields maximal elements above any member; a maximal element equals its closure. (4) is immediate.
\end{proof}

\subsection{Proofs of the main statements}
\begin{proof}[Proof of \cref{thm:main}]
(1) For each finite linearly independent $B\subseteq C$, \cref{lem:tilt} gives a linear $\ell_B$ on $\spn B$ agreeing with $\gamma$ on $C\cap\spn B$. Given two such sets $B_1,B_2$, choose a basis $B_0\subseteq B_1\cup B_2$ of $W_0:=\spn(B_1\cup B_2)$ (two independent sets need not lie in a common independent set: $\{h\}$ and $\{2h\}$). Then $\ell_{B_0}$ agrees with $\gamma$ on $C\cap W_0\supseteq B_i$, hence with $\ell_{B_i}$ on $B_i$ and therefore on $\spn B_i$. The spans $\{\spn B\}$ form a directed family with union $\spn C$ on which the functionals are consistent, so $\ell_C$ is well defined on $\spn C$. \cref{lem:noise} gives $\abs{\phi_L(t)}\le e^{-t^2\norm{\tilde\psi_W}^2/2}$ for every $W=\spn B$; if $\ell_C$ were unbounded then $\sup_W\norm{\tilde\psi_W}=\infty$ and $\phi_L(t)=0$ for $t\ne0$, contradicting continuity of $\phi_L$ at $0$. Thus $\ell_C$ is bounded, $\gamma=\ell_C$ is Lipschitz on $C$, and for $C=V$ a subspace $\gamma|_V$ is bounded and linear.
(2) For finite-dimensional spans this is \cref{lem:noise}. In general $\norm{\tilde\psi_C}^2=\sup_W\norm{\tilde\psi_W}^2$ over finitely generated subcones; for each $W$, $\phi_L(t)e^{\norm{\tilde\psi_W}^2t^2/2}$ is a characteristic function; taking $W_k$ with $\norm{\tilde\psi_{W_k}}\uparrow\norm{\tilde\psi_C}$, the pointwise limit $\phi_L(t)e^{\norm{\tilde\psi_C}^2t^2/2}$ is continuous at $0$, hence by L\'evy's continuity theorem the characteristic function of a law $M_C$.
\end{proof}

\begin{proof}[Proof of \cref{thm:core}]
(1) For $a,c\in\R$: $av\in\Nodd$, so $\gamma(av+cw)=\gamma(av)+c\gamma(w)=a\gamma(v)+c\gamma(w)$ by \cref{def:core} and \cref{lem:odd}(1); linearity on $\spn(v,w)$ forces $\spn(v,w)\subseteq\Nodd$ (\cref{lem:odd}(2)).
(2) Let $w_1,w_2\in\Hk$ and $v\in\Nodd$. By (1), $v+cw_1\in\Nodd$, so $\gamma(v+c(w_1+w_2))=\gamma(v+cw_1)+c\gamma(w_2)=\gamma(v)+c\gamma(w_1)+c\gamma(w_2)$; with $v=w_1$, $c=1$ this gives $\gamma(w_1+w_2)=\gamma(w_1)+\gamma(w_2)$, hence $\gamma(v+c(w_1+w_2))=\gamma(v)+c\gamma(w_1+w_2)$, and $w_1+w_2\in\Nodd$ by (1). Scalars: $\gamma(v+c'(cw))=\gamma(v)+c'c\gamma(w)=\gamma(v)+c'\gamma(cw)$ by \cref{lem:odd}(1). Closedness follows from continuity of $\gamma$ and closedness of $\Nodd$.
(3) For $V\in\Lfam$, $v_1,v_2\in V\subseteq\Nodd$, $w_1,w_2\in\Hk$: $\gamma(v_1+w_1)=\gamma(v_1)+\gamma(w_1)$, and $\gamma((v_1+w_1)+(v_2+w_2))=\gamma(v_1+v_2)+\gamma(w_1+w_2)=\gamma(v_1+w_1)+\gamma(v_2+w_2)$; real homogeneity likewise. If $V$ is maximal this gives $\Hk\subseteq V$. Conversely, if $w$ lies in every maximal linearity subspace then for any $v\in\Nodd$ the line $\R v$ lies in some maximal $V\ni w$, so $\gamma$ is linear on $\spn(v,w)$ and $w\in\Nodd$; thus $w\in\Hk$.
(4) Closure of $\Afam$ under sums and scalars is checked as in (2) with all $h\in\T$; $h=-w$ gives $\gamma(-w)=-\gamma(w)$, so $\Afam\subseteq\Nodd$, and then $\Afam\subseteq\Hk$; if $\Nodd=\T$ the two definitions coincide.
(5) If $\Hk=\T$ then $\gamma$ is linear by (2); the converse is trivial.
\end{proof}

\begin{proof}[Proof of \cref{prop:bounds}]
(1) For $h\in V$, $\ip{\tilde\psi_{V'}}h=\gamma(h)=\ip{\tilde\psi_V}h$. (2) Lines in $\Nodd$ belong to $\Lfam$, giving $\ge$; every $V\in\Lfam$ lies in $\Nodd$, giving $\le$. (3) \cref{thm:core}(3) and (1).
\end{proof}

\begin{proof}[Proof of \cref{cor:supremum}]
For every $C\in\Rfam^{\rm cone}(T_n)$, \cref{thm:main}(2) gives $\abs{\phi_L(t)}\le e^{-V_Ct^2/2}$; with $t_0\ne0$ such that $\phi_L(t_0)\ne0$, $V_C\le-2\log\abs{\phi_L(t_0)}/t_0^2$, so $v(T_n)<\infty$. The L\'evy argument of \cref{thm:main}(2) along $C_k$ with $V_{C_k}\uparrow v(T_n)$ gives $M$; $\sup_{V\in\Rfam}V_V\le\Vstar$ by \cref{prop:bounds}(2) and \cref{thm:main}(1).
\end{proof}

\begin{proof}[Proof of \cref{thm:kink}]
(1) For $h_0\in\overline\T$, $sh_0\mapsto s\gamma(h_0)$ is dominated by $\gamma$ on $\R h_0$ (for $s<0$ use $\gamma(h_0)+\gamma(-h_0)\ge0$), extends by Hahn--Banach to a linear $\ell\le\gamma$, bounded by $C=\sup_{\norm u\le1}\abs{\gamma(u)}$, hence $\ell=\ip\lambda\cdot$ with $\lambda\in\LamP$ and $\ip\lambda{h_0}=\gamma(h_0)$. (2) By (1), $\gamma(-h)=-\min_{\LamP}\ip\lambda h$, so $h\in\Nodd$ iff $\ip{\lambda-\lambda'}h=0$ for all $\lambda,\lambda'\in\LamP$, i.e.\ $\Nodd=\FP^\perp\cap\T$, a linear subspace on which $\gamma$ is linear; every linearity subspace lies in $\Nodd$, so $\Nodd$ is the unique maximal one and $\Hk=\Nodd$ by \cref{thm:core}(3). For $w\in\Nodd$: $\gamma(h+w)\le\gamma(h)+\gamma(w)$ and $\gamma(h)\le\gamma(h+w)+\gamma(-w)$ give $w\in\Afam$. (3) On $\Nodd$, $\lambda\mapsto\ip\lambda h$ is constant on $\LamP$ and equals $\gamma(h)$, so $\Pi_{\Hk}\lambda=\psik$; under \assT, $\FP\subseteq\T$ is closed, $\T=(\FP^\perp\cap\T)\oplus\FP$, and $\Pi_{\Hk}\lambda=\lambda-\Pi_{\FP}\lambda$ is the minimum-norm point of $\lambda+\FP=\claff\LamP$. $\Vstar=\Vk$ since $\Lfam_{\max}=\{\Hk\}$. (4) $\norm{\Pi\lambda}\le\norm\lambda$ with equality iff $\lambda\in\Hk$. (5) $\norm{\sum_aw_a\lambda_a}^2=w^\top\Sigma w$ and $\claff\LamP=\{\sum_aw_a\lambda_a:\one^\top w=1\}$. If $\one\notin\range(\Sigma)=\ker(\Sigma)^\perp$, some $v\in\ker\Sigma$ has $\one^\top v\ne0$ and $w=v/\one^\top v$ gives value $0$. If $\one\in\range(\Sigma)$, the KKT conditions $\Sigma w=\nu\one$, $\one^\top w=1$ give $w=\nu\Sigma^+\one+k$, $k\in\ker\Sigma$, $\nu=1/\one^\top\Sigma^+\one>0$, and value $\nu$.
\end{proof}

\begin{proof}[Proof of \cref{cor:core-convolution}]
(1) $\Hk\subseteq V$ for maximal $V$ by \cref{thm:core}(3). Regularity passes to subspaces, so \cref{thm:main}(2) applied to $\Hk$ gives the convolution, and \cref{cor:supremum} gives the inequality.
(2) \Cref{thm:kink}(2)--(3) gives the unique maximal subspace and its bound. Apply \cref{thm:main}(2) on that subspace. For a superlinear derivative, apply the geometry to $-\gamma$ with the sign convention in \cref{def:kink}.
\end{proof}

\begin{proof}[Proof of \cref{lem:reduction}]
$\gamma(h+w)=g(\ell h)=\gamma(h)+\gamma(w)$ for $w\in\ker\ell$, so $\ker\ell\subseteq\Afam(\gamma)\subseteq\Hk(\gamma)$; hence every linearity subspace can be enlarged by $\ker\ell$ (\cref{thm:core}(3)), maximal ones contain $\ker\ell$ and are of the form $\ell^{-1}(U)$. $\gamma$ is linear on $V$ iff $g$ is linear on $\ell(V)$. Since $\ell$ is onto and $\ker\ell=(\range\ell^*)^\perp$, $\ell^{-1}(U)=(\ell^*U^\perp)^\perp$, so $V^\perp\cap\T=\ell^*U^\perp$. On $V$, $\gamma(h)=\ip{\ell^*b}h$, so $\tilde\psi_V=\Pi_V\ell^*b$ and $V_V=\min_{u\in\ell^*U^\perp}\norm{\ell^*b-u}^2$.
\end{proof}

\begin{proof}[Proof of \cref{prop:vstar}]
$\Nodd(\gamma)=\T$ by oddness, so $\Vstar=\sup_{h\ne0}\gamma(h)^2/\norm h^2$. For $z=\ell(h)$,
\[
\min\{\norm h^2:\ell(h)=z\}=z^\top\Sigma^{-1}z
\]
(attained at $h=\ell^*\Sigma^{-1}z$), so $\Vstar=\max_i\sup_{z\in C_i}(b_i^\top z)^2/z^\top\Sigma^{-1}z$. With $z=\Sigma^{1/2}y$ and $u_i:=\Sigma^{1/2}b_i$ this is $\max_i\sup_{y\in\mathcal C_i}(u_i^\top y)^2/\abs y^2$. Moreau's decomposition $u=\Pi_{\mathcal C}u+\Pi_{\mathcal C^\circ}u$ with orthogonal summands gives $u^\top y\le(\Pi_{\mathcal C}u)^\top y\le\norm{\Pi_{\mathcal C}u}\abs y$ for $y\in\mathcal C$, with equality at $y=\Pi_{\mathcal C}u$; so
\[
\sup_{y\in\mathcal C}(u^\top y)^2/\abs y^2=\max\{\norm{\Pi_{\mathcal C}u}^2,\norm{\Pi_{\mathcal C}(-u)}^2\}.
\] If the fan is not centrally symmetric, pass to the common refinement of the fan and its negative; on it oddness carries the piece $b_i$ on $C_i$ to the piece $b_i$ on $-C_i$, so the maximum over $i$ accounts for both signs. The remaining statements follow from \cref{prop:bounds} and $\norm{\Pi_{\mathcal C_i}u_i}\le\norm{u_i}$ with equality iff $u_i\in\mathcal C_i$.
\end{proof}

\begin{proof}[Proof of \cref{prop:limitexp}]
On a subspace $V$ the continuous linear functional $\gamma|_V$ is bounded, so $\tilde\psi_V$ exists; under $h$, $\Delta(\tilde\psi_V)\sim N(\ip{\tilde\psi_V}h,V_V)$ and $\ip{\tilde\psi_V}h=\gamma(h)$ for $h\in V$. The cone statement is the same construction with $\tilde\psi_C$; its converse is \cref{thm:main}(1).
\end{proof}

\begin{proof}[Proof of \cref{thm:al}]
Under $P^n$, $(\sqrt n\mathbb P_n\phi,\sqrt n\mathbb P_nh)$ is asymptotically bivariate normal with covariance $\ip\phi h$, so by Le Cam's third lemma and contiguity $\sqrt n\{T_n-\Psi(P)\}\wto N(\ip\phi h,\norm\phi^2)$ under $\Pn h$; subtracting $c_n(h)\to\gamma(h)$ gives the displayed limit, which is free of $h$ on $C$ iff $\ip\phi h=\gamma(h)$ on $C$. The rest is \cref{prop:bounds} with $\gamma|_V=\ip\phi\cdot$.
\end{proof}

\begin{proof}[Proof of \cref{cor:alkink}]
$\Hk\subseteq \mathcal E(\phi)$ iff $\phi-\psik\perp\Hk$ iff $\phi-\psik\in\Hk^\perp=\FP\oplus\T^\perp$ (orthogonal complement in $L_2^0(P)$ of $\FP^\perp\cap\T$; the sum of the orthogonal closed subspaces $\FP\subseteq\T$ and $\T^\perp$ is closed), and $\psik+\FP=\claff\LamP$.
\end{proof}

\begin{proof}[Proof of \cref{prop:suff}]
$\hat T_V$ is a fixed linear combination of jointly AL estimators, hence AL with influence function $\ell^*(b-c^\star)=\tilde\psi_V$ (\cref{lem:reduction}). $c^\star\perp\one$ (as $U\ni\one$) gives $c^{\star\top}\mu(P)=0$, and $b^\top\mu(P)=g(\mu(P))=\Psi(P)$ because $\mu(P)\in U$. With an estimated $\hat c\in U^\perp$, $\hat c\to c^\star$: since $\hat c\perp\one$ exactly, $\hat c^\top\mu(P)=0$ exactly, and $(b-\hat c)^\top\hat\mu-(b-c^\star)^\top\hat\mu=-(\hat c-c^\star)^\top(\hat\mu-\mu(P))=o_P(1)\cdot O_P(n^{-1/2})$, so the influence function is unchanged. Regularity and efficiency: \cref{thm:al,thm:efficiency}, since $\ip{\tilde\psi_V}h=b^\top\ell(h)=\gamma(h)$ for $\ell(h)\in U$.
\end{proof}

\begin{proof}[Proof of \cref{thm:efficiency}]
This is \citet[Lemma 25.23]{vdV1998} applied with the linear tangent set $V$ and efficient influence function $\tilde\psi_V$; alternatively, by \cref{lem:noise} the residual $S-\Delta(\tilde\psi_V)$ is zero almost surely exactly when $L=N(0,V_V)$, which is the asymptotic linearity statement.
\end{proof}

\begin{proof}[Proof of \cref{prop:tradeoff}]
A $V$-efficient $T_n$ has $L=N(0,V_V)$; regularity along $V'$ would give $L=N(0,V_{V'})*M'$, of variance $\ge V_{V'}>V_V$. The second claim is \cref{thm:main}(2) and \cref{cor:supremum}; for sublinear or superlinear $\gamma$, every $V'\in\Lfam$ lies in $\Hk$ and $V_{V'}\le\Vk$.
\end{proof}

\begin{proof}[Proof of \cref{cor:onesided}]
A vertex of a polytope is exposed, so some $h$ has $\ell_a(h)>\ell_b(h)$ for all $b$ with $\lambda_b\ne\lambda_a$; the strict cone is a finite intersection of open half-spaces, hence open, and a nonempty open cone spans $\T$. \cref{thm:main}(1)--(2) give the bound with $\tilde\psi_{C_a}=\lambda_a$, the unique bounded linear functional agreeing with $\gamma=\ell_a$ on a spanning set. Regularity of an AL estimator with influence function $\lambda_a$ along $C_a$: \cref{thm:al} with $\mathcal E(\lambda_a)\supseteq C_a$. The core-efficient estimator has drift $\ip{\psik-\lambda_a}h$ on $C_a$ (\cref{thm:al}), a linear functional that vanishes on the spanning set $C_a$ iff $\psik=\lambda_a$. The plug-in's limit $\max_b(Z_b+\ell_b(h))-\ell_a(h)$ depends on $h\in C_a$ unless $\FP=\{0\}$: along an exposing direction $th_0$ it converges to $Z_a$ as $t\to\infty$, while at $h=0$ its mean exceeds $\E Z_a$ by $\E(\max_b(Z_b-Z_a))_+>0$ whenever some $\lambda_b\ne\lambda_a$.
\end{proof}

\begin{proof}[Proof of \cref{prop:plugin}]
The identity $\Sigma w^*=\one/(\one^\top\Sigma^{-1}\one)$ gives
\[
\Cov\{W,w^{*\top}Z\}
=(I-\one w^{*\top})\Sigma w^*=0.
\]
Joint Gaussianity therefore gives independence. Since $Z=W+(w^{*\top}Z)\one$, translation equivariance yields $g(Z)=g(W)+\theta w^{*\top}Z$.

For $u\in F^\perp$, write $\ell(u)=c_u\one$. Then
$\gamma(h+u)=\gamma(h)+\theta c_u$, so $F^\perp\subseteq\Afam(\gamma)$. With $\psi_0:=\sum_aw_a^*\lambda_a$, the identity for $\Sigma w^*$ shows that $\psi_0\in F^\perp$ and $\ip{\psi_0}u=c_u$ for every $u\in F^\perp$. Thus $\theta\psi_0$ is its representer and has squared norm $\theta^2/(\one^\top\Sigma^{-1}\one)$.

Under any DQM local alternative with score $h$, joint asymptotic linearity, contiguity and Le Cam's third lemma give
\[
\sqrt n\{\hat\mu-\mu(P)\}\wto Z+\ell(h).
\]
The Hadamard directional delta method at $\mu(P)$ and the target expansion $\sqrt n\{\Psi(P_{n,h})-\Psi(P)\}\to\gamma(h)=g(\ell(h))$ then give
\[
\sqrt n\{T_n-\Psi(P_{n,h})\}
\wto g(Z+\ell(h))-g(\ell(h)).
\]
Substituting $Z=W+(w^{*\top}Z)\one$ and $\ell(h)=c\one+z_\perp$ yields the stated local law. For $h\in F^\perp$, $z_\perp=0$, so this law is independent of the path and score.

Continuity and positive homogeneity imply that $g$ has at most linear growth, so the residual has a finite mean. If $g=\max$ and two gradients differ, choose $a,b$ with $\Var(W_a-W_b)>0$. Then
\[
\E\max_jW_j\ge\E\max(W_a,W_b)
=\tfrac12\E\abs{W_a-W_b}>0.
\]
The minimum follows by symmetry. For odd $g$, $W\overset d=-W$ implies $\E g(W)=0$.
\end{proof}

\begin{proof}[Proof of \cref{prop:fixed}]
(a) \cref{thm:al,prop:split}. (b) $Z_{\hat a}=w^{*\top}Z+W_{\hat a}$, $\E W_{\hat a}=\sum_aP(\hat a=a)\E W_a=0$. (c) is (a) with $w=\one/K$.
\end{proof}

\begin{proof}[Proof of \cref{prop:split}]
Additivity along $\Hk=\Afam$; the isonormal process on an orthogonal direct sum is an independent product; the core of $\gamma|_{\Hk^\perp}$ is $\Hk\cap\Hk^\perp=\{0\}$.
\end{proof}

\begin{proof}[Proof of \cref{prop:blind}]
\cref{thm:efficiency} gives asymptotic linearity with influence function $\psik$; \cref{thm:al} gives drift $\ip\psik h-\gamma(h)=-\gamma(h_\perp)$ by \cref{prop:split} and $\psik\perp h_\perp$. Nonpositivity for all $h_\perp$ is $\gamma\ge0$ on $\Hk^\perp$, i.e.\ $\gamma\ge\ip\psik\cdot$ on $\T$; for sublinear $\gamma$ this is the support-function characterization of $\psik\in\LamP$; for odd $\gamma$, $\gamma\ge0$ on a subspace forces $\gamma\equiv0$ there.
\end{proof}

\begin{proof}[Proof of \cref{thm:lam-sq}]
By \cref{lem:global-minimax}, each minimax value equals the supremum of its values on finite nonempty parameter subsets. Upper bound: $T=\Delta_\kappa(\psik)+T_\perp(Z_\perp)$ has risk $\Vk+\E_{h_\perp}(T_\perp-\gamma(h_\perp))^2$. Lower bound: a finite $I$ lies in the product of its projections; for priors $\pi_\kappa\otimes\pi_\perp$ the posterior factorizes, the Bayes estimator is the posterior mean of $\tau$, and its risk is $\E\Var(\ip\psik{h_\kappa}\mid Z_\kappa)+\E\Var(\gamma(h_\perp)\mid Z_\perp)=:B_\kappa(\pi_\kappa)+B_\perp(\pi_\perp)$. (i) $\sup_{I_\kappa,\pi_\kappa}B_\kappa=\Vk$ (if $s=0$ there is nothing to prove): on the line $\R e$, $e=\psik/s$, $\Delta_\kappa(e)\sim N(\vartheta,1)$ is sufficient and $B_\kappa=s^2b(\pi)$ with $b(\pi)$ the Bayes risk of the normal mean. For the prior with density $p_M(\vartheta)=M^{-1}\cos^2(\pi\vartheta/2M)$ on $[-M,M]$, whose Fisher information is $\int(p_M')^2/p_M=\pi^2/M^2$, the van Trees inequality \citep{GillLevit1995} gives $b(p_M)\ge(1+\pi^2/M^2)^{-1}$; to pass to finitely supported priors, quantize $p_M$ within $[-M,M]$, clip actions to $[-M,M]$ (which does not increase squared error against targets in $[-M,M]$) and use the bounded continuous loss $\min(u^2,4M^2)$, which agrees with squared error for all clipped actions and allowed targets, so that \cref{lem:finiteprior} applies; then let $M\to\infty$. (ii) For finite $I_\perp$, $\sup_{\pi_\perp}\inf_{T_\perp}\text{Bayes}=\inf_{T_\perp}\max_{I_\perp}$ by \cref{lem:minimax} (clip actions to $\conv\gamma(I_\perp)$ so that risks are bounded).
\end{proof}

\begin{proof}[Proof of \cref{thm:lam-conv}]
Gaussian symmetry and unimodality show that $\ell_\kappa$ is bowl-shaped; Fatou's lemma gives lower semicontinuity. By \cref{lem:global-minimax}, both minimax values admit finite-subexperiment representations. Upper bound: for $T=\Delta_\kappa(\psik)+U(Z_\perp)$, $\E_h\ell(T-\tau)=\E_{h_\perp}\ell_\kappa(U-\gamma(h_\perp))$. If $s=\norm\psik=0$ then $\ell_\kappa=\ell$, $\tau(h)=\gamma(h_\perp)$, and the lower bound is immediate (restrict $I$ to $\{0\}\times I_\perp$); assume $s>0$.

Lower bound for bounded continuous $\ell$: fix $T$, finite $I_\perp$ and a prior $\pi_\perp$; restrict $h_\kappa=\vartheta e$, $e:=\psik/s$; $Y:=Z_\kappa(e)=\vartheta+\epsilon$ and $Z_\kappa^{\rm anc}:=Z_\kappa-Y\ip e\cdot$ (centred, $\vartheta$-free, independent of $Y$). Prior $\vartheta\sim N(0,\tau^2)$: $\vartheta\mid Y\sim N(\rho Y,\rho)$, $\rho=\tau^2/(1+\tau^2)$, and $\xi:=(\vartheta-\rho Y)/\sqrt\rho\sim N(0,1)$ is independent of $(Y,Z_\kappa^{\rm anc},h_\perp,Z_\perp)$. The Bayes risk of $T$ is $\E\ell(T''-\gamma(h_\perp)-\sqrt\rho s\xi)=\E\phi_\rho(T''-\gamma(h_\perp))$ with $T'':=T-\rho sY$ and $\phi_\rho(c):=\E\ell(c-\sqrt\rho s\xi)$; since $(Y,Z_\kappa^{\rm anc})$ is independent of $(h_\perp,Z_\perp)$, $T''$ is a randomized estimator of $\gamma(h_\perp)$ from $Z_\perp$, so
\[
\sup_{h\in\Hk\times I_\perp}\E_h\ell(T-\tau)\ \ge\ \inf_U\E_{\pi_\perp}\phi_\rho(U-\gamma)\qquad(\rho<1).
\]
Now $\sup_c\abs{\phi_\rho(c)-\ell_\kappa(c)}\le\norm\ell_\infty\,\norm{N(0,\rho s^2)-N(0,s^2)}_{L_1}=:\norm\ell_\infty\epsilon(\rho)$ with $\epsilon(\rho)\to0$ as $\rho\uparrow1$; hence $\inf_U\E_{\pi_\perp}\phi_\rho(U-\gamma)\ge\inf_U\E_{\pi_\perp}\ell_\kappa(U-\gamma)-\norm\ell_\infty\epsilon(\rho)$ uniformly over rules, and letting $\rho\uparrow1$, $\sup_h\E_h\ell(T-\tau)\ge\inf_U\E_{\pi_\perp}\ell_\kappa(U-\gamma)$. (No monotonicity in $\rho$ is used, and none holds: for $\ell(x)=1-e^{-x^2}$, $\phi_\rho(2)$ decreases in $\rho$.) \cref{lem:minimax} turns $\sup_{\pi_\perp}\inf_U$ into $\inf_U\max_{I_\perp}$; the supremum over finite $I_\perp\subseteq B$ gives $\sup_{h\in\Hk\times B}\E_h\ell(T-\tau)\ge\mathcal N_\ell(B)$. The finite-subset version in the $e$-direction is \cref{lem:finiteprior}.

For a general lower semicontinuous $\ell$, use the finite-subexperiment representations just established and choose bounded continuous bowl-shaped $\ell_j\uparrow\ell$. For each finite $I_\perp$ clip actions to the compact interval $J:=\conv\gamma(I_\perp)$ (this does not increase any bowl-shaped loss); for each observation the posterior objective $a\mapsto\E[\ell_{j,\kappa}(a-\gamma(h_\perp))\mid Z_\perp]$ is continuous on $J$ and increases in $j$ to the lower semicontinuous limit objective, so its minimum over $J$ converges to the minimum of the limit (if $a_j$ minimize, a subsequence converges to some $a^*$, and $\lim_j\min\ge\lim_jf_k(a_j)=f_k(a^*)$ for each $k$, hence $\ge f(a^*)\ge\min f$); integrating over the observation (monotone convergence) gives convergence of the Bayes risks, and \cref{lem:minimax} and the outer suprema commute with the monotone limit. Hence $\mathcal N_{\ell_j}(B)\uparrow\mathcal N_\ell(B)$. For the left side, $\mathcal M_\ell(B)\ge\mathcal M_{\ell_j}(B)$ for every $j$, so $\mathcal M_\ell(B)\ge\sup_j\mathcal M_{\ell_j}(B)=\sup_j\mathcal N_{\ell_j}(B)=\mathcal N_\ell(B)$, and the upper bound gives equality. (Risks may be $+\infty$; on a finite set, if no rule has finite risk at every parameter then a full-support prior already has infinite Bayes risk, and otherwise \cref{lem:minimax} applies to the finite-risk rules with boundary priors approximated by full-support priors.)
\end{proof}

\begin{proof}[Proof of \cref{prop:hedging}]
By homogeneity the range of $\gamma$ on $B_r$ has diameter $rD$; the constant estimator at its midpoint has risk $\le(rD/2)^2$. For $h_1,h_2\in B_r$ with $\gamma(h_1)-\gamma(h_2)=\Delta$ and densities $p_1,p_2$ of the two observation laws, the Bayes risk of the equal-weight two-point prior is $\frac{\Delta^2}{2}\int\frac{p_1p_2}{p_1+p_2}\ge\frac{\Delta^2}{4}\int\min(p_1,p_2)=\frac{\Delta^2}{4}(1-\TV)$, and for Gaussian shifts $\TV=2\Phi(\norm{h_1-h_2}/2)-1\le\norm{h_1-h_2}/\sqrt{2\pi}\le2r/\sqrt{2\pi}$; taking $h_1,h_2$ approaching the extremes of $\gamma$ on $B_r$ gives $\Delta\uparrow rD$ (no attainment of the extremes is needed).
\end{proof}

\begin{proof}[Proof of \cref{prop:hedging-policy}]
Here $\Hk^\perp\cap\T=F_P$, and $B_1$ is its closed unit ball. Put $u_a=\lambda_a-\psik=\Pi_{F_P}\lambda_a$ and $U=\conv\{u_a:a\in\mathcal A^*\}=\LamP-\psik$. Orthogonality gives $\norm{u_a}^2=\Sigma_{aa}-\Vk$. Hence
\[
\sup_{h\in B_1}\gamma(h)
=\sup_{h\in B_1}\max_a\ip{u_a}h
=\max_a\norm{u_a}.
\]
Let $u_0$ be the minimum-norm point of $U$ and $q_0=\norm{u_0}$. For every $h\in B_1$, $\gamma(h)\ge\ip{u_0}h\ge-q_0$. If $q_0>0$, the projection inequality $\ip{u-u_0}{u_0}\ge0$ for $u\in U$ gives $\gamma(-u_0/q_0)=-q_0$. If $q_0=0$, the ball point $h=0$ attains the same lower bound. Therefore $\inf_{B_1}\gamma=-\dist(0,U)$, which proves the formula for $D$. The distance vanishes exactly when $\psik\in\LamP$; the inverse criterion under nonsingularity is \cref{rem:simplex}.
\end{proof}

\subsection{Finite-experiment and minimax tools}
\begin{lemma}[finite-support priors in the $e$-direction]\label{lem:finiteprior}
For bounded uniformly continuous bowl-shaped $\ell$ (every bounded continuous bowl-shaped loss is uniformly continuous on $\R$), fixed $I_\perp,\pi_\perp$, and priors $\pi_m\Rightarrow\pi$ on $\vartheta$, the Bayes risks of the problem ``observe $Y\sim N(\vartheta,1)$ and $Z_\perp$; estimate $s\vartheta+\gamma(h_\perp)$'' converge: $r(\pi_m)\to r(\pi)$. In particular the Gaussian prior on $\vartheta$ in the proof of \cref{thm:lam-conv} may be replaced by finitely supported priors at an arbitrarily small cost, uniformly over decision rules.
\end{lemma}
\begin{proof}
Couple $\vartheta_m\to\vartheta$ almost surely (Skorokhod). For any rule $T$, the difference of its Bayes risks under $\pi_m$ and $\pi$ is at most $\norm\ell_\infty\E\norm{N(\vartheta_m,1)-N(\vartheta,1)}_{L_1}+\E\,\omega_\ell(s\abs{\vartheta_m-\vartheta})$, where $\omega_\ell$ is the modulus of continuity of $\ell$; both terms tend to $0$ by dominated convergence, uniformly in $T$. Taking infima over $T$ gives the claim.
\end{proof}

\begin{lemma}[clipping]\label{lem:clip}
On a finite parameter set $I$ with target $\tau$, $R:=\max_I\abs\tau$, and $M\ge(2R)^2$: $\inf_T\max_{h\in I}\E_h(T-\tau)^2=\inf_T\max_{h\in I}\E_h\min\{(T-\tau)^2,M\}$ (truncate $T$ to $[-R,R]$; this does not increase $\abs{T-\tau}$).
\end{lemma}

\begin{lemma}[finite-experiment transfer for a nonlinear target]\label{lem:transfer}
Let $I\subseteq\T$ be finite and nonempty, $\ell:\R\to[0,+\infty]$ nonnegative, lower semicontinuous and bowl-shaped, and $T_n$ any estimator sequence. Then
\[
\liminf_{n}\ \max_{h\in I}\E_{\Pn h}\ell\big(\sqrt n\{T_n-\Psi(P_{n,h})\}\big)\ \ge\ \inf_T\max_{h\in I}\E_h\,\ell\big(T-\gamma(h)\big),
\]
the infimum over randomized statistics $T$ of the Gaussian-shift limit experiment. Consequently, \cref{thm:lam-sq,thm:lam-conv} give the sequence-level lower bound
\[
\sup_I\liminf_n\max_{h\in I}\E_{n,h}\ell(\cdot)
\ge\mathcal M_\ell(B)=\mathcal N_\ell(B)
\]
where the outer supremum is over finite nonempty $I\subseteq\Hk\times B$. Lower semicontinuity cannot be dropped (\cref{app:examples}(d)).
\end{lemma}
\begin{proof}
Pass to a subsequence along which the $\liminf$ of the maximal risk is attained as a limit. Let $R:=\max_{h\in I}\abs{\gamma(h)}+1$; since $c_n(h)\to\gamma(h)$ for each of the finitely many $h\in I$, $c_n(h)\in[-R,R]$ eventually. Put $\bar S_n:=\max(-R,\min(\sqrt n\{T_n-\Psi(P)\},R))$; clipping to an interval containing $c_n(h)$ does not increase the bowl-shaped loss of $\bar S_n-c_n(h)$ relative to $S_n-c_n(h)$, so it suffices to bound the risks of $\bar S_n-c_n(h)$. Let $\Lambda_n:=(\Lambda_n(h))_{h\in I}$ be the vector of log-likelihood ratios of $\Pn h$ to $P^n$ (\cref{sec:setting}). By \eqref{eq:LAN} and tightness, along a further subsequence $(\bar S_{n_k},\Lambda_{n_k})\wto(\bar S,\Lambda)$ under $P^{n_k}$ with $\Lambda(h)=\Delta(h)-\tfrac12\norm h^2$ for the isonormal $\Delta$. Le Cam's third lemma gives $\bar S_{n_k}-c_{n_k}(h)\wto\bar S^{(h)}-\gamma(h)$ under $P^{n_k}_{n_k,h}$, where $\bar S^{(h)}$ has law $B\mapsto\E[\ind_B(\bar S)e^{\Lambda(h)}]$. Since $\ell$ is nonnegative and lower semicontinuous, the Portmanteau theorem, applied directly or first to $\ell\wedge m$ and then with $m\uparrow\infty$, yields
\[
\liminf_k\E_{n_k,h}\ell\big(\bar S_{n_k}-c_{n_k}(h)\big)\ \ge\ \E\big[\ell(\bar S-\gamma(h))e^{\Lambda(h)}\big]\qquad(h\in I)
\]
(for bounded continuous $\ell$ this is an equality; for merely lower semicontinuous $\ell$ it need not be). Conditioning $\bar S$ on the finitely many coordinates $(\Delta(h))_{h\in I}$ defines a randomized rule $T$ of the limit experiment whose risk at $h$ is $\E_h\ell(T-\gamma(h))=\E[\ell(\bar S-\gamma(h))e^{\Lambda(h)}]$. Hence the limit of the maximal risk along the subsequence is at least $\max_{h\in I}\E_h\ell(T-\gamma(h))\ge\inf_T\max_{h\in I}\E_h\ell(T-\gamma(h))$. Versions of this transfer are stated by \citet{HiranoPorter2012} and \citet{Song2014}.
\end{proof}

\begin{lemma}[finite minimax theorem]\label{lem:minimax}
For a finite parameter set $I$ and randomized rules whose risk vectors form a convex set $\mathcal R\subseteq[0,\infty)^I$: $\inf_{r\in\mathcal R}\max_hr_h=\sup_{\pi\in\Delta(I)}\inf_{r\in\mathcal R}\sum_h\pi_hr_h$.
\end{lemma}
\begin{proof}
$\ge$ is trivial. If $c<\inf_r\max_hr_h$, the convex set $\mathcal R$ is disjoint from the open convex set $\{x:x_h<c\ \forall h\}$; a separating hyperplane with normal $\pi\ne0$ satisfies $\sum\pi_hr_h\ge\sum\pi_hx_h$ for all $r\in\mathcal R$ and all $x$ in the orthant, which forces $\pi\ge0$, and normalizing to $\Delta(I)$ gives $\inf_r\sum\pi_hr_h\ge c$ \citep[cf.][]{Ferguson1967}.
\end{proof}

\begin{lemma}[reduction to finite subexperiments]\label{lem:global-minimax}
Let $(Q_\theta)_{\theta\in\Theta}$ be an experiment dominated by a probability measure $Q_0$, where $\Theta$ is nonempty, and let $\tau:\Theta\to\R$. For any nonnegative, lower semicontinuous bowl-shaped loss $\ell:\R\to[0,+\infty]$,
\[
\inf_T\sup_{\theta\in\Theta}\E_\theta\ell(T-\tau(\theta))
=\sup_{\substack{I\subseteq\Theta\\0<|I|<\infty}}
  \inf_T\max_{\theta\in I}\E_\theta\ell(T-\tau(\theta)),
\]
where $T$ ranges over all randomized real-valued decision rules. The parameter set and sample space need not be separable.
\end{lemma}
\begin{proof}
Let $K=[-\infty,+\infty]$, with the compact metric topology induced by $a\mapsto\arctan a$, and let $L_\infty=\lim_{r\to\infty}\ell(r)\in[0,+\infty]$. For each $\theta$, extend $a\mapsto\ell(a-\tau(\theta))$ to $K$ by giving both endpoints the value $L_\infty$; the extension $l_\theta$ is nonnegative and lower semicontinuous. The limit at either endpoint is $L_\infty$ because $\ell$ is symmetric and nondecreasing in magnitude.

Write $E=L_1(Q_0;C(K))$, the space of Bochner-integrable $C(K)$-valued functions, and identify a probability kernel $q$ from the sample space to $K$ with
\[
\Lambda_q(f)=\int\!\int_K f(x,a)q(x,da)Q_0(dx),\qquad f\in E.
\]
The following standard compactness argument for probability kernels (cf.\ \citealp[Section 10]{Balder1995}) shows that these functionals form a weak-* compact subset $\mathcal K$ of $E^*$. To verify this assertion, consider all positive linear functionals $\Lambda\in E^*$ satisfying
$\Lambda(g\mathbf1)=\int g\,dQ_0$ for every $g\in L_1(Q_0)$.
They have norm one, and positivity and these normalization equalities are weak-* closed conditions, so the Banach--Alaoglu theorem gives compactness. 

Each such functional is represented by a probability kernel: for a countable uniformly dense rational vector subspace $D\subset C(K)$ containing $\mathbf1$, the maps $g\mapsto\Lambda(g\varphi)$ have Radon--Nikodym derivatives $u_\varphi\in L_\infty(Q_0)$. Outside one measurable $Q_0$-null set, these derivatives are rational-linear in $\varphi$, satisfy $u_{\mathbf1}=1$, and satisfy $\min_K\varphi\le u_\varphi\le\max_K\varphi$ simultaneously for every $\varphi\in D$. Thus $\varphi\mapsto u_\varphi(x)$ extends to a positive, norm-one linear functional on $C(K)$. The Riesz representation theorem gives a probability measure $q(x,\cdot)$. Its integrals against continuous functions are measurable, first for $D$ and then by uniform approximation; hence $q$ is a probability kernel on the compact metric space $K$. Set $q(x,\cdot)=\delta_0$ on the exceptional set. Approximation by simple $C(K)$-valued functions gives $\Lambda=\Lambda_q$. This argument applies to an arbitrary underlying probability space.

Put $p_\theta=dQ_\theta/dQ_0$. Choose nonnegative continuous functions $l_{\theta,j}\uparrow l_\theta$ on $K$, each bounded. Then
\[
R_\theta(q)=\int p_\theta(x)\!\int_K l_\theta(a)q(x,da)Q_0(dx)
=\sup_j\Lambda_q(p_\theta l_{\theta,j})
\]
is lower semicontinuous on $\mathcal K$, since $p_\theta l_{\theta,j}\in E$. Let $c$ denote the right-hand side of the asserted equality. If $c=+\infty$, the result is immediate. Otherwise fix $r>c$. The closed sets $\{q\in\mathcal K:R_\theta(q)\le r\}$ have the finite intersection property: for any finite set of parameters, the definition of $c$ supplies a real-valued rule with all its risks below $r$. Compactness gives one kernel with $R_\theta(q)\le r$ for every $\theta$. Move any mass at $-\infty$ and $+\infty$ to action $0$. This cannot increase any risk, because $\ell(-\tau(\theta))\le L_\infty$, and yields a real-valued randomized rule. Thus the left-hand side is at most $r$; let $r\downarrow c$. The reverse inequality is immediate.
\end{proof}

In the Gaussian-shift experiment, on the sigma-field generated by the isonormal observations, $dQ_h/dQ_0=\exp\{\Delta(h)-\norm h^2/2\}$. This identity follows first on finite coordinate cylinders and then on their generated sigma-field; it does not require a separable tangent space. Thus \cref{lem:global-minimax} applies to the experiments in Section~\ref{sec:structure}.

The allowance of non-lower-semicontinuous bowl-shaped losses in \cref{thm:lam} follows by replacing $\ell$ with its lower semicontinuous envelope $\ell_*\le\ell$. A nondecreasing function of $\abs x$ differs from its envelope at at most countably many radii, so $\int\ell_*\,dN(0,V_V)=\int\ell\,dN(0,V_V)$ when $V_V>0$. When $V_V=0$, the asserted lower bound is $\ell(0)$, which is already a pointwise lower bound on every loss. Thus no additional moment or continuity assumption is needed for that theorem.

\subsection{Further details on the bounds}
\begin{example}[an additive, continuous cone without bounded extension]\label{ex:cone}
On $H=\R\oplus\ell_2$ let $C:=\{(t,x):x_n\ge0,\ \sum_nn^2x_n\le t\}$ and $q(t,x):=\sum_nnx_n$. Then $C$ is a closed convex cone, $q$ is additive, positively homogeneous and continuous on $C$ (the tail $\sum_{n>N}nx_n\le t/(N+1)$ is uniformly small on bounded sets, and $0\le q\le t\le\norm{(t,x)}$), and $q$ extends to a continuous positively homogeneous $\gamma$ on $H$ (Tietze on the unit sphere, then radially). Any linear extension $\ell$ of $q|_C$ to $\spn C$ satisfies $\ell(e_0)=q(e_0)=0$ and $\ell(e_0+n^{-2}e_n)=1/n$, hence $\ell(e_n)=n$: unbounded. By \cref{thm:main}(1) no estimator is regular along $C$ although $\gamma$ is additive there.
\end{example}

\begin{remark}[Gaussian factors and regularity sets]\label{rem:what}
Regularity at a nonlinear point can hold along subspaces or cones on which $\gamma$ has a bounded linear extension. The convolution theorem then uses the representer on that set and identifies an independent residual. The factor guaranteed by the regularity family is $N(0,v(T_n))$, which need not be the largest Gaussian divisor of $L$. For example, let $\Psi(\theta)=\theta_1+\abs{\theta_2}$ in a bivariate normal mean model and $T_n=\bar X_1+\bar Y$, where $Y$ is independent pure noise with variance one. The agreement set is $\{h_2=0\}$, $v(T_n)=1$, $L=N(0,2)$ and $M=N(0,1)$. The core bound applies to every estimator regular along a maximal linearity subspace. A cone's bound depends on its span and the linear functional agreeing with the derivative there (\cref{cor:onesided}). Subspaces outside $\Nodd(\gamma)$ cannot support regularity (\cref{cor:impossible}), whereas one-sided cones outside this set can: for $\gamma(h)=\abs h$, the sample mean is regular along $[0,\infty)$ even though $\Nodd(\gamma)=\{0\}$.
\end{remark}

\begin{remark}[relation to existing impossibility results]\label{rem:impossibility}
\citet{HiranoPorter2012} analyze both locally asymptotically unbiased estimation and regular estimation through the Gaussian shift experiment. Their equivariance proof uses characteristic functions without moment assumptions on the estimator. \Cref{lem:tilt} localizes the necessary linearity condition to subspaces and cones; \cref{thm:main} also establishes bounded linear extension and identifies the corresponding Gaussian convolution factor. Here locally asymptotically unbiased (LAU) follows their limiting-law convention: each centered limit distribution has mean zero. Under this convention, an integrable regular limit can be recentered to satisfy LAU. The theorem of \citet{vdV1991} is often paraphrased as ``regular estimability implies differentiability''; \citet{Pfanzagl2000} makes explicit that it assumes the existence of the directional limits, in which case its content is \cref{thm:main}(1) with $C=\T$, and gives a counterexample to the unqualified paraphrase; \cref{prop:rotating} (\cref{app:rotating}) is a further such example, without a priority claim. \citet{Kaji2021} proves impossibility of equivariant-in-law estimation for weakly regular parameters by a continuity argument from Le Cam's third lemma.
\end{remark}

\begin{remark}\label{rem:core-add}
The inclusion in \cref{thm:core}(4) can be strict: for $\gamma(h)=h_1^2/\norm h$ on $\R^2$, $\Nodd=\{h_1=0\}$ is itself a linearity subspace ($\gamma\equiv0$ there), so $\Hk=\{h_1=0\}$ while $\Afam(\gamma)=\{0\}$; a polyhedral example is given in \cref{prop:addpl}(4) (\cref{app:geometry}). The core is the object relevant to regularity (\cref{sec:convolution}); the additivity space is what makes the splitting of \cref{sec:structure} work (\assS); \cref{thm:clarke} identifies it, for Lipschitz $\gamma$, with the lineality space of the support function of Clarke's subdifferential, i.e.\ with $(\partial^C\gamma(0)-\partial^C\gamma(0))^\perp$.
\end{remark}

\begin{remark}\label{rem:kersplit}
The case split in (5) is by $\one\in\range(\Sigma)$, not by $\one\in\ker\Sigma$: for $\Sigma=\operatorname{diag}(1,0)$ one has $\one\notin\ker\Sigma$ and $\one\notin\range\Sigma$, the formula $1/\one^\top\Sigma^+\one$ returns $1$, but $w=(0,1)$ gives $0$ (one active component has a degenerate score, so the tied value is locally known).
\end{remark}

\begin{remark}[variance reduction from pooling]\label{rem:dividend}
For independent active components with $\Sigma=\operatorname{diag}(\sigma_a^2)$, all $\sigma_a>0$, $\Vk=(\sum_a\sigma_a^{-2})^{-1}<\min_a\sigma_a^2$ when $K\ge2$: the bound for the common value is strictly below the bound for estimating any single tied component.
\end{remark}

\begin{remark}[piecewise-linear reduced derivatives]\label{rem:piece}
If $\gamma=g\circ\ell$ with $g$ piecewise linear (\cref{prop:pl}) and $\phi=\ell^*b$ for a piece $b$ of $g$, then $\mathcal E(\phi)=\ell^{-1}(\mathcal E(b^\top\cdot))$: an estimator that tracks one linear piece is regular exactly along the directions and one-sided cones on which that piece remains active (``the estimator of arm $j$ is regular as long as arm $j$ stays the median''; ``$\hat\mu_a$ is regular as long as $a$ stays best'', \cref{cor:onesided}).
\end{remark}

\begin{remark}\label{rem:tradeoff}
For a directional derivative that is neither sublinear nor superlinear, different maximal linearity subspaces can have different bounds $\{V_V:V\in\Lfam_{\max}\}$ and admit different regular estimator classes. The core bound applies to every such class, while a larger subspace can require a larger variance. The median example in \cref{ex:median} includes enlargements with and without an increase in the bound; \cref{cor:onesided} gives the corresponding comparison for one-sided regularity.
\end{remark}

\begin{remark}[simplex criterion]\label{rem:simplex}
For a maximum of finitely many branches with $\Sigma$ nonsingular, $\psik\in\LamP$ iff $\Sigma^{-1}\one\ge0$ componentwise. For $K=2$: iff $\Sigma_{12}\le\min(\Sigma_{11},\Sigma_{22})$. For general $K$: $w_j^*\ge0$ iff $\beta_j^\top\one\le1$, where $\beta_j:=\Sigma_{-j,-j}^{-1}\Sigma_{-j,j}$ are the population regression coefficients of $\lambda_j$ on the other active scores (block inversion: $(\Sigma^{-1}\one)_j=(1-\beta_j^\top\one)/(\Sigma_{jj}-\Sigma_{j,-j}\beta_j)$ with positive denominator): component $j$ receives nonnegative efficient weight iff its score is not more than a unit-sum combination of the others. Sufficient: $\Sigma$ an M-matrix. The pairwise reading ``$\Sigma_{ij}\le\min(\Sigma_{ii},\Sigma_{jj})$ for all pairs'' is not sufficient for $K\ge3$: $\Sigma=\left(\begin{smallmatrix}1&.9&.9\\.9&1&.7\\.9&.7&1\end{smallmatrix}\right)$ is positive definite with $w^*=(-1,1,1)$. By \cref{prop:blind}, $\Sigma^{-1}\one\ge0$ is also the condition for nonpositive drift of the core-efficient estimator under every tie-breaking alternative.
\end{remark}

\subsection{Confidence intervals and the two-programme risk comparison}\label{app:ci-proofs}

\begin{proof}[Proof of \cref{cor:ci-local}]
By \cref{thm:efficiency}, $T_n$ is asymptotically linear with influence function $\tilde\psi_V$. The local limit in \cref{thm:al} is therefore $N(\delta(h),s^2)$. Contiguity carries $\hat s_n\to s$ from $P^n$ to $\Pn h$ for every fixed path. Slutsky's theorem gives the studentized limit $N(\delta(h)/s,1)$. Its distribution is continuous, so the interval and lower-bound probabilities converge to the stated expressions. For $z=z_{1-\alpha/2}>0$, the function $a\mapsto\Phi(z-a)-\Phi(-z-a)$ is even and strictly decreasing in $\abs a>0$, establishing the two-sided comparison. The one-sided inequality follows from monotonicity of $\Phi$. On $V$, the representer identity gives $\delta(h)=0$.
\end{proof}

\begin{proof}[Verification of \eqref{eq:two-programme-minimax}]
For $T(U,V)=U+\abs V$, independence and $\E G=0$ give risk $1/2+\E(\abs{d+H}-\abs d)^2\le1/2+\E H^2=1$. For the lower bound restrict to $c=d=t>0$. Then $U,V$ are independent $N(t,1/2)$ observations, the target is $2t$, and the Fisher information for $t$ is $4$. Choose a smooth prior density $\pi_R(t)=R^{-1}\pi(t/R)$ supported on $(R,2R)$, where $\pi$ vanishes at the endpoints of $(1,2)$ and has finite prior information $J=\int(\pi')^2/\pi$. Integration by parts under this prior gives $\E[(T-2t)\,\partial_t\log\{p_t(U,V)\pi_R(t)\}]=2$. The score has second moment $4+J/R^2$, so Cauchy--Schwarz bounds the Bayes squared risk below by $4/(4+J/R^2)$. The identity is immediate for bounded $T$ and extends by truncation whenever its Bayes squared risk is finite; infinite risk already satisfies the bound. Letting $R\to\infty$ proves the minimax lower bound $1$ for arbitrary estimators, including randomized ones. This is a calculation within the Gaussian experiment, not a uniform attainment theorem for the original model.
\end{proof}

\section{Counterexamples referred to in the text}\label{app:examples}
\begin{enumerate}[label=(\alph*)]
\item \emph{\cref{cor:onesided}.} With a one-dimensional tangent space and $\gamma(h)=\max(h,2h,3h)$, component $2$ is not a vertex and $C_2=\{0\}$. With $\lambda_1=\lambda_2=1$, $\lambda_3=-1$ the strict cone of component $1$ would be empty without deduplication. With $\Sigma=\left(\begin{smallmatrix}1&1\\1&2\end{smallmatrix}\right)$, $\Sigma^{-1}\one=(1,0)$, so $\psik=\lambda_1$ and the core-efficient estimator is $\hat\mu_1$ itself, regular along $C_1$.
\item \emph{\cref{prop:fixed}(a).} For $g(z)=z_1+\abs{z_2-z_3}$ with $\Sigma=I_3$ (translation-equivariant, \assS holds, $F^\perp=\R\one\subsetneq\Hk=\{h_2=h_3\}$) the uniform average has drift $-2/3$ at $h=e_1$ although the $\Hk^\perp$-component of $h$ vanishes; the decomposed formula is not valid there, the direct formula is.
\item \emph{\cref{prop:fixed}(b).} With $\Sigma=\left(\begin{smallmatrix}1&1\\1&2\end{smallmatrix}\right)$, $W=(0,Z_2-Z_1)$ and deterministic selection of arm $1$ gives $M=\delta_0$ although $W\not\equiv0$.
\item \emph{\cref{lem:transfer}.} For $P_\theta=N_2(\theta,I_2)$, $\Psi(P_\theta)=\theta_1+\abs{\theta_2}-\theta_1^3$, $I=\{\pm e_1\}$, $T_n\equiv0$ and the upper semicontinuous loss $\ell(u)=\ind\{\abs u\ge1\}$, the finite-sample risks are $0$ (the local targets are $\pm(1-1/n)$) while the limit-experiment minimax risk is $\Phi(-1)>0$: the zero-loss action sets $(0,2)$ and $(-2,0)$ are disjoint, so every rule induces a test of $N(1,1)$ against $N(-1,1)$. The lower semicontinuous coverage loss $\ind\{\abs u>t\}$ is covered by the lemma.
\end{enumerate}

\section{Supporting geometry: differences of sublinear functions and the Clarke identity}\label{app:geometry}

Every continuous piecewise-linear positively homogeneous function on $\R^K$ is a difference of two sublinear functions \citep{BartelsKuntzScholtes1995}. The next proposition expresses the linearity condition in terms of their supporting sets and gives a subspace contained in the core.

\begin{proposition}[differences of sublinear derivatives]\label{prop:dc}
Let $\gamma=\sigma_{\Lambda_1}-\sigma_{\Lambda_2}$ with $\Lambda_i\subseteq\overline\T$ nonempty, closed, bounded and convex, $\sigma_\Lambda(h):=\max_{\lambda\in\Lambda}\ip\lambda h$, $\wid_\Lambda(h):=\sigma_\Lambda(h)+\sigma_\Lambda(-h)$, and $F_i:=\clspan(\Lambda_i-\Lambda_i)$. Then (1) $\Nodd(\gamma)=\{h:\wid_{\Lambda_1}(h)=\wid_{\Lambda_2}(h)\}$; (2) a linear subspace $V$ is a linearity subspace iff $\Pi_{\overline V}\Lambda_1=\Pi_{\overline V}\Lambda_2+v$ for some $v\in\overline V$, and then $\tilde\psi_V=v$, $V_V=\norm v^2$; (3) $(F_1+F_2)^\perp\cap\T\subseteq\Afam(\gamma)\subseteq\Hk(P)$, both inclusions possibly strict.
\end{proposition}
\begin{proof}
(1) $\gamma(h)+\gamma(-h)=\wid_{\Lambda_1}(h)-\wid_{\Lambda_2}(h)$. (2) For $h\in V$, $\sigma_{\Lambda_i}(h)=\sigma_{\Pi_{\overline V}\Lambda_i}(h)$, and $\Pi_{\overline V}\Lambda_i$ is weakly compact, hence closed, convex and bounded in $\overline V$; support functions on $\overline V$ determine such sets, and $\sigma_A-\sigma_B=\ip v\cdot$ iff $A=B+v$. (3) For $w\in(F_1+F_2)^\perp$, $\sigma_{\Lambda_i}(h+w)=\sigma_{\Lambda_i}(h)+c_i(w)$ with $c_i(w)$ the common value of $\ip\lambda w$ on $\Lambda_i$.
\end{proof}

\begin{proposition}[piecewise-linear $g$: agreement cones]\label{prop:pl}
Let $g:\R^K\to\R$ be continuous, positively homogeneous and piecewise linear with linear pieces $\ell_1,\dots,\ell_m$ on a finite fan of closed cones covering $\R^K$, and let $\mathcal E(\ell_i):=\{z:g(z)=\ell_i(z)\}$. Then $\Lfam(g)=\{U\text{ linear}:U\subseteq \mathcal E(\ell_i)\text{ for some }i\}$.
\end{proposition}
\begin{proof}
If $U\subseteq \mathcal E(\ell_i)$ then $g=\ell_i$ on $U$. Conversely, if $g$ is linear on $U$, some $U\cap C_i$ has nonempty interior relative to $U$ (finite closed cover); there $g=\ell_i$, and two linear functions on $U$ agreeing on a relatively open set agree on $U$.
\end{proof}

\begin{proposition}[additivity space of a piecewise-linear derivative; \assS]\label{prop:addpl}
Let $g$ be as in \cref{prop:pl} and let $\ell_1,\dots,\ell_m$ now denote the distinct pieces whose agreement sets have nonempty interior (full-dimensional pieces). Then:
\begin{enumerate}[label=(\arabic*)]
\item $\Afam(g)=\bigcap_{i,j}\ker(\ell_i-\ell_j)$;
\item $\partial^Cg(0)=\conv\{\ell_1,\dots,\ell_m\}$ \citep[Theorem 2.5.1]{Clarke1983}, hence $\Afam(g)=(\partial^Cg(0)-\partial^Cg(0))^\perp$;
\item \assS of \cref{sec:structure} ($\Hk(g)=\Afam(g)$) holds iff $\Hk(g)\subseteq\bigcap_{i,j}\ker(\ell_i-\ell_j)$;
\item \assS can fail for polyhedral derivatives: $g(z)=z_2+\abs{z_1}+\max\{0,\min(\abs{z_1},-z_2)\}$ on $\R^2$ has full-dimensional pieces $z_2\pm z_1$, $z_2\pm2z_1$, $\pm z_1$, odd cone $\Nodd(g)=\{z_1=0\}$ (a linearity subspace, so $\Hk(g)=\{z_1=0\}$), but $\Afam(g)=\{0\}$.
\end{enumerate}
\end{proposition}
\begin{proof}
(1) The full-dimensional pieces determine $g$ (their cones cover $\R^K$ up to a closed nowhere-dense set; $g$ is continuous), so along any line $c\mapsto z+cw$ the function $g(z+cw)$ is continuous and piecewise linear in $c$ with slopes in $\{\ell_i(w)\}$. If all $\ell_i(w)$ equal $\theta$, then $g(z+cw)=g(z)+c\theta$ and $g(w)=\theta$, so $w\in\Afam(g)$. Conversely, for $w\in\Afam(g)$ and $z$ interior to the cone of piece $i$, $g(z+cw)=\ell_i(z)+c\ell_i(w)$ for small $c$, while additivity gives $g(z)+cg(w)$; hence $\ell_i(w)=g(w)$ for every $i$. (2) Clarke's subdifferential of a piecewise-linear function at $0$ is the convex hull of the gradients of the pieces active on open sets accumulating at $0$; by homogeneity every full-dimensional cone accumulates at $0$; the lineality space of $\sigma_{\conv\{\ell_i\}}$ is $\{w:\ell_i(w)\text{ constant in }i\}$ (\cref{thm:kink}(2)). (3) From $\Afam\subseteq\Hk$. (4) On $z_2\ge0$: $g=z_2+\abs{z_1}$; on $z_2<0$, $\abs{z_1}\le\abs{z_2}$: $g=z_2+2\abs{z_1}$; on $z_2<0$, $\abs{z_1}>\abs{z_2}$: $g=\abs{z_1}$; $g(z)+g(-z)=2\abs{z_1}+\min(\abs{z_1},\abs{z_2})\ind\{z_2\ne0\}$, so $\Nodd=\{z_1=0\}$, on which $g(0,t)=t$; the common kernel of the piece differences is $\{0\}$.
\end{proof}

\begin{theorem}[the Clarke identity]\label{thm:clarke}
Let $\gamma$ be positively homogeneous and Lipschitz on a Hilbert space $H$, let $\gamma^\circ(0;w):=\limsup_{x\to0,\,t\downarrow0}t^{-1}\{\gamma(x+tw)-\gamma(x)\}$ be Clarke's generalized directional derivative at $0$, and $\partial^C\gamma(0):=\{\lambda:\ip\lambda w\le\gamma^\circ(0;w)\ \forall w\}$. Then
\begin{enumerate}[label=(\arabic*)]
\item $\gamma^\circ(0;w)=\sup_{y\in H}\{\gamma(y+w)-\gamma(y)\}$;
\item $\Afam(\gamma)=\{w:\gamma^\circ(0;w)+\gamma^\circ(0;-w)=0\}=(\partial^C\gamma(0)-\partial^C\gamma(0))^\perp$;
\item \assS holds iff $\Hk(\gamma)\subseteq(\partial^C\gamma(0)-\partial^C\gamma(0))^\perp$.
\end{enumerate}
\end{theorem}
\begin{proof}
(1) By homogeneity, $t^{-1}\{\gamma(x+tw)-\gamma(x)\}=\gamma(x/t+w)-\gamma(x/t)=:D_w(x/t)$; as $(x,t)\to(0,0^+)$ the point $x/t$ ranges over all of $H$ (take $x=ty$), so the $\limsup$ equals $\sup_yD_w(y)\le\Lip(\gamma)\norm w$. (2) By (1) and the substitution $y'=y+w$, $\gamma^\circ(0;w)+\gamma^\circ(0;-w)=\sup_yD_w(y)-\inf_yD_w(y)$, which vanishes iff $D_w$ is constant, i.e.\ iff $\gamma(y+w)-\gamma(y)=\gamma(w)$ for all $y$, i.e.\ iff $w\in\Afam(\gamma)$. The function $p(w):=\sup_yD_w(y)$ is positively homogeneous (scale $y$), subadditive ($D_{w+w'}(y)=D_w(y+w')+D_{w'}(y)$) and bounded by $\Lip(\gamma)\norm w$, hence by \cref{thm:kink}(1) the support function of a closed bounded convex set, which is $\partial^C\gamma(0)$ by definition; the lineality space of a support function is $(\Lambda-\Lambda)^\perp$ (\cref{thm:kink}(2)). (3) From (2) and \cref{thm:core}(4).
\end{proof}
For piecewise-linear $g$ this is \cref{prop:addpl}(1)--(2); for sublinear $\gamma$, $\partial^C\gamma(0)=\LamP$ and $\Afam=\Hk$. The core itself is not the lineality space of a support function in general (\cref{prop:addpl}(4)).

\section{Order statistics at ties}\label{app:order}
\begin{theorem}[order statistics]\label{thm:order}
Let $g_k(z):=z_{(k)}$ be the $k$-th largest coordinate on $\R^K$, $K\ge2$, $1\le k\le K$, $m:=\min(k-1,K-k)$, and for a subspace $U$ and index $j$ write $d^{(j)}_i:=(z_i-z_j)|_U$, $i\ne j$.
\begin{enumerate}[label=(\arabic*)]
\item $\Nodd(g_k)=\{z:z_{(k)}=z_{(K+1-k)}\}$; $g_k$ is odd iff $K=2k-1$.
\item $\Afam(g_k)=\Hk(g_k)=\R\one$; \assS holds for every $k$; in the reduced model the core is $F^\perp$ with core bound $1/\one^\top\Sigma^{-1}\one$, the same for all $k$.
\item Every linearity subspace has dimension $\le1+m$, and $1+m$ is attained: for any $j$, any $m$ disjoint pairs $\{i_s,i'_s\}$ of other indices, any $\beta_s\in[0,1]$, and $R$ the remaining indices, $U:=\{z:z_j=\beta_sz_{i_s}+(1-\beta_s)z_{i'_s}\ (s\le m),\ z_r=z_j\ (r\in R)\}$ is a maximal linearity subspace of dimension $1+m$.
\item (Gordan criterion.) A subspace $U\ni\one$ is a linearity subspace iff for some $j$ no $m+1$ of the contrasts $d^{(j)}_i$ are simultaneously positive on $U$, iff for some $j$ every $(m+1)$-subset $S\subseteq\{i\ne j\}$ satisfies $0\in\conv\{d^{(j)}_i:i\in S\}$ in $U^*$.
\end{enumerate}
\end{theorem}
\begin{proof}
(1) $g_k(-z)=-z_{(K+1-k)}$. (3), upper bound: $g_k(z+c\one)=g_k(z)+c$ gives $\R\one\subseteq\Afam\subseteq\Hk$, so $U+\R\one\in\Lfam$ (\cref{thm:core}(3)) and it suffices to treat $U\ni\one$. By \cref{prop:pl}, $U\subseteq \mathcal E(z_j)$ for some $j$: $\#\{i:z_i>z_j\}\le k-1$ and $\#\{i:z_i<z_j\}\le K-k$ on $U$; applying the first to $z$ and the second to $-z\in U$ gives $\#\{i:d^{(j)}_i(z)>0\}\le m$ for all $z\in U$. The functionals $d^{(j)}_i|_U$ have common kernel $\R\one$, so they span a space of dimension $r=\dim U-1$; $r$ linearly independent ones define an onto map $U\to\R^r$, so some $z$ makes all $r$ positive; hence $r\le m$. (3), construction: on $U$, $z_j$ lies between the members of each pair, so at most $m\le k-1$ coordinates exceed it and at most $m\le K-k$ are below it; thus $z_j$ has rank $k$ and $g_k=z_j$ on $U$; $\dim U=1+m$; maximality from the bound. (2) $\R\one\subseteq\Hk$ as above. If $m=0$, $\Nodd(g_k)=\R\one$ and $\Hk\subseteq\Nodd$. If $m\ge1$ (so $K\ge3$), let $w\notin\R\one$, pick $p\ne q$ with $w_p\ne w_q$ and $j\notin\{p,q\}$, and use the construction with $(i_1,i'_1)=(p,q)$, $\beta_1=\beta$, the remaining $m-1$ pairs from the other $K-3$ indices ($2(m-1)\le K-3$). $w\in U_\beta$ forces $w_j=\beta w_p+(1-\beta)w_q$, true for at most one $\beta$; another $\beta$ gives a maximal linearity subspace not containing $w$, so $w\notin\Hk$ (\cref{thm:core}(3)). The reduced-model statements follow from \cref{lem:reduction} and \cref{thm:kink}(3). (4) The rank-$k$ condition on $U=-U$ is equivalent to $\#\{i:d^{(j)}_i(z)>0\}\le m$ for all $z\in U$, i.e.\ for every $(m+1)$-subset $S$ the system $\{d^{(j)}_i>0:i\in S\}$ is infeasible on $U$, which by Gordan's alternative holds iff $\sum_{i\in S}\mu_id^{(j)}_i=0$ on $U$ for some $\mu\ge0$, $\mu\ne0$.
\end{proof}

\begin{proposition}[classification for $m=1$; a non-pairing family for $m\ge2$]\label{prop:m1}
\begin{enumerate}[label=(\arabic*)]
\item If $m=1$ ($k\in\{2,K-1\}$, $K\ge3$), $\Lfam_{\max}(g_k)$ consists exactly of the pairing planes $U_{j;p,q;\beta}=\{z_j=\beta z_p+(1-\beta)z_q,\ z_r=z_j\ (r\notin\{j,p,q\})\}$, $\beta\in[0,1]$, $j\notin\{p,q\}$.
\item If $m\ge2$ there are maximal linearity subspaces not of pairing form; e.g.\ for the median of five, $\{z_5=z_3,\ z_3=\tfrac13(z_1+z_2+z_4)\}$ is a $3$-dimensional linearity subspace in which $z_3$ is the mean of three others.
\end{enumerate}
\end{proposition}
\begin{proof}
(1) A maximal $U$ is $2$-dimensional, contains $\one$, and on the one-dimensional quotient $U/\R\one$ each $d^{(j)}_i$ is a scalar $c_i$ times a fixed functional; pairwise non-positivity of products means at most one $c_i>0$ and at most one $c_i<0$, the rest zero; this is the pairing form (with $\beta\in\{0,1\}$ when only one is nonzero); conversely every $U_{j;p,q;\beta}$ is a linearity subspace of dimension $2=1+m$. (2) On the displayed subspace $z_3$ is the mean of $z_1,z_2,z_4$ and $z_5=z_3$, so $z_3$ has rank $3$; the subspace is $3$-dimensional, hence maximal; the three nonzero contrasts sum to zero without being antipodal pairs.
\end{proof}
For each candidate subspace $U$, the Gordan criterion in \cref{thm:order}(4) requires checking each $(m+1)$-subset of the $K-1$ coordinate differences. The resulting maximal subspaces can vary continuously, as \cref{prop:m1} shows.

The examples above also show why the core and the additivity space must be distinguished. For $\gamma(h)=h_1^2/\norm h$, there is a unique maximal linearity subspace, but the two spaces differ (\cref{rem:core-add}); for $\gamma(h)=h_1^3/\norm h^2$, every line is a maximal linearity subspace, the core is $\{0\}$, and $\Vstar=1$ is attained only on the $e_1$-axis. Proposition~\ref{prop:addpl}(4) shows that \assS can fail even for a piecewise-linear derivative.

\subsection{Information bounds and estimators for the three-programme median}\label{app:median-details}
Retain the setting of \cref{ex:median}: $\mu(P)=\theta\one$, $\gamma=\med\circ\ell$, $\ell(h)=(\ip{\lambda_j}h)_{j=1}^3$, and the Gram matrix $\Sigma$ is positive definite. Here $\one=(1,1,1)^\top$, and $e_j$ denotes the $j$-th coordinate vector. For $0\ne a\perp\one$, let $U_a=\{z:a^\top z=0\}$ and $V_a=\ell^{-1}(U_a)$. An index $j$ is admissible for $a$ if $a_j\ne0$ and $a_ja_i\le0$ for every $i\ne j$. Such an index always exists. If all entries of $a$ are nonzero it is the unique index carrying one sign; if exactly two entries are nonzero, either of those indices is admissible. For any admissible $j$, the equation defining $U_a$ expresses $z_j$ as a convex combination of the other two coordinates, so $\med(z)=z_j$ on $U_a$. In the two-index case the two admissible coordinates agree on the plane. By \cref{thm:order,prop:m1}, these planes are exactly the maximal linearity spaces of the median.

For an admissible $j$, \cref{lem:reduction} gives
\[
\tilde\psi_{V_a}
=\lambda_j-\frac{(\Sigma a)_j}{a^\top\Sigma a}
 \sum_{i=1}^3a_i\lambda_i,
\qquad
V_{V_a}=\Sigma_{jj}-\frac{(\Sigma a)_j^2}{a^\top\Sigma a}
\ge\Vk.
\]
The formula does not depend on the admissible index chosen. When $a\propto e_j-e_k$, it reduces to the variance of the generalized least-squares pool of programmes $j,k$,
\[
V_{V_a}=\frac{\Sigma_{jj}\Sigma_{kk}-\Sigma_{jk}^2}
 {\Sigma_{jj}+\Sigma_{kk}-2\Sigma_{jk}}.
\]
For $\Sigma=I_3$, write the two entries other than an admissible $a_j$ as $-b,-c$ after a common change of sign, with $b,c\ge0$ and $b+c>0$. Then
\[
V_{V_a}=1-\frac{(b+c)^2}{(b+c)^2+b^2+c^2}\in[1/3,1/2].
\]
The lower endpoint occurs at $b=c$, giving a midpoint plane; the upper endpoint occurs at $bc=0$, giving a pair-tie plane. This proves the range and endpoint comparisons in \cref{ex:median}.

The largest subspace bound can also be evaluated directly from the order cones. For distinct $i,j,k$, put
\[
C_{ijk}=\{z:z_i\le z_j\le z_k\},\qquad
\mathcal C_{ijk}=\Sigma^{-1/2}C_{ijk},\qquad
Q_j=\max_{\{i,k\}=\{1,2,3\}\setminus\{j\}}
 \norm{\Pi_{\mathcal C_{ijk}}(\Sigma^{1/2}e_j)}^2,
\]
where the maximum includes both orderings of $i,k$ and the projection is Euclidean. Since the median equals $z_j$ on $C_{ijk}$, \cref{prop:vstar} implies
\[
\Vstar=\max_{j\le3}Q_j,\qquad
Q_j\le\Sigma_{jj}.
\]
Equality $Q_j=\Sigma_{jj}$ holds exactly when $\Sigma e_j$ lies in one of the two order cones with middle index $j$, or equivalently when $\Sigma_{jj}$ is the median of the $j$-th column of $\Sigma$. For $\Sigma=\operatorname{diag}(2,1,2)$, the three contributions are $(Q_1,Q_2,Q_3)=(1,2/3,1)$. Thus $\Vstar=1$ is supplied by coordinates $1$ and $3$, even though it is below their individual variances $\Sigma_{11}=\Sigma_{33}=2$.

For the estimator comparisons, suppose $\hat\mu$ is jointly asymptotically linear with component influence functions $\lambda_j$, and let $\hat\Sigma$ be a consistent covariance estimate. Define
\[
w^*=\frac{\Sigma^{-1}\one}{\one^\top\Sigma^{-1}\one},\qquad
\psik=\sum_{j=1}^3w_j^*\lambda_j,\qquad
\Vk=\norm\psik^2=\frac{1}{\one^\top\Sigma^{-1}\one}.
\]
The generalized least-squares pool $\hat T_\kappa=\sum_jw_j^*\hat\mu_j$ is core-efficient. The same is true with $w^*$ replaced by its plug-in estimate from $\hat\Sigma$: the weights sum to one exactly and all three means are equal at $P$. For a maximal space $V_a$ and any admissible $j$, \cref{thm:al} gives
\[
\hat T_\kappa\text{ is regular along }V_a
\quad\Longleftrightarrow\quad
w^*-e_j\in\spn\{a\}.
\]
Whenever this holds, $V_{V_a}=\Vk$, because the core-efficient influence function also represents the target on $V_a$.

If $w^*$ is in the open simplex, meaning $w_j^*>0$ for every $j$ and $\one^\top w^*=1$, the displayed condition selects exactly three distinct planes, with normals $a\propto w^*-e_j$, $j=1,2,3$. On these planes the additional regularity requirement has no variance cost. At the boundary $w^*=e_j$, the pool is programme $j$ itself and is regular along every $V_a$ for which $j$ is admissible, all with bound $\Vk$. For example,
\[
\Sigma=\begin{pmatrix}1&1&1\\1&2&1\\1&1&2\end{pmatrix}
\]
is positive definite and satisfies $\Sigma e_1=\one$, so $w^*=e_1$ and $\Vk=1$. In both the interior and boundary cases, the pool is not regular along any of the remaining maximal spaces. Wherever $V_{V_a}>\Vk$, \cref{prop:tradeoff} gives the strict information cost of the larger requirement.

An estimator attaining the bound for any specified maximal space is
\[
\hat T_a
=\hat\mu_j-\frac{(\hat\Sigma a)_j}{a^\top\hat\Sigma a}
 a^\top\hat\mu,
\]
with $j$ admissible for $a$. By \cref{prop:suff}, its influence function is $\tilde\psi_{V_a}$, so it is $V_a$-efficient. The estimated covariance affects no first-order term because $a^\top\mu(P)=0$. The covariance-based estimates may be set to the three-programme average on the event that $\hat\Sigma$ is not positive definite, whose probability tends to zero. The single-programme estimator $\hat\mu_j$ is regular along every $V_a$ for which $j$ is admissible, by its agreement with the median on that space (\cref{thm:al}). It is not regular along the sum of any two distinct such maximal spaces: their reduced planes sum to $\R^3$, on which the median is nonlinear.

Finally, put $W=(W_1,W_2,W_3)^\top$, with $W_j=\Delta(\lambda_j-\psik)$. The plug-in estimator $\med_j\hat\mu_j$ is regular along the core, with
\[
L=N(0,\Vk)*M,\qquad M=\Law(\med(W)),
\]
by \cref{prop:plugin} at the full tie. The residual has mean zero because $W$ is symmetric and the median is odd, so selection at the tie introduces no mean shift. Among linear score subspaces, the plug-in is regular only on subspaces contained in the core. Its residual therefore distinguishes its response from that of a linear pool even though both share the same core Gaussian bound.

\section{Regularity without directional differentiability}\label{app:rotating}
\begin{proposition}[regularity does not imply directional differentiability]\label{prop:rotating}
Let $X_1,\dots,X_n$ be i.i.d.\ $N(\theta,I_2)$, $\theta\in\R^2$. Put $\vartheta(r):=\sin\sqrt{\log^+(1/r)}$ for $r>0$, $u(\vartheta):=(\cos\vartheta,\sin\vartheta)$, $f(\theta):=\ip{u(\vartheta(\norm\theta))}\theta$ for $\theta\ne0$, $f(0):=0$, and $T_n:=f(\bar X_n)$ (the plug-in). Then:
\begin{enumerate}[label=(\arabic*)]
\item $f$ is continuous with $\abs{f(\theta)}\le\norm\theta$, real-analytic on $\{0<\norm\theta<1\}$;
\item at $\theta=0$, $f$ has no directional derivative in any direction $h\ne0$:
\[
t^{-1}f(th)=\ip{u(\vartheta(t\norm h))}h
\]
oscillates as $t\downarrow0$;
\item for every $h\in\R^2$ and every $h_n\to h$,
\[
\sqrt n\{T_n-f(h_n/\sqrt n)\}\wto N(0,1)
\quad\text{under }N(h_n/\sqrt n,I_2)^{\otimes n}:
\] $T_n$ is regular at $0$ along all of $\R^2$, uniformly in local sequences, with a Gaussian limit;
\item along any subsequence with $\vartheta(1/\sqrt{n_k})\to\vartheta^*$, $\sqrt{n_k}f(h/\sqrt{n_k})\to\ip{u(\vartheta^*)}h$, linear in $h$ (the subsequence form of \cref{lem:tilt}), and different subsequences give different linear functionals.
\end{enumerate}
\end{proposition}
\begin{proof}
(1)--(2) are direct ($\vartheta$ takes every value in $[-1,1]$ on every $(0,\epsilon)$; $\vartheta\mapsto\ip{u(\vartheta)}h$ is non-constant on $[-1,1]$ for $h\ne0$). (3) Under $\theta_n=h_n/\sqrt n$, $\sqrt n\bar X_n=h_n+Z$ exactly, $Z\sim N(0,I_2)$. With $\vartheta_n:=\vartheta(1/\sqrt n)$: $\sqrt nT_n=\ip{u(\vartheta(\norm{h_n+Z}/\sqrt n))}{h_n+Z}$ and $\sqrt nf(\theta_n)=\ip{u(\vartheta(\norm{h_n}/\sqrt n))}{h_n}$. For $c$ in a compact subset of $(0,\infty)$,
\[
\abs{\sqrt{\log(\sqrt n/c)}-\sqrt{\log\sqrt n}}
=\frac{\abs{\log c}}{\sqrt{\log(\sqrt n/c)}+\sqrt{\log\sqrt n}}
\to0
\]
uniformly, and $\sin$ is $1$-Lipschitz; since $\norm{h_n+Z}$ is tight with $P(\norm{h_n+Z}<\epsilon)\to0$ as $\epsilon\to0$ uniformly in $n$, $\vartheta(\norm{h_n+Z}/\sqrt n)-\vartheta_n\to0$ in probability, and $\vartheta(\norm{h_n}/\sqrt n)-\vartheta_n\to0$ when $h\ne0$ (when $h=0$ both $\sqrt nf(\theta_n)$ and $\ip{u(\cdot)}{h_n}$ tend to $0$). As $u$ is Lipschitz and $h_n+Z$ tight, $\sqrt n\{T_n-f(\theta_n)\}=\ip{u(\vartheta_n)}Z+o_P(1)$, and $\ip{u(\vartheta_n)}Z\sim N(0,1)$ exactly. (4) is immediate.
\end{proof}
Proposition~\ref{prop:rotating} illustrates why \cref{def:dpd} explicitly requires directional limits: even a regular estimator with a Gaussian limit need not supply them. \citet{Pfanzagl2000} establishes the same distinction with a counterexample. The construction here is a self-contained illustration of this distinction.

\section{Path conditions for the programme-value example}
\label{app:policy-model}
Consider the bounded-outcome model described in \cref{ex:policy}.
Let $Q$ be another law in this model, write $m_{a,Q}$ and $\pi_{a,Q}$
for its outcome regression and propensity, and let $H(P,Q)$ denote
Hellinger distance without the factor $1/\sqrt2$. Direct conditioning
gives
\[
\mu_a(Q)-\mu_a(P)-\E_Q\lambda_a
=\E_{Q_X}\left[
 \left(1-\frac{\pi_{a,Q}}{\pi_a}\right)(m_{a,Q}-\mu_a)
\right].
\]
Choose versions of the conditional laws on marginal null sets with
the same outcome and overlap bounds. If $v(x)$ is their conditional
total-variation distance, then
$|\pi_{a,Q}-\pi_a|\le v$ and
$|m_{a,Q}-\mu_a|\le 3B_0v/\epsilon$.
Moreover,
$\int v^2\,dQ_X\le4H^2(P,Q)$: total variation is bounded by
conditional Hellinger distance, and the triangle inequality between
$\sqrt{q}$, $\sqrt{p}$, and $\sqrt{q_X}\sqrt{p(\cdot\mid X)}$
bounds the latter distance in $L_2(Q_X)$ by $2H(P,Q)$.
The displayed remainder is therefore bounded in absolute value by
$12B_0 H^2(P,Q)/\epsilon^2$.
Along any DQM path, $H(P_{t,h},P)=O(t)$ and boundedness of
$\lambda_a$ gives
$\E_{P_{t,h}}\lambda_a=t\ip{\lambda_a}h+o(t)$.
This proves \eqref{eq:policy-pd} along every such path.

To verify \assT, let $h\in L_2^0(P)$ and set
\[
h_t=\max(-t^{-1/2},\min(h,t^{-1/2}))
-\E\max(-t^{-1/2},\min(h,t^{-1/2})).
\]
Then $h_t\to h$ in $L_2(P)$ and
$\|t h_t\|_\infty\le2\sqrt t\to0$.
The densities $dP_t/dP=1+t h_t$ are consequently positive for small
$t$ and define a DQM path with score $h$. They preserve the outcome
bound, and their propensity ratios are bounded below by
$(1-2\sqrt t)/(1+2\sqrt t)$.
Since $\pi_a\ge\epsilon+\eta$, they also preserve the model's
overlap bound for all sufficiently small $t$.
Thus every score in the closed space $L_2^0(P)$ is realized.
\section{Optimal individualized treatment value: geometry and attainment}\label{app:otr}
We prove \cref{cor:otr-geometry,cor:otr-attain} in four steps. First we verify the all-path derivative and identify the active gradients. We then compute their affine-hull projection, establish the feasible estimator's expansion, and compare the one-sided cones with the core.

\begin{proof}[Proof of \cref{cor:otr-geometry,cor:otr-attain}]
\emph{Step 1: paths and gradients.}
The density tilts in \cref{app:policy-model} preserve the finite support of $X$ and realize every $h\in L_2^0(P)$ while respecting the common outcome and overlap bounds. Along every DQM path, expectations of bounded functions admit their usual linear expansions. In particular, $p_x$, $P(X=x,A=a)$, and $\E\{Y\ind\{X=x,A=a\}\}$ do so. Their positive denominators at $P$ allow differentiation of ratios, giving
\[
\begin{split}
\left.\frac{d}{dt}p_{x,t}\right|_{0+}&=p_x\E(h\mid X=x),\\
\left.\frac{d}{dt}\mu_{a,t}(x)\right|_{0+}
&=\frac{\E[\ind\{X=x,A=a\}\{Y-\mu_a(x)\}h]}{p_x\pi_a(x)}
=\dot\mu_{a,h}(x).
\end{split}
\]
There are finitely many cells and rules, so differentiation of the maximum proves \eqref{eq:otr-derivative}. Write $R_d(O)=R_{d(X)}(O)$ and
\[
\lambda_d(O)=q(X)-\Psi(P)+R_d(O),\qquad
\mathcal D^*=\{d:d(x)\in\mathcal A^*(x)\ \forall x\}.
\]
Then $\gamma(h)=\max_{d\in\mathcal D^*}\ip{\lambda_d}h$ and $\LamP=\conv\{\lambda_d:d\in\mathcal D^*\}$. Two active rules can be chosen to differ in just one stratum and in any pair of its active treatments. Consequently
\[
\spn(\LamP-\LamP)
=\spn\{\ind\{X=x\}(R_a-R_b):x\le J,\ a,b\in\mathcal A^*(x)\}.
\]
Orthogonality to these differences is precisely \eqref{eq:otr-core}. The sublinear geometry theorem therefore identifies $\Hk$ as the unique maximal linearity subspace.

\emph{Step 2: the efficient projection.}
The affine hull of the active gradients consists of
\begin{equation}\label{eq:otr-affine-proof}
q(X)-\Psi(P)+\sum_aw_a(X)R_a,
\qquad \sum_{a\in\mathcal A^*(x)}w_a(x)=1,
\quad w_a(x)=0\ \text{outside }\mathcal A^*(x).
\end{equation}
To verify this representation, affine combinations of rules give weights of this form. Conversely, fix one active rule. The differences formed by changing it in one stratum generate every array of weights whose sum within each stratum is zero, proving the reverse inclusion. The conditional residuals satisfy
\[
\E(R_a\mid X)=0,\qquad
\E(R_aR_b\mid X)=0\ (a\ne b),\qquad
\E(R_a^2\mid X=x)=c_a(x)>0.
\]
Thus the squared norm of \eqref{eq:otr-affine-proof} is
\[
\Var\{q(X)\}+\sum_xp_x\sum_{a\in\mathcal A^*(x)}c_a(x)w_a(x)^2.
\]
The minimization separates by stratum. Cauchy--Schwarz gives
\[
1=\left(\sum_aw_a(x)\right)^2
\le\left(\sum_ac_a(x)w_a(x)^2\right)
\left(\sum_ac_a(x)^{-1}\right),
\]
where both sums run over $\mathcal A^*(x)$, with equality exactly at \eqref{eq:otr-weights}. This proves \eqref{eq:otr-if-bound}. All these weights are positive. The minimizing gradient also lies in the convex hull: take the convex combination of active rules with coefficients $\prod_xw_{d(x)}^*(x)$. Its conditional slope satisfies
\[
\dot\mu_{a,\psik}(x)=w_a^*(x)c_a(x)
=\left\{\sum_{b\in\mathcal A^*(x)}c_b(x)^{-1}\right\}^{-1}
\quad(a\in\mathcal A^*(x)),
\]
so $\psik\in\Hk$, consistently with the projection interpretation. Applying \cref{thm:main} proves the core convolution bound.

\emph{Step 3: feasible estimation and local regularity.}
Every empirical cell mean is root-$n$ consistent and the empirical cell probabilities and variances are consistent. Positivity of the true cell probabilities and variances implies that the exceptional event in the definition of \eqref{eq:otr-estimator} has probability tending to zero. Since there are finitely many treatments and strata,
\[
\max_{x,a}\abs{\hat\mu_a(x)-\mu_a(x)}=O_{P^n}(n^{-1/2}).
\]
Active treatments differ in estimated mean by $O_{P^n}(n^{-1/2})=o_{P^n}(\tau_n)$, while every inactive treatment has a fixed strictly positive gap from the stratum's optimal mean. Hence
\[
P^n\{\widehat{\mathcal A}(x)=\mathcal A^*(x)\text{ for every }x\}\longrightarrow1,
\qquad \max_{x,a}\abs{\hat w_a(x)-w_a^*(x)}=o_{P^n}(1).
\]
On the event of correct active sets, exact normalization and equality of the active means give
\[
\sum_a\hat w_a(x)\hat\mu_a(x)-q(x)
=\sum_aw_a^*(x)\{\hat\mu_a(x)-\mu_a(x)\}
+o_{P^n}(n^{-1/2})
\]
simultaneously over $x$. Indeed, the remainder is a finite sum of products of a consistent weight error and a root-$n$ mean error; the term multiplying the common active mean is exactly zero. Differentiating the empirical ratios gives
\[
\hat\mu_a(x)-\mu_a(x)
=\frac1n\sum_{i=1}^n
\frac{\ind\{X_i=x\}}{p_x}R_a(O_i)+o_{P^n}(n^{-1/2}).
\]
Combining this expansion with $\hat p_x-p_x=(\mathbb P_n-P)\ind\{X=x\}$ proves \eqref{eq:otr-al}. Consistency of \eqref{eq:otr-variance} follows from the same active-set event, consistency of the cell moments, and positivity of the limiting denominators.

For any DQM path with score $h$, contiguity transfers the $o_{P^n}(1)$ remainder in \eqref{eq:otr-al} and the high-probability active-set event to $\Pn h$. The joint central limit theorem for $(\psik,h)$ and Le Cam's third lemma yield
\[
\sqrt n\{\widehat\Psi_{\mathrm{eff}}-\Psi(P)\}
\wto N(\ip{\psik}h,\Vk)
\quad\text{under }\Pn h.
\]
Subtracting the all-path expansion of the target gives \eqref{eq:otr-drift}. In each stratum a positive weighted average is at most its maximum, with equality exactly when every active slope is equal. Since all $p_x>0$, $\delta(h)=0$ is equivalent to $h\in\Hk$. In particular, the centered limit is $N(0,\Vk)$ for every path whose score lies in $\Hk$, proving core regularity and attainment. Identification of the active sets here refers to those at the base law $P$; under a local tie-breaking path the estimator continues to pool them, producing precisely the displayed drift.

\emph{Step 4: one-sided cones.}
Fix $d\in\mathcal D^*$ and put $h_d=R_d/c_{d(X)}(X)$. This is a bounded mean-zero score, since $X$ has finite support and all $c_a(x)>0$. Conditional orthogonality gives
\[
\dot\mu_{d(x),h_d}(x)=1,\qquad
\dot\mu_{a,h_d}(x)=0\quad(a\ne d(x)).
\]
Thus all nonvacuous inequalities defining $C_d$ hold strictly at $h_d$. There are finitely many continuous linear inequalities, so $C_d$ contains a nonempty open set and spans $\T$. If there are no tied strata then $C_d=\T$ directly. On $C_d$, $\gamma(h)=\ip{\lambda_d}h$, and the general cone convolution theorem gives the Gaussian variance $\norm{\lambda_d}^2=V_d$ in \eqref{eq:otr-cone-price}. Equivalently, the preceding strict inequalities expose $\lambda_d$ as a vertex of the active-gradient hull. This argument does not require nonsingularity of the Gram matrix of the active rules.

The empirical fixed-rule value has influence function $\lambda_d$ by the cell-mean expansion in Step~3 (with the same harmless fallback on empty cells), so it is regular and attains $V_d$ along $C_d$. Subtracting \eqref{eq:otr-if-bound} yields the stated price. If $\mathcal A^*(x)=\{d(x)\}$, its stratum contribution is zero. If at least two treatments are active there, positivity of each $c_a(x)^{-1}$ gives
\[
\left\{\sum_{a\in\mathcal A^*(x)}c_a(x)^{-1}\right\}^{-1}<c_{d(x)}(x).
\]
Since $p_x>0$, the total price is then strictly positive.
\end{proof}

\section{Conditional maxima: path expansion and efficiency}\label{app:cm}
We first verify a sufficient observed-data condition for the expansion used in \cref{prop:cm}. We then identify the linearity core and construct the conditional projection, including the measurable and singular cases.

\begin{lemma}[bounded conditional moment ratios]\label{lem:cm-paths}
Let $U_\ell,V_\ell$, $\ell\le L$, be fixed bounded observed variables. Suppose $\E_Q[V_\ell\mid X]\ge\epsilon>0$ under every law $Q$ in the model. For fixed bounded measurable coefficient functions $a_{j\ell}$, define
\[
\eta_\ell(Q,X)=\frac{\E_Q[U_\ell\mid X]}{\E_Q[V_\ell\mid X]},\qquad
b_j(Q,X)=a_{j0}(X)+\sum_{\ell=1}^L a_{j\ell}(X)\eta_\ell(Q,X).
\]
Then the branches have uniformly bounded versions and satisfy \eqref{eq:cm-path} along every DQM path, with
\begin{equation}\label{eq:cm-ratio-residual}
r_j=\sum_{\ell=1}^L a_{j\ell}(X)
\frac{U_\ell-\eta_\ell(P,X)V_\ell}{\E_P[V_\ell\mid X]}.
\end{equation}
\end{lemma}

\begin{proof}
\emph{Step 1: a bounded conditional mean.}
Let $Z$ be bounded, $m_t(X)=\E_{P_t}[Z\mid X]$, and $m=m_0$. Write $k_t$ for the density of the absolutely continuous part of $P_{t,X}$ with respect to $P_X$. DQM implies the signed-measure expansion $P_t=P+thP+o_{\mathrm{TV}}(t)$. Applying it to bounded functions of $X$ and to their products with $Z$ gives
\[
k_t=1+t\E[h\mid X]+o_{L_1(P_X)}(t),\qquad
k_tm_t=m+t\E[Zh\mid X]+o_{L_1(P_X)}(t).
\]
The marginal singular mass is $o(t^2)$ by DQM and does not affect these density identities. Subtraction yields
\begin{equation}\label{eq:cm-ratio-step}
k_t(m_t-m)=t\E[(Z-m)h\mid X]+e_t,\qquad \|e_t\|_1=o(t).
\end{equation}
Let $B_t=\{k_t<1/2\}$. Hellinger contraction gives
\[
P_X(B_t)\le (1-2^{-1/2})^{-2}
\int(\sqrt{k_t}-1)^2\,dP_X=O(t^2).
\]
Choose bounded versions of $m_t$ on $\{k_t=0\}$. On $B_t^c$, divide \eqref{eq:cm-ratio-step} by $k_t$; the remainder remains $o_{L_1}(t)$, while $1/k_t\to1$ in probability and $|1/k_t-1|\le1$. Dominated convergence against the integrable variable $\E[(Z-m)h\mid X]$ handles this factor. On $B_t$, boundedness gives $t^{-1}\E[|m_t-m|\ind_{B_t}]=O(t)$, and the integral of the derivative over $B_t$ tends to zero. Therefore
\[
\frac{m_t-m}{t}\longrightarrow\E[(Z-m)h\mid X]
\quad\text{in }L_1(P_X).
\]
This argument permits the marginal support to change with $t$.

\emph{Step 2: ratios and affine branches.}
Apply Step~1 to $u_t=\E_{P_t}[U_\ell\mid X]$ and $v_t=\E_{P_t}[V_\ell\mid X]$. Their bounded versions can be chosen with $v_t\ge\epsilon$ also off the support of $P_{t,X}$. The identity
\[
\frac{u_t/v_t-u/v}{t}
=\frac1{v_t}\left\{\frac{u_t-u}{t}-\frac uv\frac{v_t-v}{t}\right\}
\]
and the bounded, probability-convergent multipliers $1/v_t$ show convergence in $L_1(P_X)$ to
\[
\frac{\E[(U_\ell-u)h\mid X]-(u/v)\E[(V_\ell-v)h\mid X]}v
=\E\!\left[\frac{U_\ell-\eta_\ell(P,X)V_\ell}{v}h\,\middle|\,X\right].
\]
Finite sums with bounded $X$-measurable coefficients preserve the expansion. The resulting residual is \eqref{eq:cm-ratio-residual}; it is bounded and has conditional mean zero.
\end{proof}

\begin{proof}[Proof of \cref{prop:cm}]
\emph{Step 1: the target derivative, including the moving marginal.}
Fix a DQM path with score $h$ and abbreviate $b_{j,t}=b_j(P_t,\cdot)$ and $q_t=\max_jb_{j,t}$. The Lipschitz property of the finite maximum and \eqref{eq:cm-path} give
\[
\left\|q_t-\max_j\{b_j+t\dot b_{j,h}\}\right\|_1=o(t).
\]
For each $x$, the right derivative of the finite maximum is $\max_{j\in\mathcal J^*(x)}\dot b_{j,h}(x)$. Its difference quotients are bounded by $\max_j|\dot b_{j,h}(x)|$, an integrable function. Dominated convergence therefore proves
\begin{equation}\label{eq:cm-max-L1}
\frac{q_t-q}{t}\longrightarrow
\max_{j\in\mathcal J^*(X)}\dot b_{j,h}(X)
\quad\text{in }L_1(P_X).
\end{equation}
Set $f_t=q_t-q$. The branch bound and \eqref{eq:cm-max-L1} imply $\E_P[f_t^2]\to0$. Total-variation convergence and boundedness give $\E_{P_t}[f_t^2]\to0$ as well. Cauchy--Schwarz applied to the difference of the square roots of the marginal densities now gives
\[
\left|\int f_t\,d(P_{t,X}-P_X)\right|
\le H(P_{t,X},P_X)
\{2\E_P[f_t^2]+2\E_{P_t}[f_t^2]\}^{1/2}=o(t).
\]
All functions here have bounded extensions to marginal null sets; the bound also controls the possible singular parts. Since $q$ is bounded, DQM gives $\E_{P_t}q-\E_Pq=t\E[qh]+o(t)$. Combining this identity with \eqref{eq:cm-max-L1} proves \eqref{eq:cm-derivative} for every path.

\emph{Step 2: the core in the actual tangent space.}
The derivative in \eqref{eq:cm-derivative} is continuous and sublinear, because $|\dot b_{j,h}|\le\E[|r_jh|\mid X]$ and $r_j\in L_2(P)$. Moreover,
\[
\gamma(h)+\gamma(-h)
=\E\!\left[\max_{j\in\mathcal J^*(X)}\dot b_{j,h}
-\min_{j\in\mathcal J^*(X)}\dot b_{j,h}\right].
\]
The integrand is nonnegative; its expectation is zero exactly when all active derivatives agree almost surely. This equality set is a closed linear subspace of $\T$, and on it $\gamma(h)=\ip{\lambda_0}h$. Every linearity subspace is contained in it, proving the unique maximality and the projection in \eqref{eq:cm-core}. The convolution statement follows from \cref{thm:main}.

For the remaining steps suppose $\T=L_2^0(P)$. A measurable active selector $d(X)\in\mathcal J^*(X)$ has gradient $\lambda_d=q-\Psi(P)+r_{d(X)}$. The derivative is its support function:
\[
\gamma(h)=\sup_d\ip{\lambda_d}h.
\]
Indeed, the smallest index maximizing $\dot b_{j,h}(X)$ among active indices is a measurable selector attaining the supremum. These gradients form a bounded subset of $L_2^0(P)$. Let
\[
F=\clspan\{\lambda_d-\lambda_e:d,e\text{ measurable active selectors}\}.
\]
The same conditional equality argument identifies $\Hk=F^\perp$. In particular, for bounded measurable $a(X)$ and an index $j$, the localized difference $a(X)\ind\{j\in\mathcal J^*(X)\}(r_j-r_{j_0(X)})$ belongs to $F$: for indicator $a$ it is a difference of two selector gradients, and bounded simple approximation gives the general case in $L_2(P)$.

\emph{Step 3: measurable conditional projection and its closure.}
Put $r_0=r_{j_0(X)}$ and
\[
d_j=\ind\{j\in\mathcal J^*(X)\}(r_j-r_0),\quad
B(X)=\E[dd^\top\mid X],\quad c(X)=\E[dr_0\mid X].
\]
For almost every $X$, $c(X)$ is in the range of $B(X)$. To see this, a null vector $z$ of the conditional Gram matrix satisfies $z^\top d=0$ in the corresponding conditional $L_2$ space and hence $z^\top c=0$. Thus $\beta=B^\dagger c$ solves $B\beta=c$. The Moore--Penrose inverse is a Borel measurable function of the entries of a finite matrix, so $\beta(X)$ is measurable. Define
\[
r_*=r_0-\beta(X)^\top d.
\]
Conditional least squares gives
\begin{equation}\label{eq:cm-normal}
\E[dr_*\mid X]=0,\qquad
\E[r_*^2\mid X]=\E[r_0^2\mid X]-c^\top B^\dagger c\ge0.
\end{equation}
Both $r_*$ and $\beta^\top d$ are square-integrable: their conditional squared norms are bounded by $\E[r_0^2\mid X]$. The coefficients themselves need not have an integrable norm.

The expression for $r_*$ is an active affine combination of the $r_j$, and hence gives a solution $w^*$ of \eqref{eq:cm-qp}. More explicitly, set $w_j^*=-\beta_j$ for $j\ne j_0(X)$ and $w_{j_0(X)}^*=1+\sum_{j\ne j_0(X)}\beta_j$. Inactive rows of $B$ are zero, and their coefficients in $B^\dagger c$ are zero. Consequently these weights are active, measurable, and sum to one. The conditional normal equations prove minimality; the minimizing residual is unique even if its weights are not.

To justify the closed affine-hull calculation, truncate on $\{\|\beta(X)\|\le m\}$. Step~2 shows that $\ind\{\|\beta\|\le m\}\beta^\top d\in F$, and square-integrability implies convergence in $L_2(P)$ to $\beta^\top d$. Thus $r_0-r_*\in F$. On the other hand, \eqref{eq:cm-normal} and $\E[d\mid X]=0$ give $q-\Psi(P)+r_*\in F^\perp$. It is therefore the orthogonal projection of $\lambda_0$ onto $\Hk$. Conditional centering makes the $q-\Psi(P)$ and $r_*$ components orthogonal, establishing \eqref{eq:cm-efficient}.

\emph{Step 4: estimator comparison and the residual.}
For any square-integrable $\phi_w$ with active affine weights, $\phi_w-\lambda_0=\sum_jw_jd_j$. Truncating on bounded-weight events, as in Step~3, proves $\phi_w-\lambda_0\in F$. Hence $\Pi_{\Hk}\phi_w=\psik$. The asymptotically linear estimator theorem gives core regularity and a Gaussian residual with variance $\|\phi_w-\psik\|^2$. Conditional expectation of its square is exactly the quadratic form in \eqref{eq:cm-gap}. It vanishes precisely when the two influence functions agree in $L_2(P)$. Applying the local limit in \cref{thm:al} gives \eqref{eq:cm-drift}. For possibly unbounded weights, the expectation there is of the combined sum: localization gives
\[
\sum_j w_j^*(X)\dot b_{j,h}(X)
=\E\!\left[\Big\{\sum_jw_j^*(X)r_j\Big\}h\,\middle|\,X\right]\in L_1(P_X),
\]
by Cauchy--Schwarz, even if the individual weighted derivatives are not separately integrable.

Finally, uniform active weights minimize the conditional quadratic form precisely when its gradient is constant over the active coordinates. The active block of $G$ is positive semidefinite, so this first-order condition is necessary and sufficient even if it is singular. For $u=\one/|\mathcal J^*|$, the condition is equality of its row sums. This proves the stated equality criterion for the uniform-weight estimator.
\end{proof}

\section{Instrumental-variable bounds and calibration error}\label{app:bp-cal}
The common conditional-max result gives the geometry in these applications. We verify its model conditions and then establish feasible attainment, paying particular attention to singular branch covariance matrices.

\begin{proof}[Proof of \cref{cor:bp}]
\emph{Step 1: the path derivative and conditional covariance.}
Each $p_{yd\mid v}(x)$ is the ratio
\[
\frac{\E[\ind\{Y=y,D=d,V=v\}\mid X=x]}
 {\E[\ind\{V=v\}\mid X=x]}.
\]
The numerator is bounded and the denominator has the common positive lower bound. Thus \cref{lem:cm-paths} gives, along every DQM path with score $h$, the $L_1(P_X)$ derivative
\[
\dot p_{yd\mid v,h}(X)
=\E\!\left[
\frac{\ind\{V=v\}}{e_v(X)}
\{\ind\{Y=y,D=d\}-p_{yd\mid v}(X)\}h\,
\middle|\,X\right].
\]
Finite affine combinations give $\dot b_{j,h}=\E[r_jh\mid X]$. The derivative of the lower endpoint is therefore
\[
\gamma_L(h)=\E[(q(X)-\Psi_L)h]
 +\E\max_{j\in\mathcal J^*(X)}\E[r_jh\mid X].
\]
Bounded mean-zero likelihood tilts preserve instrument overlap for all sufficiently small parameter values because $P$ has uniform slack. For an arbitrary $h\in L_2^0(P)$, the truncated-score construction in \cref{app:policy-model} realizes a DQM path with score $h$ and still preserves overlap. Thus \assT holds with $\T=L_2^0(P)$. The model here imposes no instrumental inequalities on observed-law perturbations.

Conditionally on $(X,V=v)$, the four outcome--treatment cell indicators have covariance $\operatorname{diag}(p_v)-p_vp_v^\top$. Indicators of different instrument values have zero product, and each residual block has conditional mean zero. Expanding the products in \eqref{eq:bp-residual} yields \eqref{eq:bp-gram}. \Cref{prop:cm} now gives the core, efficient influence function, and convolution bound. In particular, the first-order condition for the affine minimizer is
\begin{equation}\label{eq:bp-kkt}
G_*w_*^*=v_*\one,\qquad
v_*=w_*^{*\top}G_*w_*^*.
\end{equation}
Here $w_*^*$ denotes the active subvector of $w^*$. This condition holds also for singular $G_*$: differentiation along each vector orthogonal to $\one$ shows that $G_*w_*^*$ is constant across its active coordinates, and multiplication by $w_*^{*\top}$ identifies that constant. It implies that the minimizing residual has the same conditional covariance with every active $r_j$, so its influence function lies in $\Hk$.

\emph{Step 2: stability of the affine optimizer.}
Fix a nonempty active index set $I$ and a stratum. Let $A_{I,v}$ be the matrix whose rows are $a_{j,v}^\top$, $j\in I$. For every vector $z\in\R^{|I|}$,
\[
z^\top G_I z
=\sum_v\frac{1}{e_v}
\Var_{p_v}\{(A_{I,v}^\top z)_{Y,D}\}.
\]
If all four coordinates of each $p_v$ are positive, this is zero exactly when $A_{I,v}^\top z$ is constant across the four cells for both $v$. Hence
\begin{equation}\label{eq:bp-fixed-kernel}
K_I:=\ker G_I
=\{z:A_{I,v}^\top z\in\spn\{\one_4\},\ v=0,1\}
\end{equation}
depends only on the known coefficients and $I$. Write $\Pi_{K_I}$ for its Euclidean orthogonal projection. The minimum-Euclidean-norm affine optimizer is
\begin{equation}\label{eq:bp-optimizer}
w_I(G_I)=
\begin{cases}
\displaystyle\frac{\Pi_{K_I}\one}{\one^\top\Pi_{K_I}\one},
 &\Pi_{K_I}\one\ne0,\\[7pt]
\displaystyle\frac{G_I^+\one}{\one^\top G_I^+\one},
 &\Pi_{K_I}\one=0.
\end{cases}
\end{equation}
In the first case the minimum is zero: every optimizer is in $K_I$ with sum one, and orthogonal projection gives the least-norm such vector. In the second case $\one\in K_I^\perp=\range(G_I)$, and \eqref{eq:bp-kkt} gives the displayed solution plus an arbitrary null vector; the least-norm solution has no null component. Both denominators are positive in their respective cases. Since the kernel is fixed, the positive eigenvalues remain bounded away from zero in a sufficiently small neighbourhood of any positive cell-probability vector. The pseudoinverse and the selected optimizer are therefore continuous there. This proves the required stability without a nonsingularity assumption.

\emph{Step 3: the empirical expansion and variance estimator.}
There are finitely many strata and cells. Their empirical proportions are jointly root-$n$ consistent, every cell is nonempty with probability tending to one, and the empirical conditional cell probabilities and branches have joint root-$n$ expansions. All true inactive gaps are positive at the fixed law. The threshold conditions imply
\[
P\{\widehat{\mathcal J}(x)=\mathcal J^*(x)\text{ for every }x\}\longrightarrow1.
\]
On this event, Step 2 gives $\hat w(x)=w^*(x)+o_p(1)$ for the selected true minimizer. The error in the weights multiplies a common tied branch value, and therefore cancels to first order:
\[
\begin{split}
\hat q_w(x)-q(x)
&=\sum_{j\in\mathcal J^*(x)}w_j^*(x)\{\hat b_j(x)-b_j(x)\}\\
&\quad+\sum_{j\in\mathcal J^*(x)}\{\hat w_j(x)-w_j^*(x)\}
 \{\hat b_j(x)-b_j(x)\}\\
&=\sum_{j\in\mathcal J^*(x)}w_j^*(x)\{\hat b_j(x)-b_j(x)\}
 +o_p(n^{-1/2}).
\end{split}
\]
The omitted constant term equals $q(x)\sum_j(\hat w_j-w_j^*)=0$. For each conditional cell probability the ratio expansion is
\[
\hat p_{yd\mid v}(x)-p_{yd\mid v}(x)
=\frac1n\sum_{i=1}^n
\frac{\ind\{X_i=x,V_i=v\}}{p_xe_v(x)}
\{\ind\{Y_i=y,D_i=d\}-p_{yd\mid v}(x)\}
+o_p(n^{-1/2}).
\]
Combining these expansions with the empirical stratum proportions yields
\[
\sqrt n(\hat\Psi_L-\Psi_L)
=\frac1{\sqrt n}\sum_{i=1}^n
\left[q(X_i)-\Psi_L+\sum_jw_j^*(X_i)r_j(Z_i)\right]+o_p(1).
\]
The zero fallback alters this only on an event of probability tending to zero. Continuity, finite support and consistency of all estimated probabilities show directly that \eqref{eq:bp-estimator}'s variance estimate converges to \eqref{eq:bp-bound}. Contiguity transfers the expansion and variance consistency along every DQM local path. Le Cam's third lemma shifts the limiting mean by $\ip{\psik}h$, which equals $\gamma_L(h)$ on $\Hk$; hence the estimator is core-regular and attains the convolution bound.

\emph{Step 4: the uniform-weight comparison.}
The influence function of the uniform active rule is $q-\Psi_L+\sum_ju_jr_j$. Conditional mean-zero residuals are orthogonal to $q-\Psi_L$. In addition, \eqref{eq:bp-kkt} and the zero inactive coordinates give $(u-w^*)^\top Gw^*=v_*\one^\top(u-w^*)=0$. Expanding the quadratic form proves \eqref{eq:bp-uniform-gap}. Equality holds precisely when the two residuals agree in conditional $L_2$ almost surely. Replacing the outcome by $1-Y$ reverses the sign of every causal treatment effect, which proves the stated upper-endpoint construction.
\end{proof}

\begin{proof}[Proof of \cref{cor:calibration}]
\emph{Step 1: the derivative and projection.}
Put $\Delta(S)=m(S)-S$ and $T=\{s:\Delta(s)=0\}$. The bounded conditional-mean case of \cref{lem:cm-paths} gives $\dot m_h(S)=\E[(Y-m(S))h\mid S]$ along every DQM path. Consequently
\[
\begin{split}
\gamma_{\mathrm{cal}}(h)
&=\E[(|\Delta(S)|-\Psi_{\mathrm{cal}})h]\\
&\quad+\E[\ind\{S\notin T\}\operatorname{sgn}\{\Delta(S)\}\dot m_h(S)]
 +\E[\ind\{S\in T\}|\dot m_h(S)|].
\end{split}
\]
The truncated-score construction of \cref{app:policy-model}, now without an overlap constraint, realizes every $h\in L_2^0(P)$ while preserving the outcome and predictor bounds. Thus \assT holds with full tangent space. On $T$ the two branch residuals are $Y-m(S)$ and its negative. Their affine hull contains zero, so its minimum-norm point is zero. Off $T$ the single active residual is $\operatorname{sgn}(\Delta)(Y-m)$. \Cref{prop:cm} gives the core and influence function in \eqref{eq:cal-core-if}. Alternatively, its membership in the core follows directly from
\[
\ind\{S\in T\}\E[(Y-m(S))\psik\mid S]=0.
\]
The residual has conditional mean zero given $S$, proving the displayed variance decomposition. The convolution conclusion follows from \cref{thm:main}.

\emph{Step 2: feasible attainment.}
Suppose $S$ has finite support with every stratum probability positive. Boundedness gives joint root-$n$ consistency of the empirical stratum means. Since the nonzero values of $|\Delta(s)|$ have a positive minimum, the thresholded empirical sign
\[
\hat\xi(s)=\operatorname{sgn}\{\hat m(s)-s\}
\ind\{|\hat m(s)-s|>\tau_n\}
\]
equals $\xi(s)=\operatorname{sgn}\{m(s)-s\}$ in every stratum with probability tending to one. Empty strata contribute zero to the estimator and disappear with probability tending to one. On the simultaneous sign-recovery event, \eqref{eq:cal-estimator} is exactly
\[
\frac1n\sum_{i=1}^n\xi(S_i)(Y_i-S_i).
\]
Its expectation is $\Psi_{\mathrm{cal}}$, and
\[
\xi(S)(Y-S)-\Psi_{\mathrm{cal}}
=|\Delta(S)|-\Psi_{\mathrm{cal}}+\xi(S)(Y-m(S))=\psik.
\]
This proves asymptotic linearity. The empirical variance of the estimated summands is consistent because their signs agree with the true signs on an event of probability tending to one and the summands are bounded. Contiguity and the path derivative establish core regularity exactly as in the preceding proof.

\emph{Step 3: the existing smoothed estimator and the zero-variance case.}
The limiting influence function in \citet[Corollary 4.8]{Whitehouse2025} is $\xi(S)(Y-S)-\Psi_{\mathrm{cal}}$, which equals \eqref{eq:cal-core-if}. Their cross-fitted expansion in Corollary C.2 averages this influence function over all observations. Whenever their stated smoothing, gap and nuisance-estimation conditions give that expansion, the preceding derivative calculation and third-lemma argument prove attainment of the same core bound. If $P(S\in T)=1$, then $\Psi_{\mathrm{cal}}=0$ and $\psik=0$. The finite-support estimator is zero with probability tending to one under $P$ and, by contiguity, under each fixed DQM local alternative. On the alternative with score $h$, its scaled target-centred error therefore converges to $-\E|\dot m_h(S)|$. It is zero along the core but can be nonzero elsewhere, despite the vanishing Gaussian variance.
\end{proof}

\section{Sequential treatment: paths, recursion, and attainment}\label{app:dynamic}
We prove \cref{cor:dynamic} and the feasible assertion following it. The proof first controls the intervention laws uniformly over policies, then differentiates the Bellman recursion and identifies its active gradients. The final step verifies the empirical construction, including singular Gram matrices.

\begin{proof}
\emph{Step 1: path control and fixed-policy gradients.}
The observed law factors into the initial history law, treatment kernels, and the kernels of the next observation given $(H_t,A_t)$, with the terminal kernel including $Y$. Removing the treatment kernels defines $P^d$. Its history distribution before decision $t$ has density at most $\epsilon^{-(t-1)}$ relative to the observational history distribution. A telescoping product of transition kernels consequently gives a constant $C=C(T,\epsilon)$ such that, for any two laws $P,Q$ in the model,
\begin{equation}\label{eq:dynamic-tv}
\sup_d\TV(P^d,Q^d)\le C\TV(P,Q),\qquad
\sup_d\abs{\Psi_d(P)-\Psi_d(Q)}\le 2B_0 C\TV(P,Q).
\end{equation}
Here total variation for probability measures is $\sup_B\abs{P(B)-Q(B)}$. To see the first inequality in detail, telescope one transition kernel at a time using the $P^d$ prefix and the $Q^d$ suffix. The initial-history term is at most $\TV(P,Q)$. For the transition after decision $t$, the preceding density bound and $\pi_{t,a}\ge\epsilon$ bound its contribution by
\[
\epsilon^{-t}\E_P\big[\TV\{P(\cdot\mid H_t,A_t),Q(\cdot\mid H_t,A_t)\}\big]
\le 2\epsilon^{-t}\TV(P,Q).
\]
The last inequality follows by adding and subtracting the product of the $P$ marginal of $(H_t,A_t)$ and the $Q$ conditional kernel; the marginal total variation is at most the joint total variation. This remains valid for arbitrary versions of the $Q$ kernel on its null histories. Thus $C=1+2\sum_{t=1}^T\epsilon^{-t}$ suffices. Taking suprema shows that the second inequality in \eqref{eq:dynamic-tv} also holds for the optimal value $\Psi$.

For a fixed policy, write $V_t^d$ and $Q_t^d$ for its usual continuation values, with $V_{T+1}^d=Y$, $Q_t^d=\E[V_{t+1}^d\mid H_t,A_t]$, and $V_t^d=Q_t^d(H_t,d_t(H_t))$. Under a bounded mean-zero score $b$, use the path $dP_s=(1+sb)dP$. Every conditional transition then has density ratio
\[
\frac{1+s\E[b\mid H_{t+1}]}{1+s\E[b\mid H_t,A_t]}
\]
relative to its baseline kernel, with the terminal history interpreted to include $Y$. Expanding these finitely many ratios gives an $O(s^2)$ remainder uniformly over policies, because $b$ and the outcome are bounded. Multiplying and integrating identifies the fixed-policy derivative as the inner product with
\begin{equation}\label{eq:dynamic-fixed-gradient}
\lambda_d^{\mathrm{fix}}
=V_1^d-\Psi_d(P)+\sum_{t=1}^TW_t^d(V_{t+1}^d-Q_t^d).
\end{equation}
For example, each transition contribution equals $\E[W_t^d(V_{t+1}^d-Q_t^d)b]$ by conditional centering; the initial-history contribution is $\E[(V_1^d-\Psi_d(P))b]$. This also verifies that the gradient is mean zero. Its absolute value is bounded by $2B_0(1+\sum_t\epsilon^{-t})$, independently of the policy.

These derivatives hold along every DQM path, not only bounded tilts. Indeed, let $P_s$ have score $h$ and approximate $h$ in $L_2(P)$ by bounded mean-zero $b_m$. The path $P_{s,b_m}=(1+sb_m)P$ stays in the model for small $s$, by the overlap slack. DQM implies that its Hellinger distance from $P_s$, divided by $\abs s$, converges to $\norm{h-b_m}/2$ when Hellinger distance has no factor $1/\sqrt2$. Since total variation is at most Hellinger distance, \eqref{eq:dynamic-tv} bounds the difference of their value increments, divided by $\abs s$, by $C'\norm{h-b_m}$ in the limit. The same bound holds for the difference of the proposed derivatives, uniformly over policies, by the uniform gradient bound. First taking $s\to0$ and then $m\to\infty$ proves the expansion for every fixed policy, uniformly over $d$. To verify that every score is attainable, use $dP_s/dP=1+s h_s$, where $h_s$ is $h$ truncated at $s^{-1/2}$ and then centered. As in \cref{app:policy-model}, $h_s\to h$ in $L_2(P)$ and $\norm{s h_s}_\infty\le2\sqrt s$. Every treatment propensity ratio is at least $(1-2\sqrt s)/(1+2\sqrt s)$, so the uniform overlap slack preserves all the constraints. This realizes every score in the closed space $L_2^0(P)$ and proves \assT.

\emph{Step 2: the recursive directional derivative.}
For a bounded score $b$, backward induction shows that $(V_{t,P_s}-V_t)/s$, $s\downarrow0$, is uniformly bounded and converges in $L_1(P_{H_t})$ to $u_{t,b}$ in \eqref{eq:dynamic-recursion}. The assertion starts with $V_{T+1}=Y$. At stage $t$, conditional differentiation of the bounded fixed function $V_{t+1}$ contributes $k_{t,b}$. Conditional expectation is an $L_1$ contraction, and positivity bounds each action-specific conditional-expectation operator by $\epsilon^{-1}$ in the corresponding marginal $L_1$ norm. Therefore the changing continuation value contributes $\E[u_{t+1,b}\mid H_t,A_t]$ in $L_1$; replacing its tilted transition kernel by the baseline kernel costs $O(s)$ because the transition density ratio is uniformly $1+O(s)$ and the difference quotients are bounded. Finally, a finite maximum is Lipschitz in its arguments. If their difference quotients converge in $L_1$ to $z_a$, one may first replace those quotients by $z_a$ with $o(1)$ error; for these fixed $z_a$, pointwise directional differentiation and domination by $\max_a\abs{z_a}$ give the limit $\max_{a\in\mathcal A_t^*}z_a$ in $L_1$. This proves the backward recursion. Differentiating the initial-history expectation yields the last line of \eqref{eq:dynamic-recursion}.

The same recursion is well-defined for every $h\in L_2^0(P)$. Conditional Jensen's inequality, treatment positivity, and the finiteness of $T$ bound the integrated changes in its successive maxima by a constant times $\norm{h-b_m}$ when $h$ is replaced by $b_m$. More explicitly, use $\abs{\max_a x_a-\max_a y_a}\le\max_a\abs{x_a-y_a}$ and, for any nonnegative array $z_a(H_t)$,
\[
\E\max_a z_a(H_t)\le\epsilon^{-1}\E z_{A_t}(H_t).
\]
Together with $\abs{V_{t+1}-Q_t}\le2B_0$, repeated conditional expectations give $\abs{\gamma(h)-\gamma(b_m)}\le C''\norm{h-b_m}$. The optimal-value part of \eqref{eq:dynamic-tv} now extends the bounded-tilt calculation to every DQM path exactly as in Step 1. No margin condition on the Bellman gaps is required.

\emph{Step 3: complete optimal policies and the core.}
For any policy $d$, successive conditioning gives the value-gap identity
\begin{equation}\label{eq:dynamic-gap}
\Psi(P)-\Psi_d(P)
=\sum_{t=1}^T\E_{P^d}[V_t(H_t)-Q_t(H_t,d_t(H_t))].
\end{equation}
Each summand is nonnegative. Thus $d$ is globally optimal precisely when it selects an active treatment at every stage, almost surely under its own reached-history law. Its choices at histories not reached under $P^d$ are immaterial. Moreover its continuation value equals the optimal continuation value at reached histories: otherwise its conditional future loss would be positive on a set with positive reached-history probability, contradicting \eqref{eq:dynamic-gap}. Replacing $V_t^d,Q_t^d$ in \eqref{eq:dynamic-fixed-gradient} with $V_t,Q_t$ on the prefixes for which $W_t^d>0$ is therefore valid almost surely, yielding \eqref{eq:dynamic-gradient}.

Every $d\in\mathcal D^*(P)$ satisfies $\Psi(P_s)\ge\Psi_d(P_s)$ with equality at $s=0$, so $\gamma(h)\ge\ip{\lambda_d}h$. Conversely, perform the recursion \eqref{eq:dynamic-recursion} while choosing, at each history, an active treatment attaining its displayed maximum. Finite action sets permit a measurable selector, for example the smallest maximizing index. This produces a policy $d_h$ that selects baseline-optimal treatments everywhere and is consequently in $\mathcal D^*(P)$. Applying the linear fixed-policy derivative recursion to this chosen policy gives $\ip{\lambda_{d_h}}h=\gamma(h)$. Hence the supremum over active policies is attained for each score. In particular $\gamma$ is continuous and sublinear, and its subdifferential is the closed convex hull of their uniformly bounded gradients. The geometry and convolution statements follow from \cref{thm:kink,thm:main}; the cone statement follows by the representer on $\clspan C_d$, with the finite-hull vertex assertion supplied by \cref{cor:onesided}.

This argument also explains reachability. Two complete policies that differ only on histories they do not reach induce the same law and the same gradient in \eqref{eq:dynamic-gradient}. Requiring derivative equality at every locally tied history, independently of reachability under an optimal policy, would therefore generally be stronger than the actual core condition.

\emph{Step 4: feasible attainment on a finite history tree.}
Fix a finite, prespecified history tree, with positive observational probability for every history and treatment cell in that tree. The outcome may have any bounded conditional distribution in its terminal cells; positive conditional variances are not needed. The set $\mathcal D$ of deterministic policies is finite. On samples with a nonempty observation in every required cell, use the empirical initial probabilities, transition probabilities, treatment probabilities, and terminal means to compute $\widehat\Psi_d$, and plug these estimates into the fixed-policy formula \eqref{eq:dynamic-fixed-gradient} to obtain $\hat\lambda_d$. Center the latter empirically when forming
\[
\widehat\Sigma_{de}
=\mathbb P_n[(\hat\lambda_d-\mathbb P_n\hat\lambda_d)
                 (\hat\lambda_e-\mathbb P_n\hat\lambda_e)].
\]
This is positive semidefinite. Define the retained set by
\[
\widehat{\mathcal D}
=\{d:\max_{e\in\mathcal D}\widehat\Psi_e-\widehat\Psi_d\le\tau_n\},
\]
and use \eqref{eq:dynamic-pooling} on that set. If a required empirical cell is empty, define the value estimate to be the overall sample mean and its variance estimate to be the overall empirical outcome variance. This exceptional event has probability tending to zero, and the construction is defined for every sample.

The finitely many probabilities, terminal means, and terminal second moments are smooth ratios of empirical averages of bounded variables, with denominators bounded away from zero at $P$. The multivariate delta method, or direct expansion of the backward recursions, gives simultaneously for all $d\in\mathcal D$
\begin{equation}\label{eq:dynamic-finite-al}
\sqrt n(\widehat\Psi_d-\Psi_d(P))
=n^{-1/2}\sum_{i=1}^n\lambda_d^{\mathrm{fix}}(O_i)+o_{P^n}(1),
\qquad \widehat\Sigma-\Sigma=O_P(n^{-1/2}),
\end{equation}
where $\Sigma$ is the Gram matrix of the corresponding fixed-policy gradients. The first identity uses exactly the derivative already derived in Step 1: varying the empirical conditional probabilities and means differentiates the same finite g-formula. The covariance claim follows because each empirical covariance is a smooth function of the same cell frequencies, first moments, and second moments. The tolerance conditions imply
$P^n(\widehat{\mathcal D}=\mathcal D^*(P))\to1$: active values have $O_P(n^{-1/2})$ differences and every inactive policy has a fixed positive value gap. Work on this event from now on and restrict $\Sigma$ to active policies.

Let $w^*$ be the smallest Euclidean-norm minimizer of $w^\top\Sigma w$ subject to $\one^\top w=1$, and let $w_\rho$ minimize $w^\top\Sigma w+\rho\norm w^2$ on the same hyperplane. The unregularized minimum is attained because it is the squared distance to a finite-dimensional affine hull. Comparing the two objectives gives $\norm{w_\rho}\le\norm{w^*}$ and $w_\rho^\top\Sigma w_\rho\to\Vk$. Every limit point therefore minimizes the unregularized objective and has norm no greater than $\norm{w^*}$, so uniqueness of the smallest-norm minimizer yields $w_\rho\to w^*$. Let $e=\hat w-w_{\rho_n}$. Subtracting the two first-order conditions and using $\one^\top e=0$ gives
\[
e^\top(\widehat\Sigma+\rho_n I)e
=-e^\top(\widehat\Sigma-\Sigma)w_{\rho_n},\qquad
\norm e\le\rho_n^{-1}\norm{\widehat\Sigma-\Sigma}_{\mathrm{op}}\norm{w_{\rho_n}}.
\]
Thus $\hat w\to_P w^*$ under the stated condition $\sqrt n\rho_n\to\infty$. This proof does not require constant empirical rank or an invertible Gram matrix.

All active policy values equal $\Psi(P)$ and the weights sum exactly to one. Consequently their estimated weights introduce no first-order term from the common value, and \eqref{eq:dynamic-finite-al} implies
\[
\sqrt n(\widehat\Psi_{\mathrm{dyn}}-\Psi(P))
=\sum_{d\in\mathcal D^*(P)}w_d^*
 n^{-1/2}\sum_{i=1}^n\lambda_d(O_i)+o_{P^n}(1).
\]
The weighted gradient equals $\psik$. Covariance consistency also gives $\hat w^\top\widehat\Sigma\hat w\to_P\Vk$. Finally, contiguity transfers the expansion and the event recovering the baseline active set to every DQM local alternative. The joint central limit theorem and Le Cam's third lemma give the local mean $\ip\psik h-\gamma(h)$. It vanishes on $\Hk$, proving regularity and attainment there. Outside the core the shift need not have a fixed sign, because the efficient affine weights need not be nonnegative.
\end{proof}

\bibliographystyle{plainnat}
\bibliography{aprime_references}

\end{document}